\documentclass [leqno, 10pt]{amsart}

\AtBeginDocument{%
   \def\MR#1{}
}

\AtBeginDocument{%
   \def\DOI#1{}
}
\AtBeginDocument{%
   \def\URL#1{}
   \def\ISBN#1{}
}

\usepackage[T1]{fontenc}
\usepackage[utf8]{inputenc}

\usepackage{geometry}

\usepackage{microtype}

\usepackage[headings, myheadings]{fullpage}

\usepackage{lmodern}
\usepackage{microtype}

\usepackage[all]{xy}
\usepackage{amsmath, amsfonts, amsthm, amsbsy, amscd, amsxtra, stmaryrd}
\usepackage{mathabx} 
\usepackage{amssymb}
\usepackage[bbgreekl]{mathbbol}%
\usepackage{pgf,interval}
\usepackage[mathscr]{euscript}
\usepackage{upgreek}
\usepackage{mathtools}
\usepackage{tikz}
\usepackage[inline]{enumitem}%
\usepackage{url}
\usepackage{hyperref}
\hypersetup{
	colorlinks=true, linktocpage=true, hidelinks
}

\usepackage{etoolbox}
\apptocmd{\sloppy}{\hbadness 10000\relax}{}{}

\usepackage{thmtools}%
\newtheorem{theorem}{Theorem}[section]
\newtheorem{corollary}[theorem]{Corollary}%
\newtheorem{lemma}[theorem]{Lemma}%
\newtheorem{definition}[theorem]{Definition}%

\theoremstyle{definition}

\declaretheoremstyle[
  headfont=\normalfont\bfseries,
  sharenumber = theorem,
  bodyfont=\normalfont,
    qed = {\hbox{$\triangleleft$}}
]{examplestyle2}

\declaretheorem[
  style=examplestyle2,
  title=Example,
  refname={example,examples},
  Refname={Example,Examples}
]{example}

\declaretheorem[
  style=examplestyle2,
  title=Remark,
  refname={remark, remarks},
  Refname={Remark, Remarks}
]{remark}

\numberwithin{equation}{section}

\makeatletter
\def\l@subsection{\@tocline{2}{0pt}{1pc}{4.6em}{}}
\renewcommand{\tocsubsection}[3]{%
  \indentlabel{\@ifnotempty{#2}{\hspace*{2.3em}\makebox[2.3em][l]{%
    \ignorespaces#1 #2.\hfill}}}#3}
\makeatother

\newcommand{\grass}{\operatorname{Gr}}

\newcommand{\sect}{\mathscr{K}}

\newcommand{\idem}{\mathbb{Idem}}

\newcommand{\so}{\mathfrak{so}}

\newcommand{\tmlt}{\circledcirc}
\usepackage{stackengine} 
\usepackage{graphicx}

\newcommand{\alg}{\mathbb{A}}

\newcommand{\h}{\mathfrak{h}}
\newcommand{\kl}{\mathfrak{k}}

\newcommand{\om}{\omega}

\newcommand{\mlt}{\circ}

\newcommand{\hess}{\operatorname{Hess}}

\renewcommand{\part}{\vdash}

\newcommand{\dum}{\,\cdot\,\,}

\newcommand{\id}{\operatorname{Id}}

\renewcommand{\sp}{\mathfrak{sp}}

\newcommand{\ep}{\epsilon}

\newcommand{\eno}{\operatorname{End}}

\newcommand{\lin}{\mathsf{L}}

\newcommand{\si}{\sigma}

\newcommand{\integer}{\mathbb{Z}}

\newcommand{\re}{\operatorname{Re}}

\newcommand{\al}{\alpha}

\newcommand{\spn}{\operatorname{Span}}

\newcommand{\gl}{\mathfrak{gl}}
\newcommand{\su}{\mathfrak{su}}
\newcommand{\sll}{\mathfrak{sl}}

\newcommand{\proj}{\mathbb{P}}

\newcommand{\g}{\mathfrak{g}}

\newcommand{\p}{\mathfrak{p}}

\newcommand{\stiefel}{\mathscr{V}}

\newcommand{\ad}{\operatorname{ad}}
\newcommand{\Ad}{\operatorname{Ad}}

\newcommand{\rea}{\mathbb R}
\newcommand{\com}{\mathbb C}

\newcommand{\tr}{\operatorname{tr}}

\newcommand{\gz}{\g_{0}}

\newcommand{\crtnt}{\mathrm{Crit}^*}
\newcommand{\hv}{h^\vee}
\newcommand{\HH}{\mathcal H}

\newcommand{\norm}[1]{\lVert #1\rVert}
\newcommand{\nnorm}[1]{\lvert #1\rvert}%
\newcommand{\comass}[1]{\lVert #1\rVert_{\ast}}%
\newcommand{\fnorm}[1]{\lVert #1\rVert_{F}}
\newcommand{\subs}{\mathfrak{s}}
\newcommand{\sgn}{\operatorname{sign}}
\newcommand{\qad}{Q}

\newcommand{\Usp}{U}
\newcommand{\Uperp}{\Usp^{\perp}}

\begin{document}
\title[Sharp commutator bounds on simple Lie algebras]{Sharp commutator bounds on complex simple Lie algebras:\texorpdfstring{\\}{ } the Böttcher--Wenzel inequality and\texorpdfstring{\\}{ } the comass of the Cartan $3$-form}
\hypersetup{pdftitle={Sharp commutator bounds on complex simple Lie algebras: the Böttcher-Wenzel inequality and the comass of the Cartan 3-form}, pdfauthor={Daniel J. F. Fox}}

\author{Daniel J.~F. Fox}
\address{Departamento de Matemática Aplicada\\ Escuela Técnica Superior de Arquitectura\\ Universidad Politécnica de Madrid\\Av. Juan de Herrera 4 \\ 28040 Madrid España}
\email{daniel.fox@upm.es} 

\date{\today}

\subjclass[2020]{Primary 17B20; Secondary 15A45, 15A69, 53C35, 53C38}
\keywords{B\"ottcher--Wenzel inequality, commutator, dual Coxeter number, Dynkin index, comass, Cartan $3$-form, Nahm algebra, twisted idempotent}

\begin{abstract}
For a complex simple Lie algebra $\g$ with compact conjugation $\tau$, Killing form $B$, and Hermitian form $H(x, y) = -B(x, \tau y)$, it is shown that $H([x,y],[x,y]) \leq (h^{\vee})^{-1}H(x,x)H(y,y)$, where $h^{\vee}$ is the dual Coxeter number, with equality exactly for $H$-orthogonal pairs in $\tau$-stable long-root subalgebras isomorphic to $\sll(2, \com)$. Equivalently, the comass of the Cartan $3$-form with respect to $H$ equals its comass on the compact real form: the spectral norm of the Cartan $3$-form does not increase under complexification. For the Frobenius norm on the image of a representation of Dynkin index $\ell$ the optimal constant is $2/\ell$; this contains the Böttcher--Wenzel and Bloch--Iserles inequalities, extends the latter to complex skew-symmetric matrices, and improves them substantially for the exceptional Lie algebras of types $F_{4}$, $E_{6}$, $E_{7}$, and $E_{8}$. The proof uses the Nahm algebra of $\g$ with the conjugation induced by $\tau$: critical points of the relevant function correspond to $\tau$-twisted idempotents, and its Hessian is governed by $\tau$-twisted multiplication operators, whose anticommutation with multiplication by $i$ turns the one-sided second order condition at a maximum into a two-sided spectral bound. Subalgebras isomorphic to $\sll(2, \com)$ of Dynkin index $j$ give critical points with critical value $1/(jh^{\vee})$, whose Morse indices are computed, and the set of maximum points is a nondegenerate critical manifold. Geometrically, the maximal complex sectional curvature of a compact simple Lie group with bi-invariant metric equals its maximal sectional curvature.
\end{abstract}

\maketitle

\setcounter{tocdepth}{1}%

\tableofcontents

\section{Introduction}\label{sec:intro}
By the Cauchy--Schwarz inequality, for complex $n \times n$ matrices $X$ and $Y$ and the Frobenius norm $\fnorm{X}^{2} = \tr X X^{\dagger}$, the squared norm of the commutator $\fnorm{[X, Y]}^{2}$ is bounded from above by $c\fnorm{X}^2 \fnorm{Y}^2$ for some constant $c > 0$. The classical Böttcher--Wenzel inequality \cite{Audenaert, Bottcher-Wenzel-howbig, Bottcher-Wenzel, Lu-normalscalarcurvature} is the statement that the optimal constant is $2$,
\begin{align}\label{eq:bwclassical}
\fnorm{[X, Y]}^{2} \leq 2 \fnorm{X}^2 \fnorm{Y}^2,
\end{align}
together with the characterization of the matrix pairs realizing equality. Earlier work of Bloch--Iserles \cite{Bloch-Iserles} showed that for real skew-symmetric matrices the constant $2$ in \eqref{eq:bwclassical} can be improved to $1$, and the special case of \eqref{eq:bwclassical} for real symmetric matrices was obtained by Chern--do Carmo--Kobayashi \cite{Chern-Docarmo-Kobayashi}, as part of their study of second fundamental forms of minimal submanifolds of spheres. 

The main result described here situates the Böttcher--Wenzel and Bloch--Iserles inequalities in a Lie theoretic framework in which the inequality \eqref{eq:bwclassical} corresponds to the $A_{n}$ series, and the Bloch--Iserles results for skew-symmetric matrices to the $B_{n}$ and $D_{n}$ series. Moreover, this interpretation yields a strictly stronger inequality for the exceptional Lie algebras of types $F_{4}$, $E_{6}$, $E_{7}$, and $E_{8}$, while for $G_{2}$ it reproduces the Bloch--Iserles constant. The optimal constant in \eqref{eq:bwclassical} depends on the choice of norm, and with an appropriate choice of norm the optimal constant for a complex simple Lie algebra is the reciprocal of its dual Coxeter number, and is related to the optimal constant for a matrix Lie algebra by the Dynkin index of the defining representation. 

There is a large literature generalizing \eqref{eq:bwclassical} and extending it in various ways, partly driven by the geometric motivation of \cite{Chern-Docarmo-Kobayashi}; see for example \cite{Cheng-Vong-Wenzel, Ge, Ge-Li-Zhou, Lu-remarks, Lu-Wenzel, Wu-Liu}. Much of it concerns what are called DDVV-type inequalities for $m$ matrices in a prescribed class; this is surveyed in \cite{Ge-Li-Tang-Zhou-survey}. Ge, Li, and Zhou ask \cite[Section 1, question (2)]{Ge-Li-Zhou}, \cite[Question 4.2]{Ge-Li-Tang-Zhou-survey} for DDVV-type inequalities on Lie algebras; Theorem \ref{thm:A} below answers this for two elements of an arbitrary complex simple Lie algebra with the Hermitian norm \eqref{eq:Hdef}, and Theorem \ref{thm:B} recovers the optimal two matrix constants for the classical matrix classes. The relation with the DDVV literature is discussed further in Section \ref{sec:classical}.

Since multiples of the identity are central, it suffices to consider \eqref{eq:bwclassical} for trace-free matrices, that is on $\sll(n, \com)$, and since \eqref{eq:bwclassical} is invariant under conjugation by the unitary group, it should be the special case for the pair $(\sll(n, \com), \su(n))$ of a statement about a complex simple Lie algebra and a compact real form.

Let $\g$ be a complex simple Lie algebra with Killing form $B(x, y) = \tr\ad(x)\ad(y)$, let $\gz \subset \g$ be a compact real form, so that $\g = \gz \oplus i\gz$ and $B$ is negative definite on $\gz$, and let $\tau$ be the associated \emph{compact conjugation}: the conjugate-linear Lie algebra automorphism of $\g$ that fixes $\gz$ pointwise and acts as $-1$ on $i\gz$. Then
\begin{align}\label{eq:Hdef}
H(x,y) := -B(x,\tau y), \qquad x,y\in\g,
\end{align}
is a positive definite Hermitian form, linear in $x$ and conjugate-linear in $y$; write $\norm{x}^{2} = H(x, x)$. For example, for $\g = \sll(n, \com)$ and $\gz = \su(n)$, $\tau X = -X^{\dagger}$ and $H(X, Y) = 2n\tr XY^{\dagger}$. By a theorem of Dynkin \cite[Theorem~2.2]{Dynkin-semisimplesubalgebras}, if one complex simple Lie algebra is embedded in another, and each is equipped with invariant form equal to the multiple of its Killing form for which long roots have squared length $2$, the normalized invariant form of the larger Lie algebra restricts to a positive integer multiple of the normalized form of the smaller one; this integer is the \emph{Dynkin index} of the embedding (Definition \ref{def:dynkin}). The subalgebras isomorphic to $\sll(2, \com)$ and having Dynkin index $1$ are exactly the \emph{long-root subalgebras}, the conjugates under the adjoint group of the subalgebra $\g_{\theta} \oplus \com\theta^{\vee} \oplus \g_{-\theta}$ attached to the highest root $\theta$ (Lemma \ref{lem:Bh}). The dual Coxeter number $\hv$ of $\g$ is recalled in Section \ref{sec:prelim}, and its values are listed in Table \ref{tab:improvement}.

\begin{theorem}\label{thm:A}
Let $\g$ be a complex simple Lie algebra with dual Coxeter number $\hv$. For compact real form $\gz$ with associated compact conjugation $\tau$ and Hermitian form $H$,%
\begin{align}\label{eq:main}
\norm{[x,y]}^{2} \leq \frac{1}{h^\vee}\norm{x}^{2}\norm{y}^{2} \qquad \text{for all}\,\, x, y \in \g,
\end{align}
and the constant is optimal. If $x, y \neq 0$, then equality holds in \eqref{eq:main} if and only if $H(x, y) = 0$ and $x$ and $y$ are contained in a $\tau$-stable subalgebra $\subs \subset \g$ isomorphic to $\sll(2, \com)$ and having Dynkin index $1$.
\end{theorem}

\begin{example}\label{ex:eh}
Let $\g = \sll(2, \com)$ and $\gz = \su(2)$, so that $\tau X = -X^{\dagger}$, $H(X, Y) = 4\tr XY^{\dagger}$, and $\hv = 2$. For $e = E_{12}$ and $h = E_{11} - E_{22}$, $H(e, h) = 0$, $\norm{e}^{2} = 4$, $\norm{h}^{2} = 8$, and $[e, h] = -2e$, so
\begin{align}\label{eq:ehnorms}
\frac{\norm{[e, h]}^{2}}{\norm{e}^{2}\norm{h}^{2}} = \frac{16}{32} = \frac{1}{\hv},
\end{align}
and, by \eqref{eq:ehnorms}, equality holds in \eqref{eq:main}. However, $e$ and $h$ generate the two-dimensional subalgebra $\com e \oplus \com h$, which is not $\tau$-stable. Thus in Theorem \ref{thm:A} ``contained in'' cannot be replaced by ``generate''. The same happens in every complex simple Lie algebra $\g$: for $0 \neq e \in \g_{\theta}$, where $\theta$ is the highest root, $H(e, \theta^{\vee}) = 0$ and $e$ and $\theta^{\vee}$ lie in the $\tau$-stable subalgebra of Dynkin index $1$ attached to $\theta$, so equality holds in \eqref{eq:main} by Lemmas \ref{lem:sl2str}\ref{sl2b} and \ref{lem:lower}\ref{lower2}, while $e$ and $\theta^{\vee}$ generate the subalgebra $\g_{\theta} \oplus \com\theta^{\vee}$, which is not $\tau$-stable.
\end{example}

When $x, y \in \gz$ the equality case can be sharpened: then $x$ and $y$ generate a long-root subalgebra of $\gz$ isomorphic to $\su(2)$ (Corollary \ref{cor:compactequality}). 

Theorem \ref{thm:A} has an equivalent formulation as Theorem \ref{thm:comass}, given below, affirming the equality of the complex and the real comass of the Cartan $3$-form of $\g$.

The inequality \eqref{eq:main} can be reinterpreted as a statement about representations. For a nonzero representation $\rho:\g \to \gl(V)$, the Dynkin index $\ell_{\rho}$ of $\rho$ is the Dynkin index of $\rho$ regarded as an embedding of $\g$ in $\sll(V)$; equivalently, the trace form $\tr \rho(x)\rho(y)$ is $\ell_{\rho}$ times the normalized invariant form of $\g$ (Lemma \ref{lem:dynkinformula}). The defining representations of $\sll(n, \com)$ and $\sp(2n, \com)$ have Dynkin index $1$, the vector representation of $\so(n, \com)$, $n \geq 5$, has Dynkin index $2$, and the adjoint representation has Dynkin index $2\hv$.

\begin{theorem}\label{thm:B}
Let $\rho:\g \to \gl(V)$ be a nonzero representation of the complex simple Lie algebra $\g$ with Dynkin index $\ell_{\rho}$, let $V$ carry a Hermitian inner product invariant under the compact real form $\gz$, and let $\fnorm{X}^{2} = \tr XX^{\dagger}$ be the corresponding Frobenius norm on $\gl(V)$. Then
\begin{align}\label{eq:repbound}
\fnorm{[X, Y]}^{2} \leq \frac{2}{\ell_{\rho}}\fnorm{X}^{2}\fnorm{Y}^{2} \qquad \text{for all}\,\, X, Y \in \rho(\g),
\end{align}
the constant is optimal, and the equality case is the image under $\rho$ of the equality case of Theorem \ref{thm:A}.
\end{theorem}

The main interest of Theorems \ref{thm:A} and \ref{thm:B} is that they give a uniform Lie theoretic interpretation of inequalities such as \eqref{eq:bwclassical}; they also yield strictly stronger inequalities in specific cases.
For the defining representation of $\sll(n, \com)$, \eqref{eq:repbound} is \eqref{eq:bwclassical}, and for the vector representation of $\so(n, \com)$ it extends the Bloch--Iserles inequality from real to complex skew-symmetric matrices (not covered by the methods of \cite{Ge-Li-Zhou}; see Section \ref{sec:classical}). Because $\rho(\g)$ is a subalgebra of $\gl(V)$, the classical inequality \eqref{eq:bwclassical} gives the constant $2$ in \eqref{eq:repbound} for every $\rho$; Theorem \ref{thm:B} improves this by the factor $\ell_{\rho}$, which exceeds $1$ except for the defining representations of $\sll(n, \com)$ and $\sp(2n, \com)$ (up to duals and trivial summands). For the exceptional Lie algebras of types $F_{4}$, $E_{6}$, $E_{7}$, and $E_{8}$ the improvement is substantial: for their smallest representations the optimal constants are $\tfrac{1}{3}$, $\tfrac{1}{3}$, $\tfrac{1}{6}$, and $\tfrac{1}{30}$, while for the $7$-dimensional representation of $G_{2}$ it is the Bloch--Iserles constant $1$. These cases are listed in Table \ref{tab:improvement} and discussed in Section \ref{sec:reductive}, where invariant Hermitian norms on reductive Lie algebras are also treated.

The problem has a well understood compact counterpart. By a theorem of Vinberg \cite{Vinberg-invariantnorms}, if $x$ lies in a compact Lie algebra, then the operator norm of $\ad(x)$ with respect to any invariant norm equals the spectral radius of $\ad(x)$, and on $\gz$ the optimal constant in \eqref{eq:main} reduces to a computation with roots, which shows that it equals the reciprocal of the dual Coxeter number. In geometric terms, the maximal sectional curvature of a compact simple Lie group (here and below, a compact connected Lie group with simple Lie algebra $\gz$; its sectional curvature depends only on $\gz$ and the metric) with bi-invariant metric $-B$ is $1/(4\hv)$ \cite[p.~334, Theorem~11.1]{Helgason-dgls}, \cite[Theorem~4]{Meshcheryakov}, and by \cite[Theorems~2 and~4]{Meshcheryakov} the critical points of this sectional curvature with nonzero critical value are the planes contained in subalgebras of $\gz$ isomorphic to $\su(2)$, the critical value being $1/(4j\hv)$ when the complexification of the subalgebra has Dynkin index $j$. On $\g$, however, $\ad(x)$ is in general not normal (for instance when $x$ is nilpotent), and its operator norm is not determined by its spectrum. The content of Theorem \ref{thm:A} is that the optimal constant on $\g$ is nonetheless the same as on $\gz$. 

On $\gz$ this computation is essentially contained in Tasaki's proof \cite{Tasaki-calibration} that a normalized Cartan $3$-form is a calibration, and for $x \in \gz \cup i\gz$ it is the computation of the constants in Wang's quantitative form of the Kazhdan--Margulis theorem \cite{Wang-nilpotent, Adeboye-Wang-Wei} (Section \ref{sec:reform}). This hints at the role of the Cartan $3$-form in the proof of Theorem \ref{thm:A}, although here it was studied for reasons related to the Nahm algebra construction described below.

The function $F(x, y) = \norm{[x, y]}^{2}/(\norm{x}^{2}\norm{y}^{2})$, whose supremum is the optimal constant in \eqref{eq:main}, is closely related to the \emph{sectional nonassociativity} defined in \cite{Fox-sectna}
\begin{align}\label{eq:sectintro}
\sect(x, y) = \frac{\norm{[x, y]}^{2}}{\norm{x}^{2}\norm{y}^{2} - |H(x, y)|^{2}},
\end{align}
which is defined on the Grassmannian of complex $2$-planes in $\g$ and has the same supremum. On $\gz$, $\tfrac{1}{4}\sect$ is the sectional curvature of the bi-invariant metric $-B$, and on complex $2$-planes in $\g$ it is the complex sectional curvature (Section \ref{sec:curvature}). Although apparently trivial, the key observation here is to relate the numerator of $\sect(x, y)$ with the Cartan $3$-form $\om(x, y, z) = B([x, y], z)$. More precisely, the idea is to compare $\norm{[x, y]}^2$ with the twisted Cartan $3$-form 
\begin{align}\label{eq:twistedcartan}
H([x, y], z) = -B([x, y], \tau z) = -\om(x, y, \tau z).
\end{align}
By \eqref{eq:twistedcartan} and the Cauchy--Schwarz inequality, the supremum of \eqref{eq:sectintro} is the squared comass $\comass{\om}^{2}$ of the Cartan $3$-form $\om$, which is the maximum of the function $W(x, y, z) = |\om(x, y, z)|^{2}/(\norm{x}^{2}\norm{y}^{2}\norm{z}^{2})$ on $\g^{3}$ (Section \ref{sec:reform}). 

Theorem \ref{thm:A} is therefore equivalent to Theorem \ref{thm:comass}.
\begin{theorem}\label{thm:comass}
Let $\g$ be a complex simple Lie algebra with dual Coxeter number $\hv$. For compact real form $\gz$ with associated compact conjugation $\tau$ and Hermitian form $H$, the comass of the Cartan $3$-form $\om$ with respect to the Hermitian form $H$ on $\g$ equals its comass with respect to the Euclidean form $-B$ on $\gz$,
\begin{align}\label{eq:comassintro}
\begin{aligned}
\comass{\om} &= \sup\big\{|\om(x, y, z)|: x, y, z \in \g,\ \norm{x} = \norm{y} = \norm{z} = 1\big\}\\
&= \sup\big\{\om(x, y, z): x, y, z \in \gz,\ \norm{x} = \norm{y} = \norm{z} = 1\big\} = (\hv)^{-1/2}.
\end{aligned}
\end{align}
\end{theorem}
In the language of tensors, \eqref{eq:comassintro} says that the spectral norm of the real $3$-tensor $\om|_{\gz}$ on $(\gz, -B)$ does not change under complexification, which for real tensors of order at least $3$ fails in general (Remark \ref{rem:comass}).

The proof of Theorem \ref{thm:A} goes through characterizing the maximum points of the function $W$. The most complicated part of the argument is that which makes use of the corresponding condition on the Hessian of $W$. These computations are made more understandable and more elegant by taking advantage of a connection with commutative nonassociative algebras, specifically the \emph{Nahm algebra} of $\g$ defined by Kinyon--Sagle \cite{Kinyon-Sagle} (Section \ref{sec:nahm}).

The Killing form $B$, Hermitian metric $H$, and conjugation $\tau$ extend componentwise to a symmetric bilinear form, a Hermitian metric, and a conjugation also denoted by $B$, $H$, and $\tau$ on the threefold product $\g^{3} = \g \times \g \times \g$. The Cartan $3$-form can be viewed as the cubic form on $\g^{3}$
\begin{align}\label{eq:cubic}
C(a) := \om(a_{1}, a_{2}, a_{3}) = B(a_{1}, [a_{2}, a_{3}]), \qquad a = (a_{1}, a_{2}, a_{3}) \in \g^{3},
\end{align}
in terms of which $W(a) = |C(a)|^{2}/(\norm{a_{1}}^{2}\norm{a_{2}}^{2}\norm{a_{3}}^{2})$. The polarization of $C$ is a completely symmetric trilinear form that, together with $B$, determines a commutative multiplication $\tmlt$ on $\g^{3}$. This is a special case of the tripling construction described in \cite{Fox-cubicpoly, Fox-Tkachev}. Up to a constant normalization the algebra $(\g^{3}, \tmlt)$ is the Nahm algebra defined in \cite{Kinyon-Sagle} and the product $\tmlt$ is \begin{align}\label{nahmproduct}
(a_{1}, a_{2}, a_{3})\tmlt (b_{1}, b_{2}, b_{3}) = ([a_{2}, b_{3}] + [b_{2}, a_{3}],[a_{3}, b_{1}] + [b_{3}, a_{1}] ,[a_{1}, b_{2}] + [b_{1}, a_{2}] ).
\end{align}
That $\tmlt$ comes from polarizing the Cartan form is a way of restating the identity
\begin{align}\label{eq:cubicpolar}
B(a\tmlt a, a) = 6C(a), \qquad a \in \g^{3},
\end{align}
which is the case $a = b = c$ of Lemma \ref{lem:nahmbasic}\ref{nahm1}; by \eqref{nahmproduct}, \eqref{eq:cubicpolar} amounts to the cyclic invariance of $\om$.
Here the conjugation $\tau$ is an additional element not previously considered in \cite{Fox-Tkachev, Kinyon-Sagle}, and the Nahm algebra of $\g$ together with $\tau$ provides a convenient and conceptually clear way of organizing the computations needed here: the critical points of $W$ correspond to the set $\idem_{\tau}(\g^{3}, \tmlt)$ of nonzero \emph{$\tau$-twisted idempotents}, the nonzero solutions of $e \tmlt e = \tau e$, and the Hessian of $W$ at a $\tau$-twisted idempotent $e \in \idem_{\tau}(\g^{3}, \tmlt)$ is governed by the multiplication operator of $e$ twisted by $\tau$ (Section \ref{sec:hermitianmetrized}).

\begin{theorem}\label{thm:C}
Let $\g$ be a complex simple Lie algebra with compact conjugation $\tau$ and Nahm algebra $(\g^{3}, \tmlt)$, and extend $H$ and $\tau$ to $\g^{3}$ componentwise.
\begin{enumerate}[label=(\arabic*), ref=(\arabic*)]
\item\label{thmC1} Up to the action of $(\com^{\ast})^{3}$ by rescaling of components, the critical points of $W$ at which $W$ is nonzero are the nonzero $\tau$-twisted idempotents $e \in \idem_{\tau}(\g^{3}, \tmlt)$, and the corresponding critical value is $W(e) = \tfrac{3}{4}H(e, e)^{-1}$.
\item\label{thmC2} Every nonzero $\tau$-twisted idempotent $e \in \idem_{\tau}(\g^{3}, \tmlt)$ satisfies $H(e, e) \geq \tfrac{3}{4}\hv$, with equality if and only if the components of $e$ span a $\tau$-stable subalgebra isomorphic to $\sll(2, \com)$ of Dynkin index $1$.
\end{enumerate}
\end{theorem}

In Theorem \ref{thm:C}, maximum points of $W$ correspond to $\tau$-twisted idempotents of minimal norm. The first order and second order conditions for a maximum at such an idempotent are what the proof of Theorem \ref{thm:A} uses. The second order condition says that the $\tau$-twisted multiplication operator $\tau\lin(e)$, where $\lin(e)b := e \tmlt b$, is bounded above by $\tfrac{1}{2}$ on the complement of the directions tangent to the $(\com^{\ast})^{3}$-orbit; because $\tau\lin(e)$ anticommutes with multiplication by $i$, it is then also bounded below by $-\tfrac{1}{2}$ (Lemma \ref{thm:hessW}). Combined with the Jacobi identity, this two-sided bound forces the components of $e$ to span a $\tau$-stable subalgebra isomorphic to $\sll(2, \com)$ (Section \ref{sec:sharp}), and the proof is completed by the elementary inequality $B(h, h) \geq 4\hv$ for the neutral element $h$ of an $\sll(2, \com)$-triple (Lemma \ref{lem:Bh}). The two-sided bound reflects the strict plurisubharmonicity of $-\log W$ on $\{C \neq 0\} \subset \proj(\g)^{3}$, that is, the positivity of the line bundle $\mathcal{O}(1, 1, 1)$ of which $C$ is a section (Section \ref{sec:hermitianmetrized}). The anticommutation with multiplication by $i$ also shows that at every critical point where $W \neq 0$ the Hessian of $W$ on $\proj(\g)^{3}$ has at least $3\dim\g - 3$ negative eigenvalues, so that no such point is a local minimum (Lemma \ref{lem:indexbound}).

The same framework applies to the analysis of critical points that are not maxima. Every subalgebra of $\g$ isomorphic to $\sll(2, \com)$ is conjugate to a $\tau$-stable one, and a $\tau$-stable subalgebra of Dynkin index $j$ yields critical points of $F$, $\sect$, and $W$ with critical value $1/(j\hv)$. The spectrum of the Hessian of $W$ at such a critical point is determined by the decomposition of $\g$ under the subalgebra, through a spin-orbit coupling computation, and the critical point is a local maximum if and only if $j = 1$ (Section \ref{sec:critical}). At such a critical point the number of negative eigenvalues of the Hessian on $\proj(\g)^{3}$ equals the lower bound $3\dim\g - 3$ exactly when the subalgebra is distinguished. For the Cartan $3$-form on a compact simple Lie group this extends the observation of Hitchin \cite{Hitchin-calibrations} that the principal three-dimensional subalgebra is a critical point (Corollary \ref{cor:hitchin}). The same computation shows that the set of maximum points of $W$ in $\proj(\g)^{3}$ is a nondegenerate critical manifold in the sense of Bott, fibered over the space of $\tau$-stable subalgebras of Dynkin index $1$ with fiber the manifold of complete flags in $\com^{3}$ (Corollary \ref{cor:morsebott}). Root vectors of three long roots with sum zero give critical points of another type, with critical value $1/(2\hv)$ (Lemma \ref{thm:roottriangle}); the plane spanned by two such root vectors is a critical point of $\sect$ not contained in any subalgebra isomorphic to $\sll(2, \com)$, so the description of the critical points recalled above for real planes does not extend to complex planes (Remark \ref{rem:meshcheryakov}).

Geometrically, the equality of complex and real comass says that the maximal complex sectional curvature of a compact simple Lie group with the bi-invariant metric $-B$ equals its maximal sectional curvature $1/(4\hv)$ (Corollary \ref{cor:groupcurvature}); for irreducible compact symmetric spaces, Theorem \ref{thm:A} gives an upper bound for the complex sectional curvature that is sharp only for some of them (Section \ref{sec:curvature}).

The paper is organized as follows. Section \ref{sec:prelim} collects preliminaries on the compact conjugation, root systems, and subalgebras isomorphic to $\sll(2, \com)$. Section \ref{sec:reform} introduces $F$, $\sect$, and $W$, relates them to the comass of the Cartan form, and records the results of Vinberg and Tasaki on the compact real form. Section \ref{sec:nahm} develops the Nahm algebra with conjugation, the correspondence between critical points of $W$ and $\tau$-twisted idempotents, the Hessian of $W$, and its interpretation through plurisubharmonicity, and uses them to prove Theorems \ref{thm:A} and \ref{thm:C} (Section \ref{sec:sharp}). Section \ref{sec:reductive} treats arbitrary invariant Hermitian norms, reductive Lie algebras, and representations, proves Theorem \ref{thm:B}, derives the classical matrix inequalities, and interprets Theorem \ref{thm:A} in terms of the complex sectional curvature of compact simple Lie groups and compact symmetric spaces. The final Section \ref{sec:critical} studies critical points that need not be maxima, their Hessians, the structure of the set of maximum points, and Hitchin's critical point of Tasaki's calibration.

\section{Preliminaries}\label{sec:prelim}
Throughout, $\g$ is a complex simple Lie algebra with Killing form $B$, $\gz$ is a compact real form of $\g$, $\tau$ is the corresponding compact conjugation (meaning $\tau$ is a conjugate-linear order $2$ Lie algebra automorphism whose fixed subspace is $\gz$), and $H(x, y) = -B(x, \tau y)$, with $\norm{x}^{2} = H(x, x)$.

The following standard facts are used repeatedly. For all $a, b, x, y, z \in \g$:
\begin{enumerate}
\item $B(\tau a, \tau b) = \overline{B(a, b)}$ and $B(a, b) = -H(a, \tau b)$.
\item $H$ is a positive definite Hermitian form, linear in the first and conjugate-linear in the second argument; and $H(\tau a, \tau b) = \overline{H(a, b)}$, so that $\norm{\tau x} = \norm{x}$.
\item The Cartan form $\om(x, y, z) := B([x, y], z)$ is complex trilinear and alternating, and
\begin{align}\label{eq:omegaformulas}
\om(x, y, z) = B(x, [y, z]) = B(y, [z, x]) = -H([x, y], \tau z).
\end{align}
\item $H([z, a], b) = -H(a, [\tau z, b])$; that is, the $H$-adjoint of $\ad(z)$ satisfies
\begin{align}\label{basic4}
\ad(z)^{\dagger} = -\ad(\tau z).
\end{align}

\item By \eqref{basic4} the operator
\begin{align}\label{eq:qaddefined}
\qad(x):=\ad(x)^\dagger\ad(x) = -\ad(\tau x)\ad(x)
\end{align}
is $H$-Hermitian and nonnegative, and satisfies
\begin{align}\label{basic5}
H(\qad(x)y, y) = \norm{[x, y]}^{2} = H(\qad(y)x, x).
\end{align}
\end{enumerate}

Let $\mathfrak{t}_{0} \subset \gz$ be a maximal abelian subalgebra and $\h = \mathfrak{t}_{0} \otimes \com$, a Cartan subalgebra of $\g$ with root system $\Phi$. Since $\tau t = t$ for $t \in \mathfrak{t}_{0}$, $\ad(t)$ is skew-Hermitian by \eqref{basic4}, so the roots take imaginary values on $\mathfrak{t}_{0}$; consequently, for $X \in \g_{\mu}$ and $t \in \mathfrak{t}_{0}$, $[t, \tau X] = \tau[t, X] = \overline{\mu(t)}\tau X = -\mu(t)\tau X$, that is,
\begin{align}\label{eq:tauroot}
\tau\g_{\mu} = \g_{-\mu}, \qquad \mu \in \Phi.
\end{align}
Let $(\dum, \dum)$ be the multiple of $B$ for which long roots have squared length $2$, let $\theta$ be the highest root for some choice of positive roots $\Phi^{+}$, and for $\mu \in \Phi$ let $\mu^{\vee} \in \h$ be the coroot, so that $(\mu^{\vee}, h) = 2\mu(h)/(\mu, \mu)$ for $h \in \h$. 

The dual Coxeter number $\hv$ is defined to be half the eigenvalue on the adjoint representation of the Casimir operator defined using $(\dum, \dum)$. Equivalently \cite[(14.30) and (14.32)]{Fuchs-Schweigert}, the trace form of the adjoint representation is $2\hv$ times the normalized form,
\begin{align}\label{eq:Bnorm}
B(u, v) = 2\hv\,(u, v), \qquad u, v \in \g.
\end{align}
With $2\delta = \sum_{\al \in \Phi^{+}}\al$, there holds \cite[Section~32.4]{Etingof-liegroups}
\begin{align}\label{eq:hvweights}
\hv = 1 + \frac{(\theta, 2\delta)}{(\theta, \theta)},
\end{align}
and \eqref{eq:hvweights} gives the values of $\hv$ listed in Table \ref{tab:improvement}.

Let $G_{0}$ be the adjoint group of $\gz$, a compact connected group acting on $\g$ by complex linear automorphisms preserving $\gz$, hence commuting with $\tau$ and preserving $B$ and $H$. Every element of $\gz$ is $\Ad(G_{0})$-conjugate to an element of $\mathfrak{t}_{0}$, and the Weyl group, which acts transitively on the long roots, is realized by the normalizer of $\mathfrak{t}_{0}$ in $G_{0}$. 

The following standard facts are used (see \cite{Humphreys}).
\begin{enumerate}[label=(R\arabic*), ref=(R\arabic*)]
\item\label{R1} $(\theta,\theta)=2$; $\mu(h)\le\theta(h)$ for every root $\mu$ and every dominant $h$ in the real span of the coroots; every root $\mu$ satisfies $(\mu,\mu)\le2$, with equality if and only if $\mu$ is long.
\item\label{R2} $\dim\g_\mu=1$ for every root $\mu$; $B(\g_\mu,\g_\nu)=0$ unless $\mu+\nu=0$; for $e\in\g_\mu$ and $f\in\g_{-\mu}$, $[e,f]=B(e,f)\,t_\mu$, where $t_{\mu} \in \h$ is defined by $B(t_\mu,h)=\mu(h)$ for $h\in\h$.
\item\label{R3} Every semisimple element of $\g$ is conjugate under the adjoint group, which preserves $B$, to an element of $\h$, and every element of the real span of the coroots is conjugate under the Weyl group to a dominant element.
\end{enumerate}
By \eqref{eq:Bnorm}, $t_{\mu} = (\mu, \mu)\mu^{\vee}/(4\hv)$, and in particular $t_{\theta} = \theta^{\vee}/(2\hv)$. Since the roots are imaginary on $\mathfrak{t}_{0}$, the element $t_{\al}$ of \ref{R2} satisfies $t_{\al} = i s_{\al}$ with $s_{\al} \in \mathfrak{t}_{0}$, and by \eqref{eq:Bnorm},
\begin{align}\label{eq:salpha}
\al(x) = iB(s_{\al}, x) \text{ for } x \in \mathfrak{t}_{0}, \qquad \norm{s_{\al}}^{2} = B(t_{\al}, t_{\al}) = \frac{(\al, \al)}{2\hv} \leq \frac{1}{\hv},
\end{align}
with equality if and only if $\al$ is long.

An $\sll(2,\com)$-\emph{triple} is a triple $(e,f,h)$ of elements of $\g$ with $e \neq 0$ and satisfying
\begin{align}\label{eq:sl2triple}
[h,e]&=2e, & [h,f]&=-2f, & [e,f]&=h.
\end{align}
For a root $\mu \in \Phi$, the \emph{root subalgebra} attached to $\mu$ is
\begin{align}\label{eq:rootsubalgebra}
\subs_{\mu} := \g_{\mu} \oplus \com\mu^{\vee} \oplus \g_{-\mu}.
\end{align}
By \ref{R2} it is a subalgebra isomorphic to $\sll(2, \com)$, and by \eqref{eq:tauroot} it is $\tau$-stable. A \emph{long-root subalgebra} of $\g$ is a subalgebra of the form $\Ad(g)\subs_{\mu}$ with $g$ in the adjoint group of $\g$ and $\mu$ a long root; since every Cartan subalgebra of $\g$ is conjugate to $\h$ under the adjoint group, these are exactly the root subalgebras attached to long roots with respect to some Cartan subalgebra of $\g$.

For a complex simple Lie algebra $\mathfrak{a}$, let $(\dum, \dum)_{\mathfrak{a}}$ denote its \emph{normalized form}, the invariant symmetric bilinear form for which the long roots have squared length $2$; thus $(\dum, \dum) = (\dum, \dum)_{\g}$. If $\varphi:\mathfrak{a} \to \mathfrak{b}$ is an injective homomorphism of complex simple Lie algebras, the pullback of $(\dum, \dum)_{\mathfrak{b}}$ by $\varphi$ is an invariant symmetric bilinear form on the simple Lie algebra $\mathfrak{a}$, so it is a multiple of $(\dum, \dum)_{\mathfrak{a}}$. Dynkin \cite[Theorem~2.2]{Dynkin-semisimplesubalgebras} proved that this multiple is a positive integer $j(\varphi)$, so that
\begin{align}\label{eq:dynkinpullback}
(\varphi(x), \varphi(y))_{\mathfrak{b}} = j(\varphi)\,(x, y)_{\mathfrak{a}}, \qquad x, y \in \mathfrak{a}.
\end{align}

\begin{definition}\label{def:dynkin}
The integer $j(\varphi)$ of \eqref{eq:dynkinpullback} is the \emph{Dynkin index} of $\varphi$. The Dynkin index $j(\subs)$ of a simple subalgebra $\subs \subset \g$ is the Dynkin index of the inclusion $\subs \hookrightarrow \g$. The Dynkin index $\ell_{\rho}$ of a nonzero representation $\rho:\g \to \gl(V)$ is the Dynkin index of $\rho$ regarded as a homomorphism $\g \to \sll(V)$.
\end{definition}

In Definition \ref{def:dynkin}, $\rho$ is injective because $\g$ is simple, so $\dim V \geq 2$ because $\g$ is not abelian, and $\rho(\g) \subset \sll(V)$ because $\g = [\g, \g]$. Integrality of the Dynkin index is not used in the proofs of Theorems \ref{thm:A}--\ref{thm:C}; what is used there is the following description of it.

\vfill
\begin{lemma}\label{lem:dynkinformula}
\leavevmode\noindent
\begin{enumerate}[label=(\arabic*), ref=(\arabic*)]
\item\label{dynkinformula1} The normalized form of $\sll(n, \com)$, $n \geq 2$, is $(X, Y)_{\sll(n, \com)} = \tr XY$. Consequently, for a nonzero representation $\rho:\g \to \gl(V)$,
\begin{align}\label{eq:repindex}
\tr \rho(x)\rho(y) = \ell_{\rho}\,(x, y), \qquad x, y \in \g.
\end{align}
\item\label{dynkinformula2} If $\subs \subset \g$ is a subalgebra isomorphic to $\sll(2, \com)$ and $(e, f, h)$ is an $\sll(2, \com)$-triple spanning $\subs$, then
\begin{align}\label{eq:dynkinindex}
j(\subs) = \frac{(h, h)}{2} = \frac{B(h, h)}{4\hv}.
\end{align}
\end{enumerate}
\end{lemma}

\begin{proof}
For the diagonal Cartan subalgebra of $\sll(n, \com)$ and the root $\al = \ep_{a} - \ep_{b}$, $a \neq b$, where $\ep_{a}(H) = H_{aa}$, the element $t = E_{aa} - E_{bb}$ satisfies $\tr(tH) = \al(H)$ for diagonal $H$, so the length of $\al$ for the form $\tr XY$ is $(\al, \al) = \tr t^{2} = 2$. All roots of $\sll(n, \com)$ are long, so $\tr XY$ is the normalized form, and \eqref{eq:repindex} is \eqref{eq:dynkinpullback} for $\rho:\g \to \sll(V)$. This shows \ref{dynkinformula1}.

The linear map $\varphi:\sll(2, \com) \to \subs$ sending $E_{12}$, $E_{21}$, $E_{11} - E_{22}$ to $e$, $f$, $h$ is an isomorphism of Lie algebras by \eqref{eq:sl2triple}, and $j(\subs)$ is the Dynkin index of the composition of $\varphi$ with the inclusion $\subs \hookrightarrow \g$, because $\varphi$ preserves normalized forms. By \ref{dynkinformula1} and \eqref{eq:dynkinpullback}, $(h, h) = j(\subs)\tr(E_{11} - E_{22})^{2} = 2j(\subs)$, and $B(h, h) = 2\hv(h, h)$ by \eqref{eq:Bnorm}. This shows \ref{dynkinformula2}.
\end{proof}

\begin{lemma}\label{lem:Bh}
If $(e,f,h)$ is an $\sll(2,\com)$-triple in $\g$, then $B(h,h)\ge4\hv$, with equality if and only if $h$ is conjugate under the adjoint group to $\theta^\vee$. For a root $\mu$, $j(\subs_{\mu}) = 2/(\mu, \mu)$, which equals $1$ if and only if $\mu$ is long. For a subalgebra $\subs \subset \g$ isomorphic to $\sll(2, \com)$ the following are equivalent: $j(\subs) = 1$; $\subs$ is conjugate under the adjoint group to $\subs_{\theta}$; $\subs$ is a long-root subalgebra.
\end{lemma}
That $j(\subs) \geq 1$ also follows from Dynkin's integrality theorem; the following short argument avoids integrality and identifies the equality case.
\begin{proof}
The operator $\ad(h)$ is diagonalizable with integer eigenvalues, because $\g$ is a finite-dimensional module for $\spn\{e, f, h\}$; in particular $h$ is semisimple. By \ref{R3}, after conjugating the whole triple $(e, f, h)$ by an element of the adjoint group, which preserves $B$ and the relations \eqref{eq:sl2triple}, it can be assumed that $h \in \h$. The roots take integer values on $h$, because these values are eigenvalues of $\ad(h)$, so $h$ lies in the real span of the coroots, and by the second part of \ref{R3}, after conjugating the triple by an element of the adjoint group representing an element of the Weyl group, $h$ can be assumed dominant. Since $[h,e]=2e$ with $e\ne0$, $\mu(h)=2$ for some root $\mu$, so $\theta(h)\ge2$ by \ref{R1}. By the Cauchy--Schwarz inequality for the positive definite form $(\dum, \dum)$ on the real span of the coroots,
\begin{align}\label{eq:CSh}
4 \leq \theta(h)^2=(\theta^\vee,h)^2\le(\theta^{\vee},\theta^{\vee})(h,h)=2(h,h),
\end{align}
so $(h,h)\ge2$ and $B(h,h)=2\hv(h,h)\ge4\hv$ by \eqref{eq:Bnorm}. If $(h, h) = 2$, then equality holds throughout \eqref{eq:CSh}, so $h = c\theta^{\vee}$ with $c > 0$ and $2 = \theta(h) = 2c$. Conversely $B(\theta^\vee,\theta^\vee)=2\hv(\theta^{\vee},\theta^{\vee}) = 4\hv$. If $h = \theta^{\vee}$, then $e \in \g_{\theta}$ and $f \in \g_{-\theta}$, because $\theta$ is the only root taking the value $2$ on $\theta^{\vee}$ (the equality case of \eqref{eq:CSh} applied to the roots $\mu$ with $\mu(\theta^{\vee}) = 2$). Together with \eqref{eq:dynkinindex}, this shows that a subalgebra $\subs \simeq \sll(2, \com)$ has $j(\subs) = 1$ if and only if it is conjugate to $\subs_{\theta}$.

For the remaining claims, $\subs_{\mu}$ is spanned by an $\sll(2, \com)$-triple with neutral element $\mu^{\vee}$, and $(\mu^{\vee}, \mu^{\vee}) = 4/(\mu, \mu)$, so $j(\subs_{\mu}) = 2/(\mu, \mu)$ by \eqref{eq:dynkinindex}, and this is $1$ if and only if $\mu$ is long by \ref{R1}. The Dynkin index is invariant under conjugation. If $j(\subs) = 1$, then $\subs$ is conjugate to $\subs_{\theta}$, which is a long-root subalgebra by definition. Conversely, if $\subs = \Ad(g)\subs_{\mu}$ with $\mu$ long, then, because the Weyl group acts transitively on the long roots and is realized by the normalizer of $\mathfrak{t}_{0}$ in $G_{0}$, there is $n \in G_{0}$ with $\Ad(n)\subs_{\mu} = \subs_{\theta}$; hence $\subs$ is conjugate to $\subs_{\theta}$ and $j(\subs) = j(\subs_{\theta}) = 1$.
\end{proof}

\begin{lemma}\label{lem:sl2str}
Let $\subs\subset\g$ be a $\tau$-stable subalgebra isomorphic to $\sll(2, \com)$, with Dynkin index $j$, and let $\beta := j\hv$.
\begin{enumerate}[label=(\arabic*), ref=(\arabic*)]
\item\label{sl2a} $\subs_0:=\subs\cap\gz$ is a compact real form of $\subs$, and it has a basis $z_1,z_2,z_3$ such that
\begin{align}\label{eq:xirelations}
[z_{1}, z_{2}] &= z_{3}, \quad [z_{2}, z_{3}] = z_{1}, \quad [z_{3}, z_{1}] = z_{2}, & B(z_a,z_b) &= -\beta\,\delta_{ab}, & a, b &\in \{1, 2, 3\}.
\end{align}
For such a basis $(e, f, h) := (z_{1} + iz_{2}, -z_{1} + iz_{2}, 2iz_{3})$ is an $\sll(2, \com)$-triple with $\tau e = -f$ and $\tau h = -h$.
\item\label{sl2b} For all $x, y \in \subs$,
\begin{align}\label{eq:sl2norm}
\norm{[x,y]}^2 = \frac{1}{\beta}\Big(\norm{x}^{2}\norm{y}^{2} - |H(x, y)|^{2}\Big).
\end{align}
\end{enumerate}
\end{lemma}

\begin{proof}
Because $\tau\subs = \subs$, $\subs = \subs_{0} \oplus i\subs_{0}$, and $-B$ is a positive definite invariant inner product on $\subs_{0}$, so $\subs_{0}$ is a three-dimensional compact semisimple Lie algebra, hence isomorphic to $\su(2)$, which has a basis satisfying the first relations of \eqref{eq:xirelations}. By invariance, for $(a, b, c)$ a cyclic permutation of $(1, 2, 3)$, $B(z_{a}, z_{a}) = B(z_{a}, [z_{b}, z_{c}]) = B([z_{a}, z_{b}], z_{c}) = B(z_{c}, z_{c})$ and $B(z_{a}, z_{b}) = B(z_{a}, [z_{c}, z_{a}]) = 0$, so $B(z_{a}, z_{b}) = -\beta'\delta_{ab}$ for some $\beta' > 0$. The triple relations \eqref{eq:sl2triple} for $(z_{1} + iz_{2}, -z_{1} + iz_{2}, 2iz_{3})$ are direct computations from \eqref{eq:xirelations}, and $\tau$ acts on them as stated because $\tau z_{a} = z_{a}$. Since $B(2iz_{3}, 2iz_{3}) = 4\beta'$, \eqref{eq:dynkinindex} gives $\beta' = j\hv = \beta$.

For \ref{sl2b}, in the coordinates $x = \sum_{a=1}^{3}x_{a}z_{a}$, $y = \sum_{a=1}^{3}y_{a}z_{a}$, the bracket is the complex bilinear cross product of the coefficient vectors and $H$ is $\beta$ times the standard Hermitian product, so \eqref{eq:sl2norm} is $\beta$ times the Lagrange identity $|x \times y|^{2} = |x|^{2}|y|^{2} - |x \cdot \bar{y}|^{2}$.
\end{proof}

\begin{lemma}\label{lem:lower}
\leavevmode\noindent
\begin{enumerate}[label=(\arabic*), ref=(\arabic*)]
\item\label{lower1} Every subalgebra of $\g$ isomorphic to $\sll(2, \com)$ is conjugate under the adjoint group to a $\tau$-stable subalgebra.
\item\label{lower2} Let $0 \neq e \in \g_{\theta}$ with $\norm{e}^{2} = 2\hv$, and put $f := -\tau e$ and $h := \theta^{\vee}$. Then $(e, f, h)$ is an $\sll(2, \com)$-triple spanning the $\tau$-stable root subalgebra $\subs_{\theta}$ of \eqref{eq:rootsubalgebra}, which has Dynkin index $1$, $H(e, f) = 0$, and $\norm{[e, f]}^{2} = \norm{e}^{2}\norm{f}^{2}/\hv$.
\item\label{lower3} Every $\tau$-stable subalgebra of $\g$ isomorphic to $\sll(2, \com)$ and having Dynkin index $1$ is $\Ad(G_{0})$-conjugate to $\subs_{\theta}$.
\end{enumerate}
\end{lemma}
\begin{proof}
\ref{lower1} Let $\subs \simeq \sll(2, \com)$ and let $\subs_{c}$ be a compact real form of $\subs$. The analytic subgroup $S_{c}$ of the adjoint group $G$ of $\g$ with Lie algebra $\subs_{c}$ is compact, being the image of $SU(2)$. The analytic subgroup $G_{0}$ with Lie algebra $\gz$ is a maximal compact subgroup of $G$, every compact subgroup of $G$ is contained in a maximal compact subgroup, and the maximal compact subgroups of $G$ are conjugate \cite[Chapter VI]{Helgason}, so $gS_{c}g^{-1} \subset G_{0}$ for some $g \in G$. Then $\Ad(g)\subs_{c} \subset \gz$, and $\Ad(g)\subs = \Ad(g)\subs_{c} \oplus i\Ad(g)\subs_{c}$ is $\tau$-stable.

\ref{lower2} By \eqref{eq:tauroot}, $f \in \g_{-\theta}$, and $B(e, f) = -B(e, \tau e) = \norm{e}^{2} = 2\hv$, so by \ref{R2} and $t_{\theta} = \theta^{\vee}/(2\hv)$, $[e, f] = \theta^{\vee} = h$. Since $\theta(\theta^{\vee}) = 2$, $(e, f, h)$ is an $\sll(2, \com)$-triple, and $\tau f = -\tau\tau e = -e$ and $\tau h = [\tau e, \tau f] = [-f, -e] = -h$, so its span is $\tau$-stable; by \eqref{eq:dynkinindex} its Dynkin index is $(\theta^{\vee}, \theta^{\vee})/2 = 1$. By \ref{R2}, $H(e, f) = -B(e, \tau f) = B(e, e) = 0$, and \eqref{eq:sl2norm} with $\beta = \hv$ gives the last claim.

\ref{lower3} Let $\subs$ be $\tau$-stable of Dynkin index $1$, let $z_{1}, z_{2}, z_{3}$ be as in Lemma \ref{lem:sl2str}\ref{sl2a}, and let $(e, f, h) = (z_{1} + iz_{2}, -z_{1} + iz_{2}, 2iz_{3})$. Since $z_{3} \in \gz$ is $\Ad(G_{0})$-conjugate into $\mathfrak{t}_{0}$, it can be assumed that $h \in i\mathfrak{t}_{0} \subset \h$. The roots take integer values on $h$, because these are eigenvalues of $\ad(h)$, so $h$ lies in the real span of the coroots, and since the Weyl group is realized by the normalizer of $\mathfrak{t}_{0}$ in $G_{0}$, $h$ can be assumed dominant. As $j(\subs) = 1$, $(h, h) = 2$ by \eqref{eq:dynkinindex}, and the equality case in the proof of Lemma \ref{lem:Bh} gives $h = \theta^{\vee}$, $e \in \g_{\theta}$, and $f \in \g_{-\theta}$, so $\subs = \subs_{\theta}$.
\end{proof}

\section{The functions $F$, $\sect$, $W$ and the comass of the Cartan form}\label{sec:reform}

This section describes the relations between the function $F$ to be maximized and the auxiliary functions $\sect$ and $W$ and the comass of the Cartan $3$-form.

Let $\stiefel(\g)$ be the set of pairs $(x, y) \in \g \times \g$ that are linearly independent over $\com$, and let $\grass(2, \g)$ be the Grassmannian of complex $2$-planes in $\g$. Define functions
\begin{align}\label{eq:FKWdef}
\begin{aligned}
&F(x,y) := \frac{\norm{[x, y]}^2}{\norm{x}^{2}\norm{y}^{2}}, && x, y \in \g\setminus\{0\},\\
&\sect(x,y) := \frac{\norm{[x, y]}^2}{\norm{x}^{2}\norm{y}^{2} - |H(x, y)|^{2}}, && (x, y) \in \stiefel(\g),\\
& W(x,y,z):=\frac{|\omega(x,y,z)|^2}{\norm{x}^{2}\,\norm{y}^{2}\,\norm{z}^{2}}, && x, y, z \in \g\setminus\{0\}.
\end{aligned}
\end{align}
The denominator of $\sect$ is positive on $\stiefel(\g)$ by the Cauchy--Schwarz inequality. The function $\sect$ is called the \emph{sectional nonassociativity}. If $u = ax + by$ and $v = cx + dy$ with $ad - bc \neq 0$, then the numerator and the denominator of $\sect$ are both multiplied by $|ad - bc|^{2}$, so
\begin{align}\label{eq:sectinvariance}
\sect(ax + by, cx + dy) = \sect(x, y),
\end{align}
and $\sect$ descends to a function on $\grass(2, \g)$, also denoted $\sect$. The functions $F$ and $W$ are invariant under rescaling each argument by a nonzero complex number, so they descend to $\proj(\g)^{2}$ and $\proj(\g)^{3}$. The \emph{comass} of the Cartan form with respect to $H$ is
\begin{align}\label{eq:comass}
\comass{\om} := \sup\big\{|\om(x, y, z)|: \norm{x} = \norm{y} = \norm{z} = 1\big\} = \sup_{x, y, z \in \g \setminus \{0\}} W(x, y, z)^{1/2}.
\end{align}

\begin{remark}\label{rem:comass}
The quantity \eqref{eq:comass} is a comass in the usual sense \cite{Harvey-Lawson}: for a real alternating form on a Euclidean space, the supremum over unit vectors equals the maximum over orthonormal frames, by Hadamard's inequality. Regard $\g$ as the Euclidean space $(\g^{\rea}, \re H)$, on which $\re\om$ is a real alternating trilinear form. Since $\re\om \leq |\om|$, and replacing $x$ by $e^{i\phi}x$ makes $\om(x, y, z)$ real and nonnegative,
\begin{align}\label{eq:comassreal}
\comass{\om} = \sup\big\{\re\om(x, y, z): \norm{x} = \norm{y} = \norm{z} = 1\big\} = \comass{\re\om},
\end{align}
the comass of $\re\om$ on $(\g^{\rea}, \re H)$. On the other hand, $\om$ restricted to $\gz$ is a real alternating form on the Euclidean space $(\gz, -B)$, whose comass is the quantity studied by Tasaki \cite{Tasaki-calibration} (Remark \ref{rem:calibration}). By \eqref{eq:comassreal} and Corollary \ref{thm:compact}\ref{compact1} below, Theorem \ref{thm:A} asserts that these two comasses are equal, both being $(\hv)^{-1/2}$.

In the language of tensors, \eqref{eq:comass} is the injective norm of $\om$ regarded as a complex $3$-tensor on the Hermitian space $(\g, H)$, which in the literature on tensor computations and optimization is called the \emph{spectral norm}; for alternating tensors it coincides with the comass, by the Hadamard argument above. For tensors of order at least $3$, the spectral norm of a real tensor depends in general on whether the tensor is regarded as real or as complex \cite{Friedland-Lim-nuclear}. The restriction of $\om$ to $\gz$ is a real $3$-tensor on $(\gz, -B)$ whose complexification is $\om$ on $(\g, H)$, and by Theorem \ref{thm:A} and Corollary \ref{thm:compact}\ref{compact1} its complex and real spectral norms coincide. This is a property of the Cartan form, not a formal consequence of the definitions.

The relevance of comass to inequalities of DDVV type was observed by Choi and Lu \cite{Choi-Lu}, who used calibrations to prove cases of the DDVV inequality; the present paper goes in the other direction, using the extremal problem for the commutator to compute the comass of the Cartan $3$-form on the Hermitian space $(\g, H)$.
\end{remark}

\begin{lemma}\label{fkwmaxlemma}
\leavevmode\noindent
\begin{enumerate}[label=(\arabic*), ref=(\arabic*)]
\item\label{fkw1} $F(x, y) \leq \sect(x, y)$ for $(x, y) \in \stiefel(\g)$, with equality if and only if $H(x, y) = 0$ or $[x, y] = 0$.
\item\label{fkw2} $W(x, y, z) \leq F(x, y)$, with equality if and only if $[x, y] = 0$ or $z \in \com\,\tau[x, y]$. In particular $F(x,y)=\max_{z \in \g \setminus \{0\}}W(x,y,z)$.
\item\label{fkw3} $F$, $\sect$, and $W$ attain their suprema, and
\begin{align}\label{eq:fkwcomass}
\max_{\sigma \in \grass(2, \g)} \sect(\sigma) = \max_{x, y \in \g \setminus \{0\}} F(x, y) = \max_{x, y, z \in \g \setminus \{0\}} W(x, y, z) =\comass{\om}^2 > 0.
\end{align}
\end{enumerate}
\end{lemma}

\begin{proof}
Claim \ref{fkw1} is clear from \eqref{eq:FKWdef}, and \ref{fkw2} is the Cauchy--Schwarz inequality, by \eqref{eq:omegaformulas},
\begin{align}\label{eq:Fomega1}
|\omega(x,y,z)|^{2} = |H([x, y], \tau z)|^{2} \leq \norm{[x, y]}^{2}\norm{z}^{2},
\end{align}
with its equality case, divided by $\norm{x}^{2}\norm{y}^{2}\norm{z}^{2}$. For \ref{fkw3}, the three functions descend to compact manifolds, so attain their maxima; $\max\sect = \max F$ because every plane has an $H$-orthogonal basis, on which $F = \sect$; $\max F = \max W = \comass{\om}^{2}$ by \eqref{eq:Fomega1} and \eqref{eq:comass}, which gives \eqref{eq:fkwcomass}; and the maximum is positive because $\g$ is not abelian.
\end{proof}

By Lemma \ref{fkwmaxlemma}, the inequality \eqref{eq:main} of Theorem \ref{thm:A} is equivalent to $\comass{\om}^{2} \leq 1/\hv$. Lemmas \ref{lem:sl2str} and \ref{lem:lower} identify the natural candidates for maximum points. By \eqref{eq:sl2norm}, $\sect$ is constant on $\grass(2, \subs)$ for a $\tau$-stable subalgebra $\subs \simeq \sll(2, \com)$ of Dynkin index $j$:
\begin{align}\label{eq:sectconstant}
\sect(\sigma) = \frac{1}{j\hv} \qquad \text{for every } \sigma \in \grass(2, \subs) \subset \grass(2, \g).
\end{align}
By Lemmas \ref{lem:Bh} and \ref{lem:lower}\ref{lower2}, the largest value of \eqref{eq:sectconstant} is $1/\hv$, attained for $j = 1$, so that
\begin{align}\label{eq:lowerbound}
\comass{\om}^{2} \geq \frac{1}{\hv}.
\end{align}
The content of Theorem \ref{thm:A} is the reverse inequality, together with the statement that every maximum point is of this form.

On $\gz$ the optimal constant can be computed from results of Vinberg, and its value is the calibration property of Tasaki's form. The rest of this section records these results in order to make the comparison with Theorem \ref{thm:A} precise.

Let $G_{0}$, $\mathfrak{t}_{0}$, $\h$, $\Phi$, and $s_{\al}$ be as in Section \ref{sec:prelim}; $G_{0}$ also preserves $F$.

By a theorem of Vinberg \cite[Theorem~2]{Vinberg-invariantnorms}, if $N$ is an $\Ad(G_{0})$-invariant norm on $\gz$, then for every $x \in \gz$ the operator norm $\sup\{N([x, y])/N(y): y \in \gz \setminus \{0\}\}$ of $\ad(x)$ is the spectral radius of $\ad(x)$, independently of $N$; Vinberg also showed that the invariant norms on $\gz$ correspond to the Weyl group invariant norms on $\mathfrak{t}_{0}$ \cite[Theorem~1]{Vinberg-invariantnorms} (see also \cite{Berestovksii-Nikonorov}). Since every element of $\gz$ is conjugate into $\mathfrak{t}_{0}$, and the spectral radius of $\ad(t)$, $t \in \mathfrak{t}_{0}$, is $\max_{\al \in \Phi}|\al(t)|$, this gives the following.
\begin{corollary}\label{cor:vinberg}
For every $\Ad(G_{0})$-invariant norm $N$ on $\gz$,
\begin{align}\label{eq:vinbergbound}
\sup_{x, y \in \gz\setminus\{0\}}\frac{N([x, y])}{N(x)N(y)} = \max_{\al \in \Phi}\,\max\big\{|\al(t)|: t \in \mathfrak{t}_{0},\ N(t) = 1\big\}.
\end{align}
\end{corollary}
For $N = \norm{\dum}$, \eqref{eq:salpha} and the Cauchy--Schwarz inequality give $\max\{|\al(t)| : t \in \mathfrak{t}_{0},\ \norm{t} = 1\} = \norm{s_{\al}}$, so the right side of \eqref{eq:vinbergbound} is $\max_{\al \in \Phi}\norm{s_{\al}} = (\hv)^{-1/2}$, attained at $t = s_{\theta}/\norm{s_{\theta}}$. This is the restriction of \eqref{eq:main} to $\gz$. It is also Tasaki's inequality \cite[(8) in the proof of Theorem~5]{Tasaki-calibration}, and, in a different normalization and for $x \in \gz \cup i\gz$, $y \in \g$, the computation of Wang's constants in \cite{Wang-nilpotent}, \cite[Proposition 4.1]{Adeboye-Wang-Wei}, which enter Wang's quantitative form of the Kazhdan--Margulis theorem for the complex group regarded as a real semisimple Lie group with Cartan decomposition $\g = \gz \oplus i\gz$. These constants are suprema of the operator norm of $\ad(x)$ over unit vectors $x$ in $\gz$ or in $i\gz$, where $\ad(x)$ is normal; Theorem \ref{thm:A} removes this restriction on $x$. The following corollary records it together with the equality case.

\begin{corollary}\label{thm:compact}
\leavevmode\noindent
\begin{enumerate}[label=(\arabic*), ref=(\arabic*)]
\item\label{compact1} For $x \in \gz\setminus\{0\}$, $\sup_{y \in \g \setminus \{0\}}F(x, y)$ is attained at some $y \in \gz$, and $F(x, y) \leq 1/\hv$ for all $x \in \gz\setminus\{0\}$ and $y \in \g \setminus \{0\}$, with equality for some $x, y \in \gz$.
\item\label{compact2} A pair $(x, y) \in (\gz \setminus \{0\})^{2}$ satisfies $F(x, y) = 1/\hv$ if and only if it is $\Ad(G_{0})$-conjugate to a pair $(c s_{\theta}, y')$ with $c \in \rea \setminus \{0\}$ and $0 \neq y' \in \gz \cap (\g_{\theta} \oplus \g_{-\theta})$.
\end{enumerate}
\end{corollary}

\begin{proof}
\ref{compact1} For $x \in \gz$, $\ad(x)$ is skew-Hermitian by \eqref{basic4} and commutes with $\tau$, so its operator norm on $\g$ is its spectral radius and is attained on a $\tau$-stable eigenspace of $\qad(x) = -\ad(x)^{2}$, which meets $\gz$. Hence $\sup_{y \in \g\setminus\{0\}}F(x, y) = \sup_{y \in \gz\setminus\{0\}}F(x, y)$ is attained in $\gz$, and \ref{compact1} follows from Corollary \ref{cor:vinberg} with $N = \norm{\dum}$.

\ref{compact2} If $F(x, y) = 1/\hv$, conjugate $x$ into $\mathfrak{t}_{0}$. By \eqref{eq:vinbergbound} and \eqref{eq:salpha}, $|\al(x)| = |B(s_{\al}, x)| = \norm{x}/(\hv)^{1/2}$ for some root $\al$, so equality holds in the Cauchy--Schwarz inequality $|B(s_{\al}, x)| \leq \norm{s_{\al}}\norm{x} \leq \norm{x}/(\hv)^{1/2}$; hence $\al$ is long and $x \in \rea s_{\al}$, and after a Weyl group element, $x = cs_{\theta}$. For a root $\mu$, \eqref{eq:salpha} gives $|\mu(s_{\theta})|^{2} = B(s_{\mu}, s_{\theta})^{2} \leq \norm{s_{\mu}}^{2}\norm{s_{\theta}}^{2} \leq \norm{s_{\theta}}^{2}/\hv$, with equality if and only if $\mu$ is long and $s_{\mu} \in \rea s_{\theta}$, that is, $\mu = \pm\theta$; so the top eigenspace of $\qad(cs_{\theta})$ is $\g_{\theta} \oplus \g_{-\theta}$, and $y$ lies in its intersection with $\gz$. The converse is clear, which proves \ref{compact2}.
\end{proof}

Since on $\gz$ the sectional curvature of the bi-invariant metric $-B$ is $\tfrac{1}{4}F$ on orthogonal pairs, for pairs in $\gz$ Corollary \ref{thm:compact} amounts to the determination in \cite[Theorems~2 and~4]{Meshcheryakov} of the maximal sectional curvature of $(G_{0}, -B)$, written there as $1/(2i(\mathrm{ad}))$ with $i(\mathrm{ad}) = 2\hv$, and of the planes realizing it (Remark \ref{rem:meshcheryakov}).

For $x \in \g$, the operator $\ad(x)$ need not be normal, and its operator norm need not equal its spectral radius: for nilpotent $x \neq 0$ the spectral radius is $0$ while the operator norm is positive. Neither Vinberg's theorem nor Corollary \ref{cor:vinberg} applies to such $x$. Theorem \ref{thm:A} shows that nonetheless, for the Hermitian norm $H$, the supremum over $\g$ equals the supremum over $\gz$.

\begin{remark}\label{rem:calibration}
A closed $p$-form is a calibration if its comass, its maximum on unit simple $p$-vectors, is at most $1$ \cite{Harvey-Lawson}. Tasaki \cite[Theorem~5]{Tasaki-calibration} proved that on a compact simple Lie group with bi-invariant metric $-B$, the bi-invariant $3$-form 
\begin{align}\label{eq:tasakiform}
\phi(x, y, z) = (\hv)^{1/2}\,(-B)([x, y], z) = -(\hv)^{1/2}\,\om(x, y, z), \qquad x, y, z \in \gz,
\end{align}
(normalized there by the length $(\hv)^{-1/2}$ of the highest root for the metric $-B$; see \eqref{eq:salpha}) is a calibration, whose calibrated $3$-planes are exactly the $\Ad(G_{0})$-conjugates of the long-root subalgebra $\su(2) \subset \gz$, suitably oriented. Ohnita and Tasaki \cite[Theorem~2 and Lemma~5]{Ohnita-Tasaki} showed that every submanifold calibrated by $\phi$ is totally geodesic, hence an open piece of a congruent copy of the long-root subgroup, so that this subgroup is, up to congruence, the unique volume minimizer in its real homology class. Since calibrated planes exist, the calibration property says that the comass of $\om$ with respect to the Euclidean form $-B$ on $\gz$ is $(\hv)^{-1/2}$, in agreement with Corollary \ref{thm:compact}\ref{compact1}. See also \cite{Hitchin-calibrations}, where the calibration \eqref{eq:tasakiform} and its relation with the principal three-dimensional subgroup are discussed.
\end{remark}

\section{The Nahm algebra and the proof of Theorems \ref{thm:A} and \ref{thm:C}}\label{sec:nahm}
\subsection{Preliminaries on the Nahm algebra}\label{sec:nahmprelim}
Write elements of $\g^{3} = \g \oplus \g \oplus \g$ as $a = (a_{1}, a_{2}, a_{3})$, and for $i \in \{1, 2, 3\}$ let $\iota_{i}:\g \to \g^{3}$ be the embedding of $\g$ as the $i$th factor, so that $a = \sum_{i=1}^{3}\iota_{i}(a_{i})$. 

\begin{definition}\label{def:nahm}
The \emph{Nahm algebra} $(\g^{3}, \tmlt)$ of a Lie algebra $\g$ is $\g^{3}$ with the multiplication $\tmlt$ of \eqref{nahmproduct}.
\end{definition}

Nahm algebras were introduced by Kinyon and Sagle \cite{Kinyon-Sagle} with the multiplication normalized to be one half of \eqref{nahmproduct}; that normalization makes the diagonal embedding of a commutative algebra into its triple algebra a homomorphism, a property not used here. The multiplication \eqref{nahmproduct} is the Lie algebra case of the \emph{triple algebra} of an arbitrary algebra \cite[Section~11]{Fox-trace}, \cite{Fox-cubicpoly, Fox-Tkachev}, and its normalization is the one for which the Killing form polarizes the cubic form \eqref{eq:cubic}, as Lemma \ref{lem:nahmbasic} shows.

The forms $B$ and $H$ and the conjugation $\tau$ are extended to $\g^{3}$ componentwise, and the same symbols are used for the extensions:
\begin{align}\label{eq:componentwise}
B(a,b):=\sum_{i=1}^3B(a_i,b_i), \qquad H(a,b):=\sum_{i=1}^3H(a_i,b_i), \qquad \tau a:=(\tau a_1,\tau a_2,\tau a_3).
\end{align}
The facts listed at the beginning of Section \ref{sec:prelim} for $H$, $B$, and $\tau$ on $\g$ hold verbatim for the extensions \eqref{eq:componentwise}. 
The real part $\re H$ of $H$ is a real inner product on $\g^{3}$, with norm $\nnorm{a}:=H(a,a)^{1/2}$; the norm $\nnorm{\dum}$ on $\g^{3}$ is not to be confused with either the norm $\norm{\dum}$ on $\g$ or with the modulus of a complex number, although $\nnorm{a}^{2} = \sum_{i=1}^{3}\norm{a_{i}}^{2}$. Let $J$ denote multiplication by $i$ on $\g^{3}$, so that $\re H(Ja, b) = -\re H(a, Jb)$. Recall from \eqref{eq:cubic} the cubic form $C(a) = \om(a_{1}, a_{2}, a_{3})$ on $\g^{3}$, in terms of which $W(a) = |C(a)|^{2}/(\norm{a_{1}}^{2}\norm{a_{2}}^{2}\norm{a_{3}}^{2})$.

\begin{lemma}\label{lem:nahmbasic}
Let $\lin:\g^{3} \to \eno(\g^{3})$ be the multiplication endomorphism of the Nahm algebra $(\g^{3}, \tmlt)$ defined by $\lin(a)b := a \tmlt b$. 
For $a, b, c \in \g^{3}$:
\begin{enumerate}[label=(\arabic*), ref=(\arabic*)]
\item\label{nahm1} $\tmlt$ is complex bilinear and commutative, and $B(a \tmlt b, c)$ is the full polarization of $C$:
\begin{align}\label{eq:polarization}
B(a \tmlt b, c) = C(a + b + c) - C(a+b) - C(b+c) - C(c+a) + C(a) + C(b) + C(c).
\end{align}
In particular $B(a \tmlt b, c)$ is symmetric in $a$, $b$, $c$, so $B$ is invariant: $B(a \tmlt b, c) = B(a, b \tmlt c)$.
\item\label{nahm2} $a \tmlt a = 2\big([a_{2}, a_{3}], [a_{3}, a_{1}], [a_{1}, a_{2}]\big)$, and the first and second derivatives of $C$ at $a$ are
\begin{align}\label{eq:derivativesC}
dC_{a}(b) = \tfrac{1}{2}B(a \tmlt a, b), \qquad d^{2}C_{a}(b, c) = B(a \tmlt b, c).
\end{align}
\item\label{nahm3} $\tau$ is a conjugate-linear order $2$ automorphism of $(\g^{3}, \tmlt)$ and $H(a\tmlt b, c) = H(b, \tau a \tmlt c)$; that is, the $H$-adjoint $\lin(a)^{\dagger}$ of $\lin(a)$ satisfies $\lin(a)^{\dagger} = \lin(\tau a)$.
\item\label{nahm4} $\tr \lin(a) = 0$.
\end{enumerate}
\end{lemma}
\begin{proof}
Bilinearity and commutativity are evident from \eqref{nahmproduct}. By \eqref{nahmproduct} and the cyclic invariance of $\om$,
\begin{align}\label{eq:polarizationproof}
B(a \tmlt b, c) = \sum_{\si \in S_{3}}\sgn(\si)\om(a_{\si(1)}, b_{\si(2)}, c_{\si(3)}),
\end{align}
where $S_{3}$ is the symmetric group and $\sgn:S_{3}\to \{\pm 1\}$ the sign representation, and this equals the right-hand side of \eqref{eq:polarization} by the trilinearity of $\om$.
This proves \eqref{eq:polarization} and \ref{nahm1}. The formula for $a \tmlt a$ is \eqref{nahmproduct}, and \eqref{eq:derivativesC} follows from \eqref{eq:polarization}, since the derivatives of the cubic form $C$ are given by its polarization. Because $\tau$ is a conjugate-linear order $2$ automorphism of $\g$, \eqref{nahmproduct} gives $\tau(a \tmlt b) = \tau a \tmlt \tau b$. By \ref{nahm1}, $H(a \tmlt b, c) = -B(a \tmlt b, \tau c) = -B(b, a \tmlt \tau c) = -B(b, \tau(\tau a \tmlt c)) = H(b, \tau a \tmlt c)$. This shows \ref{nahm3}.
Finally, by \eqref{nahmproduct} the $i$th component of $\lin(a)b$ does not depend on $b_{i}$, so $\lin(a)$ has vanishing diagonal blocks and hence has trace zero. This shows \ref{nahm4}.
\end{proof}
By Lemma \ref{lem:nahmbasic}, $(\g^{3}, \tmlt)$ is a commutative algebra with a nondegenerate invariant symmetric bilinear form and trace-free multiplication operators, which is the class of algebras studied in \cite{Fox-cubicpoly, Fox-trace}. The novel ingredient here is the conjugation $\tau$, and with it the notion of a $\tau$-twisted idempotent.

\begin{definition}\label{def:twisted}
In a complex algebra $\alg$ equipped with a conjugation $\tau$, that is, a conjugate-linear algebra automorphism satisfying $\tau^{2} = \id_{\alg}$, an element $e \in \alg$ is a \emph{$\tau$-twisted idempotent} if $e \tmlt e = \tau e$. The set of nonzero $\tau$-twisted idempotents in $\alg$ is denoted $\idem_{\tau}(\alg)$.
\end{definition}
In what follows, $\tau$-twisted idempotents are those $\idem_{\tau}(\g^{3}, \tmlt)$ of the Nahm algebra $(\g^{3}, \tmlt)$ with the conjugation $\tau$ of \eqref{eq:componentwise}.

\subsection{Critical points of $W$ and $\tau$-twisted idempotents}
Write $\crtnt(f)$ for the set of critical points of a smooth function $f$ at which $f$ is nonzero. 

The torus $(\com^{\ast})^{3}$ acts effectively on $\g^{3}$ by $c \cdot a = (c_{1}a_{1}, c_{2}a_{2}, c_{3}a_{3})$, and this action satisfies $W(c \cdot a) = W(a)$ and
\begin{align}\label{caca}
(c\cdot a) \tmlt (c\cdot a) = \hat{c}\cdot (a\tmlt a),
\end{align}
where 
\begin{align}\label{hatc}
\hat{c} = (c_{2}c_{3}, c_{3}c_{1}, c_{1}c_{2}).
\end{align}

\begin{lemma}\label{lem:critW}
An element $a \in (\g \setminus\{0\})^{3}$ with $C(a) \neq 0$ lies in $\crtnt(W)$ if and only if
\begin{align}\label{eq:W}
&[a_{1}, a_{2}] = -\frac{C(a)}{\norm{a_{3}}^{2}}\tau a_{3},&
&[a_{2}, a_{3}] = -\frac{C(a)}{\norm{a_{1}}^{2}}\tau a_{1},&
&[a_{3}, a_{1}] = -\frac{C(a)}{\norm{a_{2}}^{2}}\tau a_{2}.&
\end{align}
Equivalently, $a \in (\g \setminus\{0\})^{3}$ lies in $\crtnt(W)$ if and only if $C(a) \neq 0$ and $a \tmlt a \in (\com^{\ast})^{3}\cdot \tau a$.
\end{lemma}
\begin{proof}
Let $a \in (\g\setminus\{0\})^{3}$ satisfy $C(a) \neq 0$, and let $\Pi(a) := \prod_{i=1}^{3}\norm{a_{i}}^{2}$. For $b \in \g^{3}$, by \eqref{eq:derivativesC}, 
\begin{align}\label{eq:dC}
d(|C|^{2})_{a}(b) = 2\re\big(\overline{C(a)}\tfrac{1}{2}B(b, a \tmlt a)\big) = -\re H(b, C(a)\tau(a \tmlt a)),
\end{align}
while 
\begin{align}\label{eq:dPi}
d\Pi_{a}(b) = 2\Pi(a)\sum_{i=1}^{3}\re H(b_{i}, a_{i})/\norm{a_{i}}^{2}. 
\end{align}
Since $W(a) = |C(a)|^{2}/\Pi(a)$, \eqref{eq:dC} and \eqref{eq:dPi} show 
\begin{align}\label{eq:dW}
dW_{a}(b) = -\frac{1}{\Pi(a)}\re H\big(b, C(a)\tau(a \tmlt a)\big) - 2W(a)\sum_{i=1}^{3}\frac{\re H(b_{i}, a_{i})}{\norm{a_{i}}^{2}}.
\end{align}
Because $\re H$ is nondegenerate and $W(a)\Pi(a) = |C(a)|^{2}$, by \eqref{eq:dW}, $dW_{a} = 0$ if and only if $C(a)\tau(a \tmlt a)_{i} = -2|C(a)|^{2}a_{i}/\norm{a_{i}}^{2}$ for each $i$, that is, $(a \tmlt a)_{i} = -2C(a)\tau a_{i}/\norm{a_{i}}^{2}$, which yields \eqref{eq:W} by Lemma \ref{lem:nahmbasic}\ref{nahm2}.

By \eqref{eq:W}, if $a \in \crtnt(W)$, then $a \tmlt a \in (\com^{\ast})^{3}\cdot \tau a$. If $C(a) \neq 0$ and $a \tmlt a = c \cdot \tau a$ with $c \in (\com^{\ast})^{3}$, then $2[a_{1}, a_{2}] = c_{3}\tau a_{3}$ by Lemma \ref{lem:nahmbasic}\ref{nahm2}, so by \eqref{eq:omegaformulas}, $C(a) = B([a_{1}, a_{2}], a_{3}) = \tfrac{1}{2}c_{3}B(\tau a_{3}, a_{3}) = -\tfrac{1}{2}c_{3}\norm{a_{3}}^{2}$, and $\tfrac{1}{2}c_{3} = -C(a)/\norm{a_{3}}^{2}$. As the same identities hold when the indices are permuted cyclically, this shows \eqref{eq:W}, so $a \in \crtnt(W)$.
\end{proof}

\begin{lemma}\label{thm:idem}
\leavevmode\noindent
\begin{enumerate}[label=(\arabic*), ref=(\arabic*)]
\item\label{idem2} For $e\in \idem_{\tau}(\g^{3}, \tmlt)$, the components $e_{1}$, $e_{2}$, and $e_{3}$ are pairwise $H$-orthogonal, satisfy $\norm{e_{1}} = \norm{e_{2}} = \norm{e_{3}} = 3^{-1/2}\nnorm{e} > 0$, and, for distinct $i, j \in \{1, 2, 3\}$,
\begin{align}\label{eq:idemnorms}
 W(e) = F(e_{i}, e_{j}) =  \frac{3}{4H(e, e)} = -\frac{1}{8C(e)}.
\end{align}
\item\label{idem1} $\crtnt(W) = (\com^{\ast})^{3}\cdot \idem_{\tau}(\g^{3}, \tmlt)$.
\item\label{idem3} The Lie subalgebra $\subs$ of $\g$ generated by the components of $e \in \idem_{\tau}(\g^{3}, \tmlt)$ is $\tau$-stable and semisimple. 
\end{enumerate}
\end{lemma}
\begin{proof}
By Lemma \ref{lem:nahmbasic}\ref{nahm2}, for $(i, j, k)$ a cyclic permutation of $(1, 2, 3)$, $e \tmlt e = \tau e$ means that $2[e_{j}, e_{k}] = \tau e_{i}$. By the invariance of $B$, $2\om(e_{i}, e_{j}, e_{k}) = 2C(e) = 2B(e_{i}, [e_{j}, e_{k}]) = B(e_{i}, \tau e_{i}) = -\norm{e_{i}}^{2}$, so the components of $e$ have a common norm, which is positive because $e \neq 0$. Similarly $H(e_{i}, e_{j}) = -B(e_{i}, \tau e_{j}) = -2B(e_{i}, [e_{k}, e_{i}]) = 2B([e_{i}, e_{i}], e_{k}) = 0$. The identity \eqref{eq:idemnorms} and claim \ref{idem2} follow.

By \ref{idem2}, every $e \in \idem_{\tau}(\g^{3}, \tmlt)$ has nonzero components and $C(e) \neq 0$, so, since $e \tmlt e = \tau e$, Lemma \ref{lem:critW} gives $(\com^{\ast})^{3}\cdot \idem_{\tau}(\g^{3}, \tmlt) \subset \crtnt(W)$. Each map $a \mapsto c\cdot a$, $c \in (\com^{\ast})^{3}$, is a linear automorphism of $\g^{3}$ preserving $W$, so it maps critical points of $W$ to critical points of $W$. Let $a \in \crtnt(W)$ and $c \in (\com^{\ast})^{3}$. By \eqref{caca}, \eqref{hatc}, \eqref{eq:W}, and the conjugate-linearity of $\tau$,
\begin{align}\label{cdotasquared}
\begin{aligned}
(c\cdot a)\tmlt (c\cdot a) - \tau(c \cdot a) = \hat{c}\cdot (a\tmlt a) - \bar{c}\cdot \tau a = -2C(a)\left(\tfrac{c_{2}c_{3}}{\norm{a_{1}}^{2}}, \tfrac{c_{1}c_{3}}{\norm{a_{2}}^{2}}, \tfrac{c_{1}c_{2}}{\norm{a_{3}}^{2}} \right)\cdot \tau a - \bar{c}\cdot \tau a.
\end{aligned}
\end{align}
Let $z \in \com^{\ast}$ solve $z^{3} = -\tfrac{C(a)^{2}}{|C(a)|^{2}}$ and define $c \in (\com^{\ast})^{3}$ by
\begin{align}\label{eq:cchoice}
c = \tfrac{z}{2C(a)}\left(\norm{a_{2}}\norm{a_{3}}, \norm{a_{3}}\norm{a_{1}}, \norm{a_{1}}\norm{a_{2}}\right).
\end{align}
Writing $P := \norm{a_{1}}\norm{a_{2}}\norm{a_{3}}$ and $\kappa := z/(2C(a))$, the choice \eqref{eq:cchoice} gives $c_{j}c_{k} = \kappa^{2}P\norm{a_{i}}$ and $\bar{c}_{i} = \bar{\kappa}P/\norm{a_{i}}$ for $\{i, j, k\} = \{1, 2, 3\}$, so the $i$th coefficient in \eqref{cdotasquared} is $-(2C(a)\kappa^{2} + \bar{\kappa})P/\norm{a_{i}}$, which vanishes because $|z| = 1$ and $z^{3}\overline{C(a)} = -C(a)$. Hence \eqref{cdotasquared} equals $0$, so $c \cdot a \in \idem_{\tau}(\g^{3}, \tmlt)$ and $a = c^{-1}\cdot(c \cdot a) \in (\com^{\ast})^{3}\cdot\idem_{\tau}(\g^{3}, \tmlt)$. This shows \ref{idem1}. 

By \eqref{nahmproduct}, the components of $e \tmlt e = \tau e$ are brackets of components of $e$, so they lie in $[\subs, \subs] \subset \subs$. Hence $\tau$ maps the generators of $\subs$, and so all iterated brackets of them, into $\subs$, and so $\tau\subs = \subs$ because $\tau^{2} = \id_{\g}$. Then $\subs = \subs_{0} \oplus i\subs_{0}$ with $\subs_{0} = \subs \cap \gz$, and $-B$ is a positive definite invariant inner product on $\subs_{0}$, so $\subs_{0}$, and hence $\subs$, is reductive, $\subs = \mathfrak{z}(\subs) \oplus [\subs, \subs]$. The components of $e = \tau(e \tmlt e) = \tau e \tmlt \tau e$ are brackets of components of $\tau e$, which lie in $\subs$, so they lie in $[\subs, \subs]$. Hence $\subs = [\subs, \subs]$, $\mathfrak{z}(\subs) = 0$, and $\subs$ is semisimple. This shows \ref{idem3}.
\end{proof}

Note that the Lie subalgebra $\subs$ generated by the components of $e \in \idem_{\tau}(\g^{3}, \tmlt)$ depends only on the orbit $(\com^{\ast})^{3}\cdot e$ of $e$.

By Lemma \ref{thm:idem}, $\max W = \comass{\om}^{2}$ is attained exactly at the $(\com^{\ast})^{3}$-orbits of the $\tau$-twisted idempotents of minimal norm, and \eqref{eq:idemnorms} shows that Theorem \ref{thm:C}\ref{thmC2} is equivalent to $\comass{\om}^{2} = 1/\hv$ together with the description of the maximizing idempotents.

\subsection{The Hessian of $W$}\label{sec:hessian}
The purpose of this subsection is to calculate the Hessian of $W$ at $e \in \idem_{\tau}(\g^{3}, \tmlt)$.

\begin{lemma}\label{hess2}
For $e \in \idem_{\tau}(\g^{3}, \tmlt)$ and $i \in \{1, 2, 3\}$, let $\hat{e}_{i} := \iota_{i}(e_{i}) \in \g^{3}$ be the image of $e_{i}$ under the linear embedding of $\g$ as the $i$th factor, and define
\begin{align}\label{eq:UR}
&\Usp_{e} := \spn_{\rea}\{\hat{e}_{i}, J\hat{e}_{i}: i \in \{1, 2, 3\}\}, &
&\Uperp_{e} = \{b \in \g^{3}: H(b_{i}, e_{i}) = 0 \text{ for } i \in \{1, 2, 3\}\}.
\end{align}
Then $\Uperp_{e}$ is the orthogonal complement with respect to $\re H$ of $\Usp_{e}$ in $\g^3$ and the real linear operator $\tau\lin(e)$ is self-adjoint with respect to the real inner product $\re H$, anticommutes with $J$, and preserves $\Usp_{e}$ and $\Uperp_{e}$.
\end{lemma}
\begin{proof}
Note that $\Usp_{e}$ is the real tangent space at $e$ of the orbit $(\com^{\ast})^{3}\cdot e$ and consequently $J\Usp_{e} = \Usp_{e}$. That $\Uperp_{e}$ is the orthogonal complement of $\Usp_{e}$ with respect to $\re H$, follows from $H(b_{i}, e_{i}) = \re H(b, \hat{e}_{i}) + i\re H(b, J\hat{e}_{i})$, which shows that $\Uperp_{e} = \{b \in \g^{3}: \re H(b_{i}, e_{i}) = 0 = \re H(b_{i}, Je_{i}) \text{ for } i \in \{1, 2, 3\}\}$, from which there follows $J\Uperp_{e} = \Uperp_{e}$ because $\re H(Jb_{i}, e_{i}) = - \re H(b_{i}, Je_{i})$.

By the invariance of $B$ with respect to $\tmlt$ (Lemma \ref{lem:nahmbasic}\ref{nahm1}) and the commutativity of $\tmlt$, $\re H(a, \tau\lin(e)b) = -\re B(a, e \tmlt b) = -\re B(a\tmlt b, e)$ is symmetric in $a$ and $b$, so $\tau\lin(e)$ is self-adjoint with respect to $\re H$.

By \eqref{nahmproduct} and $e \tmlt e = \tau e$, $e \tmlt \hat{e}_{1} = (0, [e_{3}, e_{1}], [e_{1}, e_{2}]) = \tfrac{1}{2}(0, \tau e_{2}, \tau e_{3})$, and applying $\tau$ gives $\tau\lin(e)\hat{e}_{1} = \tfrac{1}{2}(\hat{e}_{2} + \hat{e}_{3})$. Since $\lin(e)$ is complex linear and $\tau$ is conjugate-linear, $\tau\lin(e)J = -J\tau\lin(e)$. Together with the preceding this implies $\tau\lin(e)J\hat{e}_{1} = -\tfrac{1}{2}(J\hat{e}_{2} + J\hat{e}_{3})$. The same computations for $\hat{e}_{2}$ and $\hat{e}_{3}$ show that
\begin{align}\label{idem4}
&\tau\lin(e)\hat{e}_{i} = \tfrac{1}{2}(\hat{e}_{j} + \hat{e}_{k}), && \tau\lin(e)J\hat{e}_{i} = -\tfrac{1}{2}(J\hat{e}_{j} + J\hat{e}_{k})
\end{align}
for $\{i, j, k\}$ a permutation of $\{1, 2, 3\}$. From \eqref{idem4} it follows that $\tau\lin(e)\Usp_{e} \subset \Usp_{e}$ and, moreover, that the restriction of $\tau \lin(e)$ to $\Usp_{e}$ is diagonalizable with eigenvalues $\pm 1$, each with multiplicity $1$, and $\pm 1/2$, each with multiplicity $2$, so that $\tau\lin(e)\Usp_{e} = \Usp_{e}$.
Finally, $\tau\lin(e)\Uperp_{e} \subset \Uperp_{e}$ follows from $\tau \lin(e)\Usp_{e} = \Usp_{e}$ by the self-adjointness of $\tau\lin(e)$ and nondegeneracy of $\re H$.
\end{proof}

For a smooth function $f$ on a real vector space and a critical point $p$ of $f$, the Hessian $\hess f_{p}$ is the symmetric bilinear form with associated quadratic form $\hess f_{p}(b) = \tfrac{d^{2}}{d\varepsilon^{2}}\big|_{\varepsilon = 0}f(p + \varepsilon b)$; for any curve $\gamma$ with $\gamma(0) = p$, $(f \circ \gamma)''(0) = \hess f_{p}(\gamma'(0))$.

The Hessian of $W$ at $e$ vanishes in the directions $\Usp_{e}$ tangent to the $(\com^{\ast})^{3}$-orbit, and \eqref{eq:hessR} describes it on the complement $\Uperp_{e}$.
\begin{lemma}\label{thm:hessW}
Let $e \in \idem_{\tau}(\g^{3}, \tmlt)$, and let $\Usp_{e}$ and $\Uperp_{e}$ be as in \eqref{eq:UR}. 
For $b \in \g^{3}$, 
\begin{align}\label{eq:hessW}
\hess W_{e}(b) = 16W(e)^{2}\Big(\re H\big(b, (\tau\lin(e) - \tfrac{1}{2})b\big) + 2W(e)\big(|H(b, e)|^{2} - 4(\eta_{1}\eta_{2} + \eta_{2}\eta_{3} + \eta_{3}\eta_{1})\big)\Big),
\end{align}
where $\eta_{i} := \re H(b_{i}, e_{i})$ for $i \in \{1, 2, 3\}$.
In particular, for $b \in \Uperp_{e}$,
\begin{align}\label{eq:hessR}
\hess W_{e}(b) = 16W(e)^{2}\re H\big(b, (\tau\lin(e) - \tfrac{1}{2})b\big),
\end{align}
and
\begin{align}\label{eq:hessJ}
\hess W_{e}(b) + \hess W_{e}(Jb) = -16W(e)^{2}\nnorm{b}^{2}.
\end{align}
Moreover $\hess W_{e}(u, b) = 0$ for $u \in \Usp_{e}$ and $b \in \g^{3}$, so $\hess W_{e}(b) = \hess W_{e}(b_{\Uperp_{e}})$, where $b_{\Uperp_{e}}$ is the orthogonal projection of $b$ on $\Uperp_{e}$, and $\hess W_{e}(b_{\Uperp_{e}})$ is given by \eqref{eq:hessR}.
\end{lemma}
\begin{proof}
By \eqref{eq:idemnorms}, $C(e) = -1/(8W(e))$ and $\norm{e_{i}}^{-2} = 4W(e)$ for $i \in \{1, 2, 3\}$. Let $b \in \g^{3}$. From
\begin{align}
\norm{e_{i} + \varepsilon b_{i}}^{2} = \norm{e_{i}}^{2} + 2\varepsilon\re H(b_{i}, e_{i}) + \varepsilon^{2}\norm{b_{i}}^{2} = \norm{e_{i}}^{2}\big(1 + 4W(e)(2\varepsilon\eta_{i} + \varepsilon^{2}\norm{b_{i}}^{2})\big), 
\end{align}
and $\re H(b, e) = \eta_{1} + \eta_{2} + \eta_{3}$ there follows
\begin{align}\label{eq:numdenfull2}
\begin{aligned}
\prod_{i=1}^{3}\frac{\norm{e_{i} + \varepsilon b_{i}}^{2}}{\norm{e_{i}}^{2}} &= 1 + 8W(e)\varepsilon\re H(b, e) + \varepsilon^{2}\Big(4W(e)\nnorm{b}^{2} + 64W(e)^{2}(\eta_{1}\eta_{2} + \eta_{2}\eta_{3} + \eta_{3}\eta_{1})\Big) + O(\varepsilon^{3}).
\end{aligned}
\end{align}
By \eqref{eq:derivativesC}, 
\begin{align}\label{dceb}
&dC_{e}(b) = \tfrac{1}{2}B(e\tmlt e, b) = \tfrac{1}{2}B(\tau e, b) = -\tfrac{1}{2}H(b, e), &&d^{2}C_{e}(b, b) = B(b, e\tmlt b) = -H(b, \tau\lin(e)b). 
\end{align}
Since $C(e)= -1/(8W(e))$ is real it follows from \eqref{dceb} that 
\begin{align}
\begin{aligned}
C(e + \varepsilon b) &=C(e) +\varepsilon dC_{e}(b) + \tfrac{1}{2}\varepsilon^{2}d^{2}C_{e}(b, b) + O(\varepsilon^{3})\\
& = C(e) - \tfrac{1}{2}\varepsilon H(b, e) - \tfrac{1}{2}\varepsilon^{2}H(b, \tau\lin(e)b) + O(\varepsilon^{3})\\
&= C(e)\big(1 + 4W(e)\varepsilon H(b, e) + 4W(e)\varepsilon^{2}H(b, \tau\lin(e)b)\big) + O(\varepsilon^{3}),
\end{aligned}
\end{align}
which yields
\begin{align}\label{eq:numdenfull}
\begin{aligned}
\frac{|C(e + \varepsilon b)|^{2}}{C(e)^{2}} &= 1 + 8W(e)\varepsilon\re H(b, e) + \varepsilon^{2}\Big(16W(e)^{2}|H(b, e)|^{2} + 8W(e)\re H(b, \tau\lin(e)b)\Big) + O(\varepsilon^{3}).
\end{aligned}
\end{align}
Because the first order terms in \eqref{eq:numdenfull2} and \eqref{eq:numdenfull} coincide, the coefficient of $\varepsilon^{2}$ in the quotient $W(e + \varepsilon b)$ is $W(e)$ times the difference of the coefficients of $\varepsilon^{2}$, and $\hess W_{e}(b)$ is twice this coefficient, which gives \eqref{eq:hessW}. If $b \in \Uperp_{e}$, then $H(b, e) = 0$ and $\eta_{1} = \eta_{2} = \eta_{3} = 0$, so \eqref{eq:hessW} becomes \eqref{eq:hessR}. By Lemma \ref{hess2}, $J\Uperp_{e} = \Uperp_{e}$ and $\tau\lin(e)$ anticommutes with $J$, so $\re H(Jb, \tau\lin(e)Jb) = -\re H(b, \tau\lin(e)b)$ for $b \in \Uperp_{e}$, and adding \eqref{eq:hessR} for $b$ and for $Jb$ gives \eqref{eq:hessJ}.

An element $w \in \com^{3}$ of the Lie algebra of $(\com^{\ast})^{3}$ generates the linear vector field $a \mapsto w \cdot a = (w_{1}a_{1}, w_{2}a_{2}, w_{3}a_{3})$ on $\g^{3}$. Differentiating the identity $dW_{a}(w \cdot a) = 0$, which expresses the invariance of $W$, at $a = e$ in the direction $b$ and using $dW_{e} = 0$ gives $\hess W_{e}(w \cdot e, b) = 0$. The vectors $w \cdot e$, $w \in \com^{3}$, span $\Usp_{e}$, which proves the last claim.
\end{proof}

\subsection{Proof of Theorems \ref{thm:A} and \ref{thm:C}}\label{sec:sharp}
The proofs of Theorems \ref{thm:A} and \ref{thm:C} consist of a second order estimate (Lemma \ref{lem:secondorder}) and a rigidity argument (Lemma \ref{lem:rigidity}). Besides Lemma \ref{fkwmaxlemma} and the results of Sections \ref{sec:nahmprelim}--\ref{sec:hessian}, the proof of Theorem \ref{thm:A} uses only Lemma \ref{lem:Bh} (in the form $j \geq 1$), the identity \eqref{eq:sl2norm} of Lemma \ref{lem:sl2str}, and the lower bound \eqref{eq:lowerbound}, which rests on Lemmas \ref{lem:Bh} and \ref{lem:lower}\ref{lower2}. In particular it uses neither the integrality of the Dynkin index nor the equality case of Lemma \ref{lem:Bh}, which is needed only to identify the subalgebras of Dynkin index $1$ with the long-root subalgebras.

By Lemma \ref{fkwmaxlemma}\ref{fkw3} and Lemma \ref{thm:idem}\ref{idem1}, $W$ attains its maximum value $\comass{\om}^{2}$ at a $\tau$-twisted idempotent $e$, so that $W(e) = \comass{\om}^{2}$.

\begin{lemma}\label{lem:secondorder}
Let $e\in \g^{3}$ be a $\tau$-twisted idempotent at which $W$ assumes its maximum value $W(e) = \comass{\om}^{2}$.
\begin{enumerate}[label=(\arabic*), ref=(\arabic*)]
\item\label{so1} $\nnorm{\tau\lin(e)b} \leq \tfrac{1}{2}\nnorm{b}$ for all $b \in \Uperp_{e}$; that is, on the orthogonal complement of the $(\com^{\ast})^{3}$-orbit, the twisted multiplication operator of the maximizer $e$ is bounded by $1/2$.
\item\label{so2} For $i \in \{1, 2, 3\}$ and every $v \in \g$ with $H(v, e_{i}) = 0$,
\begin{align}\label{eq:clubs}
\sum_{a \in \{1, 2, 3\} \setminus \{i\}}\norm{[e_{a}, v]}^{2} = \nnorm{e \tmlt \iota_{i}(v)}^{2} \leq \tfrac{1}{4}\norm{v}^{2}.
\end{align}
\end{enumerate}
\end{lemma}
\begin{proof}
Since $e$ is a maximum point of $W$, the Hessian $\hess W_{e}$ is negative semidefinite, so by \eqref{eq:hessJ}, $-16W(e)^{2}\nnorm{b}^{2} \leq \hess W_{e}(b) \leq 0$ for $b \in \Uperp_{e}$. By \eqref{eq:hessR} this means $|\re H(b, \tau\lin(e)b)| \leq \tfrac{1}{2}\nnorm{b}^{2}$ on $\Uperp_{e}$, and, because $\tau\lin(e)$ is self-adjoint with respect to $\re H$ and preserves $\Uperp_{e}$, its norm on $\Uperp_{e}$ is at most $\tfrac{1}{2}$.
This shows \ref{so1}.

By \eqref{nahmproduct}, the $i$th component of $e \tmlt \iota_{i}(v)$ is $0$ and its other two components are $\pm[e_{a}, v]$, $a \in \{1, 2, 3\}\setminus\{i\}$, which gives the equality in \eqref{eq:clubs}. Since $H(v, e_{i}) = 0$, $\iota_{i}(v) \in \Uperp_{e}$, and since $\tau$ is an isometry, \ref{so1} gives $\nnorm{e \tmlt \iota_{i}(v)}^{2} = \nnorm{\tau\lin(e)\iota_{i}(v)}^{2} \leq \tfrac{1}{4}\norm{v}^{2}$.
\end{proof}

\begin{lemma}\label{lem:rigidity}
Let $e\in \g^{3}$ be a $\tau$-twisted idempotent at which $W$ assumes its maximum value. Then the span in $\g$ of the components of $e$, $\spn_{\com}\{e_{1}, e_{2}, e_{3}\}$ is a $\tau$-stable subalgebra of $\g$ isomorphic to $\sll(2, \com)$.
\end{lemma}
\begin{proof}
Write $f := \tau e$. By \eqref{eq:idemnorms} of Lemma \ref{thm:idem}, since $W(e) = \comass{\om}^{2}$,
\begin{align}\label{eq:efnorms}
H(e_{a}, e_{b}) = H(f_{a}, f_{b}) = \frac{\delta_{ab}}{4\comass{\om}^{2}}, \qquad a, b \in \{1, 2, 3\},
\end{align}
so that $\spn_{\com}\{e_{1}, e_{2}, e_{3}\}$ has complex dimension $3$. Because $e \tmlt e = \tau e$ and $\tau$ is an order $2$ automorphism of $(\g^{3}, \tmlt)$ (Lemma \ref{lem:nahmbasic}\ref{nahm3}),
\begin{align}\label{eq:cyclicrelations}
e \tmlt e = f, \qquad f \tmlt f = \tau(e \tmlt e) = e.
\end{align}
By Lemma \ref{lem:nahmbasic}\ref{nahm2} this means
\begin{align}\label{eq:cyclicrelations2}
&[e_{j}, e_{k}] = \tfrac{1}{2}f_{i}, && [f_{j}, f_{k}] = \tfrac{1}{2}e_{i},
\end{align}
for $(i, j, k)$ a cyclic permutation of $(1, 2, 3)$.
By \eqref{eq:cyclicrelations2} and the Jacobi identity,
\begin{align}\label{eq:cyclicrelations3}
\tfrac{1}{2}\sum_{a=1}^{3}[e_{a}, f_{a}] = [e_{1}, [e_{2}, e_{3}]] + [e_{2}, [e_{3}, e_{1}]] + [e_{3}, [e_{1}, e_{2}]] = 0.
\end{align}
Since $\tau e_{a} = f_{a}$ and $\tau f_{a} = e_{a}$, $\qad(e_{a}) = -\ad(f_{a})\ad(e_{a})$ and $\qad(f_{a}) = -\ad(e_{a})\ad(f_{a})$, so
\begin{align}
\qad(e_{a}) - \qad(f_{a}) = \ad(e_{a})\ad(f_{a}) - \ad(f_{a})\ad(e_{a}) = \ad([e_{a}, f_{a}]).
\end{align}
With \eqref{eq:cyclicrelations3} this yields
\begin{align}\label{eq:cyclicrelations4}
\sum_{a=1}^{3}\qad(e_{a}) = \sum_{a=1}^{3}\qad(f_{a}).
\end{align}
Since $[f_{a}, f_{m}] = \pm\tfrac{1}{2}e_{c}$ for $a \neq m$, where $c$ is the index in $\{1, 2, 3\}$ different from $a$ and $m$, by \eqref{eq:cyclicrelations2} and \eqref{eq:efnorms}, evaluating both sides of \eqref{eq:cyclicrelations4} on $f_{m}$ yields
\begin{align}\label{eq:sumrule}
\begin{aligned}
\sum_{a=1}^3\norm{[e_a,f_m]}^2&=\sum_{a=1}^3H(\qad(e_{a})f_{m}, f_{m}) = \sum_{a=1}^3H(\qad(f_{a})f_{m}, f_{m})= \sum_{a=1}^3\norm{[f_a,f_m]}^2= \frac{1}{8\comass{\om}^{2}}.
\end{aligned}
\end{align}
Fix $m$ and define
\begin{align}
v = f_{m} - 4\comass{\om}^{2}\sum_{a = 1}^{3}H(f_{m}, e_{a})e_{a},
\end{align}
so that, by \eqref{eq:efnorms}, $H(v, e_{a}) = 0$ for all $a \in \{1, 2, 3\}$ and $\norm{v}^{2} = \frac{1}{4\comass{\om}^{2}} - 4\comass{\om}^{2}\sum_{a=1}^{3}|H(f_{m}, e_{a})|^{2}$.

For $a \in \{1, 2, 3\}$ there holds
\begin{align}
[e_{a}, f_{m}] = [e_{a}, v] + 4\comass{\om}^{2}\sum_{b \in \{1, 2, 3\}\setminus\{a\}}H(f_{m}, e_{b})[e_{a}, e_{b}],
\end{align}
and by \eqref{eq:cyclicrelations2}, for $b \neq a$, $[e_{a}, e_{b}] = \pm \tfrac{1}{2}f_{c}$, where $c$ is the index with $\{a, b, c\} = \{1, 2, 3\}$; as $b$ runs over the two indices different from $a$, so does $c$. The two summands are orthogonal, because by \eqref{basic4} and \eqref{eq:cyclicrelations2}, $H([e_{a}, v], f_{c}) = -H(v, [f_{a}, f_{c}]) = \mp \tfrac{1}{2}H(v, e_{b}) = 0$. Since $f_{1}, f_{2}, f_{3}$ are orthogonal of squared norm $1/(4\comass{\om}^{2})$ by \eqref{eq:efnorms},
\begin{align}\label{eq:rigidsplit}
\norm{[e_{a}, f_{m}]}^{2} = \norm{[e_{a}, v]}^{2} + \comass{\om}^{2}\sum_{b \in \{1, 2, 3\}\setminus\{a\}}|H(f_{m}, e_{b})|^{2}, \qquad a \in \{1, 2, 3\}.
\end{align}
Summing \eqref{eq:rigidsplit} over $a$ and using \eqref{eq:sumrule} gives
\begin{align}\label{rsp1}
\sum_{a=1}^{3}\norm{[e_{a}, v]}^{2} = \frac{1}{8\comass{\om}^{2}} - 2\comass{\om}^{2}\sum_{a = 1}^{3}|H(f_{m}, e_{a})|^2 = \tfrac{1}{2}\norm{v}^{2}.
\end{align}
On the other hand, summing \eqref{eq:clubs} over $i \in \{1, 2, 3\}$, in which each $a$ occurs twice, gives
\begin{align}\label{rsp2}
2\sum_{a=1}^{3}\norm{[e_{a}, v]}^{2} \leq \tfrac{3}{4}\norm{v}^{2}.
\end{align}
Combining \eqref{rsp1} and \eqref{rsp2} gives
\begin{align}
\norm{v}^{2} \leq \tfrac{3}{4}\norm{v}^{2},
\end{align}
so $v = 0$ and $f_{m} \in \spn_{\com}\{e_{1}, e_{2}, e_{3}\}$. Hence $\spn_{\com}\{e_{1}, e_{2}, e_{3}\}$ is $\tau$-stable, and by \eqref{eq:cyclicrelations2} it contains the brackets $[e_{j}, e_{k}] = \tfrac{1}{2}f_{i}$. It follows that the Lie subalgebra $\subs_{e}$ of $\g$ generated by the components of $e$ equals $\spn_{\com}\{e_{1}, e_{2}, e_{3}\}$ and is $\tau$-stable; it is semisimple by Lemma \ref{thm:idem}\ref{idem3}. Being semisimple of dimension $3$, $\subs_{e}$ is isomorphic to $\sll(2, \com)$.
\end{proof}

\begin{proof}[Proof of Theorems \ref{thm:A} and \ref{thm:C}]
Let $e$ be a $\tau$-twisted idempotent at which $W$ attains its maximum, and let $j(\subs_{e})$ be the Dynkin index of the subalgebra $\subs_{e} = \spn_{\com}\{e_{1}, e_{2}, e_{3}\}$ of Lemma \ref{lem:rigidity}. By \eqref{eq:idemnorms}, $e_{1}, e_{2} \in \subs_{e}$ are nonzero and $H$-orthogonal, and $F(e_{1}, e_{2}) = W(e) = \comass{\om}^{2}$, while by \eqref{eq:sl2norm} applied to $e_{1}, e_{2}$, $F(e_{1}, e_{2}) = 1/(j(\subs_{e})\hv)$, so that $1/(j(\subs_{e})\hv) = \comass{\om}^{2}$. By Lemma \ref{lem:Bh}, $j(\subs_{e}) \geq 1$, while by  \eqref{eq:lowerbound}, $\frac{1}{\hv} \leq\comass{\om}^{2}$, so that
\begin{align}\label{eq:sharpchain}
\frac{1}{\hv} \leq \comass{\om}^{2} = \frac{1}{j(\subs_{e})\hv} \leq \frac{1}{\hv}.
\end{align}
By \eqref{eq:sharpchain}, $\comass{\om}^{2} = 1/\hv$ and $j(\subs_{e}) = 1$. This proves \eqref{eq:main}, by Lemma \ref{fkwmaxlemma}, and its optimality. It also shows that for every $\tau$-twisted idempotent $e$ at which $W$ attains its maximum, the components of $e$ span a $\tau$-stable subalgebra isomorphic to $\sll(2, \com)$ and having Dynkin index $1$.

Theorem \ref{thm:C}\ref{thmC1} is Lemma \ref{thm:idem}\ref{idem2},\ref{idem1}. For Theorem \ref{thm:C}\ref{thmC2}, a nonzero $\tau$-twisted idempotent $e$ satisfies $3/(4H(e, e)) = W(e) \leq \comass{\om}^{2} = 1/\hv$ by \eqref{eq:idemnorms}, with equality if and only if $W$ attains its maximum at $e$, in which case the components of $e$ span a $\tau$-stable subalgebra of Dynkin index $1$, as just shown. Conversely, if the components of $e$ span such a subalgebra $\subs$, then by \eqref{eq:idemnorms} and \eqref{eq:sl2norm} applied to the orthogonal pair $e_{1}, e_{2}$, $W(e) = F(e_{1}, e_{2}) = 1/\hv$.

For the equality case of Theorem \ref{thm:A}, let $p, q \neq 0$ satisfy $F(p, q) = 1/\hv$. Then $[p, q] \neq 0$ and $W(p, q, \tau[p, q]) = F(p, q) = \comass{\om}^{2}$ by Lemma \ref{fkwmaxlemma}\ref{fkw2}, so $(p, q, \tau[p, q]) \in \crtnt(W)$. By Lemma \ref{thm:idem}\ref{idem1}, there is $c \in (\com^{\ast})^{3}$ such that the rescaling $e = c \cdot (p, q, \tau[p, q])$ is a $\tau$-twisted idempotent at which $W$ attains its maximum, so the components of $e$, which are multiples of $p$, $q$, and $\tau[p, q]$, are pairwise orthogonal by \eqref{eq:idemnorms} and span a $\tau$-stable subalgebra $\subs$ of Dynkin index $1$; hence $H(p, q) = 0$ and $p, q \in \subs$. Conversely, if $H(p, q) = 0$ and $p, q$ lie in such a subalgebra, then $F(p, q) = 1/\hv$ by \eqref{eq:sl2norm}.
\end{proof}

By Lemma \ref{fkwmaxlemma}\ref{fkw1}, the value of $\sect$ on a plane equals the value of $F$ on an orthogonal basis of the plane. Hence Theorem \ref{thm:A} and \eqref{eq:sectconstant} give the following description of the maximum of $\sect$.
\begin{corollary}\label{cor:sectmax}
For all $\sigma \in \grass(2, \g)$, $\sect(\sigma)  \leq 1/\hv$, with equality if and only if $\sigma \subset \subs$ for a $\tau$-stable subalgebra $\subs \simeq \sll(2, \com)$ of Dynkin index $1$.
\end{corollary}

For pairs in the compact real form, ``contained in'' in Theorem \ref{thm:A} can be replaced by ``generate''.
\begin{corollary}\label{cor:compactequality}
For $x, y \in \gz\setminus\{0\}$, $F(x, y) = 1/\hv$ if and only if $H(x, y) = 0$ and $x$ and $y$ generate a subalgebra $\subs_{0} \subset \gz$ isomorphic to $\su(2)$ whose complexification has Dynkin index $1$.
\end{corollary}
\begin{proof}
If $F(x, y) = 1/\hv$, Theorem \ref{thm:A} gives $H(x, y) = 0$ and $x, y \in \subs$ for a $\tau$-stable $\subs$ of Dynkin index $1$, so $x, y \in \subs_{0} = \subs \cap \gz \simeq \su(2)$. In the basis of Lemma \ref{lem:sl2str}\ref{sl2a}, $x$ and $y$ have real, nonzero, orthogonal coefficient vectors, so $x$, $y$, $[x, y]$ is a basis of $\subs_{0}$, and $x, y$ generate $\subs_{0}$. The converse follows from \eqref{eq:sl2norm}.
\end{proof}

\subsection{Hermitian metrized commutative algebras}\label{sec:hermitianmetrized}
This subsection situates Lemmas \ref{thm:idem} and \ref{thm:hessW} in the general context of metrized commutative algebras; it is not used elsewhere in the paper.

For a commutative real algebra $(\alg, \mlt)$ with an invariant inner product $\langle\dum, \dum\rangle$, the critical points of the cubic form $\langle x \mlt x, x\rangle$ on the unit sphere correspond to idempotents, and the Hessian at an idempotent $e$ is a multiple of $\lin(e) - \tfrac{1}{2}$ on the tangent space to the sphere; see \cite[Lemma $2.3$]{Tkachev-correction} and \cite{Fox-trace, Fox-Tkachev, Tkachev-universality}. As is explained below, Lemmas \ref{thm:idem} and \ref{thm:hessW} are the specialization to Nahm algebras of Hermitian analogues of these observations, in which idempotents are replaced by $\tau$-twisted idempotents, $\lin(e)$ is replaced by the $\tau$-twisted operator $\tau\lin(e)$, and the sphere replaced by the complement of the $(\com^{\ast})^{3}$-orbit. 

The new feature is that $\tau\lin(e)$ anticommutes with $J$, so its spectrum on $\Uperp_{e}$ is symmetric about $0$, and the one-sided bound $\tau\lin(e) \leq \tfrac{1}{2}$ that holds on $\Uperp_{e}$ at a maximum of $W$ implies the two-sided bound $\nnorm{\tau\lin(e)b} \leq \tfrac{1}{2}\nnorm{b}$. Equivalently, the identity \eqref{eq:hessJ} of Lemma \ref{thm:hessW} holds at every $\tau$-twisted idempotent $e$, whether or not $e$ is a maximum point, and at a maximum point both summands in \eqref{eq:hessJ} are nonpositive, so each lies between $-16W(e)^{2}\nnorm{b}^{2}$ and $0$. This two-sided bound is the key step in the proof of Theorem \ref{thm:A}, and it has a conceptual interpretation. The cubic form $C$ is a section of the line bundle $\mathcal{O}(1, 1, 1) = \bigotimes_{i=1}^{3}\pi_{i}^{\ast}\mathcal{O}(1)$ on $\proj(\g)^{3}$, where $\pi_{i}$ is the projection on the $i$th factor, and $W$ is its squared norm for the Hermitian metric induced by $H$. Since $\log|C|^{2}$ is pluriharmonic where $C \neq 0$, on this open set $i\partial\bar\partial(-\log W) = \sum_{i=1}^{3}\pi_{i}^{\ast}\omega_{\mathrm{FS}}$, where $\omega_{\mathrm{FS}}$ is the Fubini--Study form of $(\g, H)$, so $-\log W$ is strictly plurisubharmonic. At a critical point $e$ with $W(e) \neq 0$, $\hess(-\log W)_{e}(b) + \hess(-\log W)_{e}(Jb)$ is a positive multiple of $(i\partial\bar\partial(-\log W))(b, Jb) > 0$, and this is the identity \eqref{eq:hessJ}: the two-sided bound is the positivity of the curvature of $\mathcal{O}(1, 1, 1)$.

The function $W$ is natural from the point of view of giving a Lie theoretic interpretation of the Böttcher--Wenzel inequality, but from the point of view of commutative nonassociative algebras such as the Nahm algebra, a different but closely related function $V$ is more natural and makes sense in greater generality. Although not strictly necessary for this paper, this is described briefly here because it may be of interest in other contexts.

Let $(\alg, \mlt)$ be a finite-dimensional commutative complex algebra, let $B$ be a nondegenerate symmetric complex bilinear form on $\alg$ that is invariant, $B(x \mlt y, z) = B(x, y \mlt z)$ for $x, y, z \in \alg$, and let $\tau$ be a conjugation of $(\alg, \mlt)$ (Definition \ref{def:twisted}) such that $B(\tau x, \tau y) = \overline{B(x, y)}$ and $H(x, y) := -B(x, \tau y)$ is positive definite. Call $(\alg, \mlt, B, \tau)$ a \emph{Hermitian metrized commutative algebra}, write $\nnorm{x}^{2} = H(x, x)$, and define
\begin{align}\label{eq:VA}
&C_{\alg}(a) := B(a \mlt a, a),& &V_{\alg}(a) := \frac{|C_{\alg}(a)|}{\nnorm{a}^{3}},& &a \in \alg \setminus \{0\}.&
\end{align}
Suppose $C_{\alg}$ is not identically zero; by polarization and the nondegeneracy of $B$, this means $\alg \mlt \alg \neq \{0\}$. Since $d(C_{\alg})_{a}(b) = 3B(a \mlt a, b) = -3H(b, \tau(a \mlt a))$, the computation in the proof of Lemma \ref{lem:critW} shows that an element $a$ with $C_{\alg}(a) \neq 0$ is a critical point of $V_{\alg}$ if and only if
\begin{align}\label{eq:VAcrit}
a \mlt a = -\frac{C_{\alg}(a)}{\nnorm{a}^{2}}\,\tau a.
\end{align}
Conversely, if $a \mlt a \in \com^{\ast}\tau a$, then pairing with $a$ using $B$ shows that the scalar is $-C_{\alg}(a)/\nnorm{a}^{2}$, so $a$ is critical. Since $(za) \mlt (za) = z^{2}(a \mlt a)$ and $\tau(za) = \bar{z}\tau a$, a suitable choice of $z \in \com^{\ast}$ makes $za$ a $\tau$-twisted idempotent, so the critical points of $V_{\alg}$ at which $V_{\alg}$ is nonzero are exactly the elements of $\com^{\ast}\cdot\idem_{\tau}(\alg)$. In contrast with Lemma \ref{thm:idem}\ref{idem1}, only a single complex scalar is involved. For $e \in \idem_{\tau}(\alg)$, $C_{\alg}(e) = B(\tau e, e) = -\nnorm{e}^{2}$, so $V_{\alg}(e) = 1/\nnorm{e}$, and the maximum value of $V_{\alg}$ is $1/\min\{\nnorm{e}: e \in \idem_{\tau}(\alg)\}$. Moreover, for $e \in \idem_{\tau}(\alg)$ and $b \in \alg$ with $H(b, e) = 0$, the expansion used in the proof of Lemma \ref{thm:hessW} gives
\begin{align}\label{eq:hessVA}
\hess(V_{\alg}^{2})_{e}(b) = \frac{12}{\nnorm{e}^{4}}\re H\big(b, (\tau\lin(e) - \tfrac{1}{2})b\big),
\end{align}
where $\lin(e)b := e \mlt b$. The operator $\tau\lin(e)$ is self-adjoint with respect to $\re H$, anticommutes with multiplication by $i$, has $e$ and $ie$ as eigenvectors with eigenvalues $1$ and $-1$, and therefore preserves the complex orthogonal complement of $e$. Hence the argument of Lemma \ref{lem:secondorder}\ref{so1} applies verbatim to show: if $V_{\alg}$ attains its maximum at $e \in \idem_{\tau}(\alg)$, then $\nnorm{\tau\lin(e)b} \leq \tfrac{1}{2}\nnorm{b}$ for all $b$ with $H(b, e) = 0$. By \eqref{eq:hessVA}, this is the Hermitian counterpart of the situation for metrized commutative algebras recalled at the beginning of this subsection: the Hessian at a $\tau$-twisted idempotent is governed by $\tau\lin(e) - \tfrac{1}{2}$. The new feature is the two-sided bound forced by the anticommutation with $i$.

The Nahm algebra $(\g^{3}, \tmlt, B, \tau)$ has a special feature not possessed by a general Hermitian metrized commutative algebra, which is that it decomposes as a direct sum of three equidimensional mutually orthogonal subalgebras each of which squares to $\{0\}$, namely the copies of $\g$ in the direct product $\g \times \g \times \g$. This decomposition is what makes the function $W$ meaningful and what gives rise to the three independent rescalings. 

For the Nahm algebra, $C_{\g^{3}} = 6C$ by \eqref{eq:cubicpolar}, and the arithmetic-geometric mean inequality applied to the norms of the components gives
\begin{align}\label{eq:VWcompare}
V_{\g^{3}}(a)^{2} = 36\,W(a)\,\frac{\prod_{i=1}^{3}\norm{a_{i}}^{2}}{\nnorm{a}^{6}} \leq \tfrac{4}{3}W(a),
\end{align}
with equality if and only if $\norm{a_{1}} = \norm{a_{2}} = \norm{a_{3}}$. The components of a $\tau$-twisted idempotent have equal norms by \eqref{eq:idemnorms}, and the elements $\zeta \in (\com^{\ast})^{3}$ with $|\zeta_{1}| = |\zeta_{2}| = |\zeta_{3}| = 1$ and $\zeta_{1}\zeta_{2}\zeta_{3} = 1$ map $\idem_{\tau}(\g^{3}, \tmlt)$ to itself. Hence the critical points of $V_{\g^{3}}$ at which it is nonzero are the critical points of $W$ whose components have equal norms, and, by \eqref{eq:VWcompare}, the maximum value of $V_{\g^{3}}^{2}$ is $\tfrac{4}{3}\comass{\om}^{2} = 4/(3\hv)$. On $\Uperp_{e}$ the form \eqref{eq:hessVA} is a positive multiple of \eqref{eq:hessR}. On the remaining four directions $\Usp_{e} \cap \{b: H(b, e) = 0\}$, $\tau\lin(e)$ has eigenvalues $\pm\tfrac{1}{2}$, each with multiplicity $2$ (Lemma \ref{hess2}): the eigenvectors with eigenvalue $\tfrac{1}{2}$ are null directions of \eqref{eq:hessVA}, tangent to the action of the phases $\zeta$ just described, and those with eigenvalue $-\tfrac{1}{2}$ are negative directions, along which the components are rescaled by positive reals with product $1$, so that $W$ is constant and $V_{\g^{3}}$ decreases by \eqref{eq:VWcompare}. Thus $W$ is the function adapted to the $(\com^{\ast})^{3}$-symmetry of the commutator problem, while $V_{\g^{3}}$ is the function intrinsic to the Hermitian metrized commutative algebra $(\g^{3}, \tmlt, B, \tau)$.

\section{Invariant norms, representations, matrix inequalities, and curvature}\label{sec:reductive}

This section shows how the results here recover the Böttcher--Wenzel and Bloch--Iserles inequalities, and interprets Theorem \ref{thm:A} in terms of the complex sectional curvature of compact simple Lie groups and compact symmetric spaces.

\subsection{Invariant Hermitian forms}
Let $\g$ be a complex reductive Lie algebra, $\g = \mathfrak{z}\oplus\subs_{1}\oplus\dots\oplus\subs_{r}$ with $\mathfrak{z}$ the center and $\subs_{k}$ simple ideals. A \emph{compact real form} of $\g$ is a real subalgebra $\gz$ with $\g = \gz \oplus i\gz$ admitting an $\ad(\gz)$-invariant inner product; its compact conjugation $\tau$ is defined as before. A Hermitian form $\HH$ on $\g$ is \emph{$\gz$-invariant} if
\begin{align}\label{eq:invariance}
\HH([z, x], y) + \HH(x, [z, y]) = 0 \qquad \text{for all } z \in \gz \text{ and } x, y \in \g.
\end{align}
Every $\Ad(K)$-invariant Hermitian form, for $K$ a connected compact group with Lie algebra $\gz$, is $\gz$-invariant. For $\HH$ positive definite write $F_{\HH}(x, y) := \HH([x, y], [x, y])/(\HH(x, x)\HH(y, y))$.

\begin{lemma}\label{lem:schur}
\leavevmode\noindent
\begin{enumerate}[label=(\arabic*), ref=(\arabic*)]
\item\label{schur1} If $\g$ is simple, every positive definite $\gz$-invariant Hermitian form is $\HH = \kappa H$ for a unique $\kappa > 0$, and then $F_{\HH} = \kappa^{-1}F$.
\item\label{schur2} If $\g$ is reductive, the ideals $\mathfrak{z}, \subs_{1}, \dots, \subs_{r}$ are $\tau$-stable and pairwise $\HH$-orthogonal for every positive definite $\gz$-invariant $\HH$, $\kl_{k} := \subs_{k} \cap \gz$ is a compact real form of $\subs_{k}$, and $\HH|_{\subs_{k}} = \kappa_{k}H_{k}$ with $\kappa_{k} > 0$, where $H_{k}$ is the form \eqref{eq:Hdef} of $\subs_{k}$.
\end{enumerate}
\end{lemma}
\begin{proof}
\ref{schur1} Write $\HH(x, y) = H(Sx, y)$ with $S$ positive definite and $H$-Hermitian. By \eqref{eq:invariance} for $\HH$ and for $H$ (\eqref{basic4}), $S$ commutes with $\ad(z)$ for $z \in \gz$, hence, being complex linear, with $\ad(\g)$. Its eigenspaces are then ideals, so $S = \kappa\id_{\g}$.

\ref{schur2} A compact Lie algebra is the direct sum of its center and compact simple ideals, whose complexifications are simple, so by uniqueness of the decomposition of $\g$ these complexifications are the $\subs_{k}$, which are therefore $\tau$-stable. Because $\kl_{k} = [\kl_{k}, \kl_{k}]$, every $x \in \subs_{k}$ is a sum of brackets $[z, x']$ with $z \in \kl_{k}$, $x' \in \subs_{k}$, and \eqref{eq:invariance} gives $\HH(x, y) = 0$ whenever $y$ lies in $\mathfrak{z}$ or in $\subs_{l}$ with $l \neq k$. The last claim is \ref{schur1} applied to $\subs_{k}$.
\end{proof}

\begin{corollary}\label{thm:reductive}
Let $\g$ be complex reductive and non-abelian, with compact real form $\gz$, and let $\HH$ be a positive definite $\gz$-invariant Hermitian form, with $\kappa_{k}$ as in Lemma \ref{lem:schur}\ref{schur2}. Let $h_{k}^{\vee}$ be the dual Coxeter number of $\subs_{k}$ and $\lambda^{\ast} := \max_{1 \leq k \leq r}(\kappa_{k}h_{k}^{\vee})^{-1}$. Then
\begin{align}\label{eq:redbound}
\HH([x, y], [x, y]) \leq \lambda^{\ast}\,\HH(x, x)\,\HH(y, y) \qquad \text{for all } x, y \in \g,
\end{align}
the constant is optimal, and for $x, y \neq 0$ equality holds if and only if $x$ and $y$ lie in a single ideal $\subs_{k}$ with $(\kappa_{k}h_{k}^{\vee})^{-1} = \lambda^{\ast}$ and satisfy there the equality condition of Theorem \ref{thm:A}.
\end{corollary}
\begin{proof}
Write $x = x_{0} + \sum_{k=1}^{r}x_{k}$ and $y = y_{0} + \sum_{k=1}^{r}y_{k}$ according to the decomposition of $\g$, and put $a_{k} = \HH(x_{k}, x_{k})$ and $b_{k} = \HH(y_{k}, y_{k})$ for $0 \leq k \leq r$. Then $[x, y] = \sum_{k=1}^{r}[x_{k}, y_{k}]$, and by Lemma \ref{lem:schur} and Theorem \ref{thm:A} applied in each $\subs_{k}$,
\begin{align}\label{eq:chain}
\HH([x, y], [x, y]) = \sum_{k=1}^{r}\HH([x_{k}, y_{k}], [x_{k}, y_{k}]) \leq \sum_{k=1}^{r}\frac{a_{k}b_{k}}{\kappa_{k}h_{k}^{\vee}} \leq \lambda^{\ast}\sum_{k=1}^{r}a_{k}b_{k} \leq \lambda^{\ast}\Big(\sum_{k=0}^{r}a_{k}\Big)\Big(\sum_{k=0}^{r}b_{k}\Big),
\end{align}
which is \eqref{eq:redbound}. Equality is attained by an extremal pair of Theorem \ref{thm:A} in an ideal realizing $\lambda^{\ast}$. If equality holds with $x, y \neq 0$, then some $a_{k}b_{k} > 0$, and equality in the last step of \eqref{eq:chain} forces all the other components of $x$ and $y$ to vanish; the remaining equalities then give the stated condition.
\end{proof}

\subsection{Representations}
Let $\g$ be simple and $\rho: \g \to \gl(V)$ a nonzero representation, with Dynkin index $\ell_{\rho}$ (Definition \ref{def:dynkin}), so that \eqref{eq:repindex} holds by Lemma \ref{lem:dynkinformula}\ref{dynkinformula1}. The Dynkin index is additive under direct sums \cite[(2.24)]{Dynkin-semisimplesubalgebras}, and $\ell_{\mathrm{ad}} = 2\hv$ by \eqref{eq:Bnorm}. Let $\Lambda(\rho) \subset \h^{\ast}$ be the set of weights of $\rho$ and $m_{\mu}$ the multiplicity of $\mu \in \Lambda(\rho)$. Evaluating \eqref{eq:repindex} on $x = y = \theta^{\vee}$ and using $(\theta^{\vee}, \theta^{\vee}) = 2$ gives
\begin{align}\label{eq:indexweights}
\ell_{\rho} = \frac{1}{2}\sum_{\mu \in \Lambda(\rho)}m_{\mu}\,\mu(\theta^{\vee})^{2}.
\end{align}

\begin{proof}[Proof of Theorem \ref{thm:B}]
Since the Hermitian inner product on $V$ is $\gz$-invariant, $\rho(\gz)$ consists of skew-Hermitian operators, so $\rho(x)^{\dagger} = -\rho(\tau x)$ for $x \in \g$. By \eqref{eq:repindex}, \eqref{eq:Bnorm}, and \eqref{eq:Hdef},
\begin{align}\label{eq:frobeniusH}
\tr \rho(x)\rho(y)^{\dagger} = -\ell_{\rho}(x, \tau y) = \frac{\ell_{\rho}}{2\hv}H(x, y).
\end{align}
By \eqref{eq:frobeniusH},
\begin{align}\label{eq:fh1}
&\fnorm{\rho(x)}^{2} = \tr \rho(x)\rho(x)^{\dagger} = \frac{\ell_{\rho}}{2\hv}\norm{x}^{2},&&\fnorm{\rho([x, y])}^{2} = \tr \rho([x, y])\rho([x, y])^{\dagger} = \frac{\ell_{\rho}}{2\hv}\norm{[x, y]}^{2}.& 
\end{align}
By \eqref{eq:fh1}, since $\rho$ is injective and $[\rho(x), \rho(y)] = \rho([x, y])$, for $x, y \in \g\setminus\{0\}$
\begin{align}\label{eq:frobeniusratio}
\begin{aligned}
\frac{\fnorm{[\rho(x), \rho(y)]}^{2}}{\fnorm{\rho(x)}^{2}\fnorm{\rho(y)}^{2}} = \frac{2\hv}{\ell_{\rho}}F(x, y),
\end{aligned}
\end{align}
and Theorem \ref{thm:B}, with the optimal constant $(2\hv/\ell_{\rho})(1/\hv) = 2/\ell_{\rho}$ and the stated equality case, follows from \eqref{eq:frobeniusratio} and Theorem \ref{thm:A}.
\end{proof}

Theorem \ref{thm:B} depends on $\rho$ only through $\ell_{\rho}$. Because $\rho(\g) \subset \gl(V)$, the classical inequality \eqref{eq:bwclassical} always gives the constant $2$, so Theorem \ref{thm:B} improves \eqref{eq:bwclassical} on $\rho(\g)$ exactly when $\ell_{\rho} > 1$. By Lemma \ref{lem:indexone} below, the only representations of Dynkin index $1$ are, up to trivial summands, the defining representations of $\sll(n, \com)$ and $\sp(2n, \com)$ and their duals (recall $\sll(2, \com) = \sp(2, \com)$, $\so(5, \com) \simeq \sp(4, \com)$ and $\so(6, \com) \simeq \sll(4, \com)$). More useful is the comparison with the classical subalgebra containing $\rho(\g)$: if $\rho$ is irreducible and orthogonal, then $\rho(\g) \subset \so(V)$, and if it is symplectic, $\rho(\g) \subset \sp(V)$; the optimal constants on these are $1$ (if $\dim V \geq 4$) and $2$ (Corollary \ref{cor:matrix} below). Table \ref{tab:improvement} lists, for each simple type, the dual Coxeter number, the representation $V$ of smallest dimension, its Dynkin index computed from \eqref{eq:indexweights}, the optimal constant $2/\ell_{V}$ of \eqref{eq:repbound}, and the optimal constant for the classical Lie algebra $\sll(V)$, $\so(V)$, or $\sp(V)$ containing $\rho(\g)$. For the orthogonal representations of $B_{n}$, $D_{n}$, and $G_{2}$ in Table \ref{tab:improvement}, Theorem \ref{thm:B} recovers the Bloch--Iserles constant $1$, whereas for $F_{4}$, $E_{6}$, $E_{7}$, and $E_{8}$ it improves the constant of the containing classical Lie algebra by the factors $3$, $6$, $12$, and $30$, respectively. The constant $1/\hv$ of Theorem \ref{thm:A} is the case of the adjoint representation, which for $E_{8}$ is the smallest one.

\begin{lemma}\label{lem:indexone}
A nonzero representation $\rho$ of a complex simple Lie algebra $\g$ has Dynkin index $\ell_{\rho} = 1$ if and only if, up to trivial summands, $\rho$ is the defining representation of $\sll(n, \com)$, $n \geq 2$, or its dual, or the defining representation of $\sp(2n, \com)$, $n \geq 1$.
\end{lemma}
\begin{proof}
The index is additive under direct sums, takes positive integer values on nontrivial representations \cite[Theorem~2.2]{Dynkin-semisimplesubalgebras}, and vanishes on trivial ones. Hence a representation of index $1$ is the sum of an irreducible representation $V(\lambda)$ of index $1$ and trivial summands. By \cite[Theorem~2.5]{Dynkin-semisimplesubalgebras}, $\ell_{V(\lambda)} = \dim V(\lambda)\,(\lambda, \lambda + 2\delta)/\dim \g$, where $2\delta = \sum_{\al \in \Phi^{+}}\al$. Let $\varpi_{1}, \dots, \varpi_{r}$ be the fundamental weights. If $\lambda = \lambda' + \varpi_{j}$ with $\lambda'$ dominant and nonzero, then $\dim V(\lambda) > \dim V(\varpi_{j})$ by the Weyl dimension formula, and $(\lambda, \lambda + 2\delta) - (\varpi_{j}, \varpi_{j} + 2\delta) = (\lambda', \lambda' + 2\varpi_{j} + 2\delta) > 0$, so $\ell_{V(\lambda)} > \ell_{V(\varpi_{j})} \geq 1$. Hence $\lambda$ is a fundamental weight with $\ell_{V(\lambda)} = 1$. By \cite[Proposition~2.6]{Laszlo-Sorger}, the index of every representation is divisible by $2$, $6$, $12$, or $60$ unless $\g$ has type $A_{r}$ or $C_{r}$, and the fundamental weights of index $1$ are $\varpi_{1}$ and $\varpi_{r}$ for $A_{r}$ and $\varpi_{1}$ for $C_{r}$.
\end{proof}

\begin{table}[h]
\centering
\renewcommand{\arraystretch}{1.3}
\begin{tabular}{cccccccc}
\hline
$\g$ & $\dim\g$ & $\hv$ & $V$ & type of $V$ & $\ell_{V}$ & $2/\ell_{V}$ & classical constant\\
\hline
$A_n$ $(n\geq 1)$ & $n(n+2)$ & $n+1$ & $\com^{n+1}$ & complex & $1$ & $2$ & $2$ \\
$B_n$ $(n\geq 3)$ & $n(2n+1)$ & $2n-1$ & $\com^{2n+1}$ & orthogonal & $2$ & $1$ & $1$ \\
$C_n$ $(n\geq 2)$ & $n(2n+1)$ & $n+1$ & $\com^{2n}$ & symplectic & $1$ & $2$ & $2$ \\
$D_n$ $(n\geq 4)$ & $n(2n-1)$ & $2n-2$ & $\com^{2n}$ & orthogonal & $2$ & $1$ & $1$ \\
$G_2$ & $14$ & $4$ & $\com^{7}$ & orthogonal & $2$ & $1$ & $1$ \\
$F_4$ & $52$ & $9$ & $\com^{26}$ & orthogonal & $6$ & $\tfrac{1}{3}$ & $1$ \\
$E_6$ & $78$ & $12$ & $\com^{27}$ & complex & $6$ & $\tfrac{1}{3}$ & $2$ \\
$E_7$ & $133$ & $18$ & $\com^{56}$ & symplectic & $12$ & $\tfrac{1}{6}$ & $2$ \\
$E_8$ & $248$ & $30$ & $\com^{248}$ & orthogonal & $60$ & $\tfrac{1}{30}$ & $1$ \\
\hline
\end{tabular}
\caption{Dual Coxeter numbers, the smallest representation $V$, its Dynkin index $\ell_{V}$, the optimal constant $2/\ell_{V}$ in \eqref{eq:repbound} for the Frobenius norm on $\rho(\g) \subset \gl(V)$, and the optimal constant for the classical Lie algebra $\sll(V)$, $\so(V)$, or $\sp(V)$ containing $\rho(\g)$ (Corollary \ref{cor:matrix}).}
\label{tab:improvement}
\end{table}

\subsection{The classical matrix inequalities}\label{sec:classical}
For the classical Lie algebras in their defining representations, with $\tau X = -X^{\dagger}$, the Killing forms are
\begin{align}\label{eq:killingtrace}
&B_{\sll(n, \com)}(X, Y) = 2n\tr XY, \qquad B_{\sp(2n, \com)}(X, Y) = (2n+2)\tr XY,&
&B_{\so(n, \com)}(X, Y) = (n-2)\tr XY,
\end{align}
where $\sp(2n, \com)$ is the Lie algebra of $X$ with $X^{t}\Omega + \Omega X = 0$ for $\Omega = \left(\begin{smallmatrix}0 & I_{n}\\ -I_{n} & 0\end{smallmatrix}\right)$ and $\so(n, \com)$ that of the complex skew-symmetric matrices; both are closed under $X \mapsto X^{\dagger}$. By \eqref{eq:killingtrace} and \eqref{eq:Bnorm}, with $\hv$ equal to $n$, $n+1$, $n-2$ respectively, the defining representations have Dynkin indices $1$, $1$, and $2$ (for $n \geq 5$), while for $\so(3, \com) \simeq \sll(2, \com)$ the index is $4$, and for $\so(4, \com) \simeq \sll(2, \com) \oplus \sll(2, \com)$ the restriction of $\com^{4}$ to each ideal is two copies of the defining representation of $\sll(2, \com)$, of index $2$.

\begin{corollary}\label{cor:matrix}
Let $\fnorm{X}^{2} = \tr XX^{\dagger}$ be the Frobenius norm of the complex matrix $X$, and let $e_{1}, \dots, e_{n}$ be the standard basis of $\com^{n}$. In each item, equality is characterized for nonzero $X$ and $Y$ (respectively $A$ and $B$).
\begin{enumerate}[label=(\arabic*), ref=(\arabic*)]
\item\label{mat1} (Böttcher--Wenzel) $\fnorm{[X, Y]}^{2} \leq 2\fnorm{X}^{2}\fnorm{Y}^{2}$ for $X, Y \in \gl(n, \com)$, $n \geq 2$, with equality if and only if $\tr XY^{\dagger} = 0$ and there is a unitary matrix $U$ such that $UXU^{-1}$ and $UYU^{-1}$ lie in the subalgebra $\sll(2, \com) \subset \gl(n, \com)$ of trace-free matrices acting on $\spn_{\com}\{e_{1}, e_{2}\}$ and annihilating $e_{3}, \dots, e_{n}$; for example $X = E_{12}$, $Y = E_{21}$. This is the characterization of Cheng, Vong, and Wenzel \cite[Theorems~3.1 and 3.2]{Cheng-Vong-Wenzel}.
\item\label{mat2} $\fnorm{[X, Y]}^{2} \leq 2\fnorm{X}^{2}\fnorm{Y}^{2}$ for $X, Y \in \sp(2n, \com)$, $n \geq 1$, with equality if and only if $\tr XY^{\dagger} = 0$ and there is $U \in Sp(2n, \com) \cap U(2n)$ such that $UXU^{-1}$ and $UYU^{-1}$ lie in the subalgebra $\sp(2, \com) \subset \sp(2n, \com)$ acting on $\spn_{\com}\{e_{1}, e_{n+1}\}$ and annihilating the other basis vectors.
\item\label{mat3} For complex skew-symmetric $X, Y \in \so(n, \com)$, $n \geq 3$,
\begin{align}\label{eq:BImatrix}
\fnorm{[X, Y]}^{2} \leq c_{n}\fnorm{X}^{2}\fnorm{Y}^{2}, \qquad c_{3} = \tfrac{1}{2}, \quad c_{n} = 1 \ (n \geq 4).
\end{align}
For $n = 3$ equality holds if and only if $\tr XY^{\dagger} = 0$. For $n \geq 4$ equality holds if and only if $\tr XY^{\dagger} = 0$ and there is $U \in SO(n)$ such that $UXU^{-1}$ and $UYU^{-1}$ lie in one of the two simple ideals of the subalgebra $\so(4, \com) \subset \so(n, \com)$ acting on $\spn_{\com}\{e_{1}, \dots, e_{4}\}$ and annihilating $e_{5}, \dots, e_{n}$. For real skew-symmetric matrices \eqref{eq:BImatrix} is the inequality of Bloch--Iserles \cite{Bloch-Iserles}.
\item\label{mat4} (Chern--do Carmo--Kobayashi) $\fnorm{[A, B]}^{2} \leq 2\fnorm{A}^{2}\fnorm{B}^{2}$ for real symmetric $n \times n$ matrices $A$ and $B$, with equality if and only if there are $O \in O(n)$ and $a, b \in \rea$ such that $OAO^{-1} = a(E_{11} - E_{22})$ and $OBO^{-1} = b(E_{12} + E_{21})$.
\end{enumerate}
\end{corollary}
\begin{proof}
The inequalities in \ref{mat1}--\ref{mat3} follow from Theorem \ref{thm:B} and the indices just computed, using Corollary \ref{thm:reductive} for $\gl(n, \com) = \com I \oplus \sll(n, \com)$ and for $\so(4, \com)$. For the equality cases, take on $\sll(n, \com)$, $\sp(2n, \com)$, and $\so(n, \com)$ the compact conjugation $\tau X = -X^{\dagger}$, whose compact groups act by conjugation through $SU(n)$, $Sp(2n, \com) \cap U(2n)$, and $SO(n)$. On each of these Lie algebras $H$ is a positive multiple of $\tr XY^{\dagger}$, so the condition $H(x, y) = 0$ of Theorem \ref{thm:A} is $\tr XY^{\dagger} = 0$. By Theorem \ref{thm:B}, equality holds exactly for the images of the equality pairs of Theorem \ref{thm:A}. By Lemma \ref{lem:lower}\ref{lower3}, every $\tau$-stable subalgebra $\subs \simeq \sll(2, \com)$ of Dynkin index $1$ is conjugate under the compact group $G_{0}$ to the subalgebra $\subs_{\theta}$ attached to the highest root, so it suffices to identify this subalgebra in each case. For $\sll(n, \com)$, $n \geq 2$, the subalgebra attached to the highest root is the subalgebra $\sll(2, \com)$ of \ref{mat1}; by Corollary \ref{thm:reductive}, equality in $\gl(n, \com)$ forces $X, Y \in \sll(n, \com)$. For $\sp(2n, \com)$, with $\Omega$ as above, the highest root is $2\ep_{1}$ and its subalgebra is the subalgebra $\sp(2, \com)$ of \ref{mat2}. For $\so(n, \com)$ with $n \geq 5$, the subalgebra attached to the highest root $\ep_{1} + \ep_{2}$ lies, together with the one attached to $\ep_{1} - \ep_{2}$, in a subalgebra $\so(4, \com)$ acting on a $4$-dimensional subspace spanned by real vectors, as one of its two simple ideals, which gives \ref{mat3} after conjugation by $SO(n)$. For $n = 4$, Corollary \ref{thm:reductive} applies to the two ideals of $\so(4, \com) \simeq \sll(2, \com) \oplus \sll(2, \com)$, each of which carries the constant $1$; for $n = 3$, $\so(3, \com) \simeq \sll(2, \com)$ is its own unique subalgebra isomorphic to $\sll(2, \com)$, so the only condition is $\tr XY^{\dagger} = 0$. The pairs $E_{12}, E_{21}$ and, for $\so(n, \com)$, the real matrices
\begin{align}\label{eq:BIextremal}
\begin{aligned}
&n = 3: & X &= E_{12} - E_{21}, & Y &= E_{23} - E_{32}, & \frac{\fnorm{[X, Y]}^{2}}{\fnorm{X}^{2}\fnorm{Y}^{2}} &= \frac{2}{2\cdot 2} = \frac{1}{2},\\
&n \geq 4: & X &= E_{12} - E_{21} + E_{34} - E_{43}, & Y &= E_{13} - E_{31} - E_{24} + E_{42}, & \frac{\fnorm{[X, Y]}^{2}}{\fnorm{X}^{2}\fnorm{Y}^{2}} &= \frac{16}{4\cdot 4} = 1,
\end{aligned}
\end{align}
padded by zeros for $n > 4$, attain equality, as the values in \eqref{eq:BIextremal} show.

For \ref{mat4}, $iA$ and $iB$ lie in $\mathfrak{u}(n)$ and $\fnorm{[iA, iB]} = \fnorm{[A, B]}$, so the inequality is the restriction of \ref{mat1}. If equality holds, then by \ref{mat1}, $\tr AB = 0$ and there is a $2$-dimensional subspace $P \subset \com^{n}$ such that $A$ and $B$ vanish on its Hermitian orthogonal complement and are trace-free on $P$. A real symmetric matrix of rank one is not trace-free on a subspace containing its range, so $P$ is the sum of the ranges of $A$ and $B$, which is the complexification of a real plane. After conjugation by an element of $O(n)$, $P = \spn_{\com}\{e_{1}, e_{2}\}$, and $A$ and $B$ are trace-free real symmetric $2 \times 2$ blocks, orthogonal for the trace form. A rotation of this plane makes $A = a(E_{11} - E_{22})$, and then $B = b(E_{12} + E_{21})$. Conversely, for these matrices $\fnorm{[A, B]}^{2} = 8a^{2}b^{2} = 2\fnorm{A}^{2}\fnorm{B}^{2}$.
\end{proof}

For Hermitian matrices $A$ and $B$, Corollary \ref{cor:matrix}\ref{mat1} applied to $iA$ and $iB$ gives $\fnorm{[A, B]}^{2} \leq 2\fnorm{A}^{2}\fnorm{B}^{2}$, which is the two-matrix DDVV-type inequality for Hermitian matrices. Since the trace-free Hermitian $2 \times 2$ matrices form the real span of $H_{1} = E_{11} - E_{22}$, $H_{2} = E_{12} + E_{21}$, and $H_{3} = i(E_{21} - E_{12})$, on which conjugation by $SU(2)$ acts through $SO(3)$ preserving the trace form, the equality case of \ref{mat1} shows that equality holds for nonzero $A$ and $B$ if and only if, up to unitary conjugation, $A = \operatorname{diag}(aH_{1}, 0)$ and $B = \operatorname{diag}(bH_{2}, 0)$ with $a, b \in \rea \setminus \{0\}$. The same equality case is obtained, for $\gz = \su(n)$ multiplied by $i$, from Corollary \ref{thm:compact}\ref{compact2}.
In \ref{mat4}, the real symmetric matrices, multiplied by $i$, form the tangent space of the compact symmetric space $SU(n)/SO(n)$ (up to the center), and Chern--do Carmo--Kobayashi obtained the characterization of the equality case as well; from the present point of view it is the case of $SU(n)/SO(n)$ of the curvature bounds for symmetric spaces in Section \ref{sec:curvature}.

For $m$ matrices $B_{1}, \dots, B_{m}$ in a prescribed class, DDVV-type inequalities bound $\sum_{r, s = 1}^{m}\fnorm{[B_{r}, B_{s}]}^{2}$ by a multiple $c\big(\sum_{r=1}^{m}\fnorm{B_{r}}^{2}\big)^{2}$; see the survey \cite{Ge-Li-Tang-Zhou-survey}. For $m = 2$, a DDVV-type inequality with constant $c$ is equivalent to the bound $\fnorm{[X, Y]}^{2} \leq 2c\fnorm{X}^{2}\fnorm{Y}^{2}$, as follows by rescaling $X$ and optimizing. Hence the optimal two-matrix DDVV-type constants for real symmetric, real skew-symmetric, Hermitian, skew-Hermitian, general complex, and complex skew-symmetric matrices are given by Corollary \ref{cor:matrix} and its restrictions. The published DDVV-type constants for these classes \cite{Ge, Ge-Xu-You-Zhou, Ge-Li-Zhou} concern $m \geq 3$ matrices and are different; for example, for $m \geq 3$ Hermitian matrices the optimal constant is $\tfrac{4}{3}$, and the extremal triples are built from the matrices $H_{1}, H_{2}, H_{3}$ described after Corollary \ref{cor:matrix} \cite[Theorem 1.1]{Ge-Xu-You-Zhou}. For general complex matrices the two-matrix case is \cite[Theorem 2.3(2)]{Ge-Li-Zhou}, which is derived there from \eqref{eq:bwclassical} and the equality case of \cite{Cheng-Vong-Wenzel}; its equality case, a pair unitarily conjugate to $(\operatorname{diag}(\tilde{B}_{1}, 0), \operatorname{diag}(\tilde{B}_{2}, 0))$ with $\tilde{B}_{1}, \tilde{B}_{2}$ trace-free, orthogonal $2 \times 2$ blocks of equal norm, is the statement that $x$ and $y$ are orthogonal elements of equal norm in a $\tau$-stable long-root subalgebra of $\sll(n, \com)$. The passage from real to complex matrices in \cite{Ge-Li-Zhou} proceeds by decomposing a complex matrix into Hermitian and skew-Hermitian, or real and imaginary, parts, and for $m \geq 3$ this yields optimal constants. For complex skew-symmetric matrices \cite[Corollary 2.6]{Ge-Li-Zhou} treats $m \geq 3$ matrices, with constants $\tfrac{1}{3}$ ($n = 3$) and $\tfrac{2}{3}$ ($n \geq 4$); the two matrix bound \eqref{eq:BImatrix} does not follow from it, nor does it follow from the decomposition $X = X_{1} + iX_{2}$, $Y = Y_{1} + iY_{2}$ of complex skew-symmetric matrices into real skew-symmetric real and imaginary parts: for $n \geq 4$, combining $\fnorm{[X, Y]}^{2} \leq 2\sum_{a=1}^{2}\sum_{b=1}^{2}\fnorm{[X_{a}, Y_{b}]}^{2}$ with the real Bloch--Iserles inequality yields only the constant $2$ in \eqref{eq:BImatrix}. 

That the Bloch--Iserles constant survives complexification is an instance of the general phenomenon of Theorem \ref{thm:A}, that the optimal constant on $\g$ equals that on $\gz$ for every complex simple $\g$; Theorem \ref{thm:A} thus answers, for two elements and the Hermitian norm \eqref{eq:Hdef}, the question of \cite[Section 1, question (2)]{Ge-Li-Zhou}, \cite[Question 4.2]{Ge-Li-Tang-Zhou-survey} on DDVV-type inequalities for Lie algebras.

\begin{remark}
For $\g = \gl(n, \com)$ with $\tau X = -X^{\dagger}$, the operator $\qad(X) = -\ad(\tau X)\ad(X)$ of \eqref{eq:qaddefined} is $Y \mapsto [X^{\dagger}, [X, Y]]$, the operator $T_{X}$ of Lu and Wenzel \cite{Ge-Li-Lu-Zhou, Lu-Wenzel} whose eigenvalues are the subject of their conjectures; the inequality \eqref{eq:bwclassical} is $\lambda_{1}(T_{X}) \leq 2\fnorm{X}^{2}$, and its analogue on a general complex simple $\g$ is $\qad(x) \leq (\hv)^{-1}\norm{x}^{2}\id_{\g}$, which follows from Theorem \ref{thm:A} because $H(\qad(x)y, y) = \norm{[x, y]}^{2}$ by \eqref{basic5}. The Lu--Wenzel conjectures on partial sums of the eigenvalues of $T_{X}$ have evident formulations for $\qad(x)$, and it would be interesting to know whether the bounds they predict are also governed by root data. Likewise, the DDVV-type inequality $\sum_{r, s = 1}^{m}\norm{[x_{r}, x_{s}]}^{2} \leq c_{m}(\g)\big(\sum_{r = 1}^{m}\norm{x_{r}}^{2}\big)^{2}$ for $m \geq 3$ elements of $\g$ is not treated here; Theorem \ref{thm:A} gives $c_{2}(\g) = 1/(2\hv)$, and for $\gz = \so(n)$ and $\mathfrak{u}(n)$ the constants $c_{m}(\gz)$, $m \geq 3$, are known \cite{Ge, Ge-Xu-You-Zhou}.
\end{remark}

\subsection{Complex sectional curvature}\label{sec:curvature}
The equality of complex and real comass in \eqref{eq:comassintro} has a geometric meaning. Let $(\mathcal{N}, \langle\dum, \dum\rangle)$ be a Riemannian manifold with curvature $R(X, Y) = \nabla_{X}\nabla_{Y} - \nabla_{Y}\nabla_{X} - \nabla_{[X, Y]}$ and $R(X, Y, Z, V) = \langle R(X, Y)Z, V\rangle$. Extend $\langle\dum, \dum\rangle$ and $R$ complex multilinearly to $T_{p}\mathcal{N} \otimes \com$, and let $\langle\langle u, v\rangle\rangle = \langle u, \bar{v}\rangle$. The \emph{complex sectional curvature} of a complex $2$-plane $\sigma = \spn_{\com}\{u, v\} \subset T_{p}\mathcal{N} \otimes \com$ is \cite{Micallef-Moore}
\begin{align}\label{eq:complexsectional}
\mathsf{K}_{\com}(\sigma) = \frac{R(u, v, \bar{v}, \bar{u})}{\langle\langle u, u\rangle\rangle\langle\langle v, v\rangle\rangle - |\langle\langle u, v\rangle\rangle|^{2}};
\end{align}
it depends only on $\sigma$ and equals the sectional curvature when $\sigma$ is the complexification of a real plane. Let $G_{0}$ be as in Section \ref{sec:prelim}, with bi-invariant metric $-\kappa B$, $\kappa > 0$; everything that follows depends only on $\gz$ and $\kappa$, so it holds for any compact connected Lie group with Lie algebra $\gz$. Then $R(X, Y)Z = -\tfrac{1}{4}[[X, Y], Z]$ for left-invariant vector fields \cite{Helgason}, the complexified tangent space at the identity is $\g$, and $\langle\langle u, v\rangle\rangle = \kappa H(u, v)$. By the invariance of $B$, $R(u, v, \tau v, \tau u) = \tfrac{\kappa}{4}B([u, v], \tau[v, u]) = \tfrac{\kappa}{4}\norm{[u, v]}^{2}$, so
\begin{align}\label{eq:KCgroup}
\mathsf{K}_{\com}(\sigma) = \frac{1}{4\kappa}\sect(\sigma), \qquad \sigma \in \grass(2, \g),
\end{align}
and Corollaries \ref{cor:sectmax} and \ref{thm:compact} give the following.
\begin{corollary}\label{cor:groupcurvature}
The complex sectional curvature of $(G_{0}, -\kappa B)$ satisfies $0 \leq \mathsf{K}_{\com} \leq 1/(4\kappa\hv)$, with equality exactly for the complex planes contained in the complexification of a subalgebra of $\gz$ isomorphic to $\su(2)$ whose complexification has Dynkin index $1$. In particular, the maximal complex sectional curvature equals the maximal sectional curvature.
\end{corollary}
On real planes the maximum is classical \cite[p.~334, Theorem~11.1]{Helgason-dgls}, \cite{Liu-curvatureestimates}, \cite[Theorem~4]{Meshcheryakov} (see Remark \ref{rem:meshcheryakov}). The extremal complex planes need not be complexifications of real planes: the plane spanned by $e$ and $h$ in Example \ref{ex:eh} is extremal but not $\tau$-stable.

Let $\mathcal{N} = U/K$ be an irreducible compact symmetric space with $U$ a compact simple Lie group, $\mathfrak{u} = \mathfrak{k} \oplus \p$ the Cartan decomposition, and the metric $-\kappa B$ on $\p \simeq T_{o}\mathcal{N}$. Since $R(X, Y)Z = -[[X, Y], Z]$ for $X, Y, Z \in \p$ \cite{Helgason}, the same computation gives $\mathsf{K}_{\com} = \kappa^{-1}\sect$ on $\grass(2, \p \otimes \com)$, so $\mathsf{K}_{\com} \leq 1/(\kappa\hv)$ by Corollary \ref{cor:sectmax}, where $\hv$ is the dual Coxeter number of $\mathfrak{u} \otimes \com$. If equality holds for $\spn_{\com}\{x, y\} \subset \p \otimes \com$, then by Theorem \ref{thm:A} the subalgebra spanned by $x$, $y$, and $\tau[x, y] \in \mathfrak{k} \otimes \com$ is $\tau$-stable of Dynkin index $1$ and is stable under the involution of the symmetric pair, so its compact form contains a real plane in $\p$ of curvature $1/(\kappa\hv)$ by \eqref{eq:sl2norm}. Hence the bound is attained if and only if the maximal sectional curvature, the squared length of the highest restricted root \cite[p.~334, Theorem~11.1]{Helgason-dgls}, \cite[Table~5.1]{Liu-curvatureestimates}, equals $1/(\kappa\hv)$. This holds for $SU(n)/SO(n)$, which gives Corollary \ref{cor:matrix}\ref{mat4}, and for Hermitian symmetric spaces, whose highest root is noncompact. It fails for the spheres $SO(n+1)/SO(n)$, $n \geq 4$, of constant curvature $1/(2\kappa\hv)$, and for $SU(2n)/Sp(n)$ and $E_{6}/F_{4}$, whose maximal sectional curvature is also $1/(2\kappa\hv)$; for the latter two, the commutator inequalities for quaternionic and octonionic Hermitian matrices of \cite[Theorem~6.3]{Fox-sectna} are therefore not consequences of Theorem \ref{thm:A}. I do not know whether, for every irreducible compact symmetric space $U/K$ with $U$ a compact simple Lie group, the maximal complex sectional curvature equals the maximal sectional curvature, as Corollary \ref{cor:groupcurvature} shows for compact simple Lie groups.

\section{Critical points and Dynkin indices}\label{sec:critical}
The maxima of the functions $F$, $\sect$, and $W$ have been described. Their other critical points having nonzero critical value are also closely linked to interesting subalgebras of $\g$. 
This section describes these critical points, their critical values, and the Hessian of $W$ at them, and shows that the set of maximum points of $W$ is a nondegenerate critical manifold.

Lemma \ref{lem:critFK} records the equations of the critical points of $F$ and $\sect$ and some relations between critical points of $F$, $\sect$, and $W$.
\begin{lemma}\label{lem:critFK}
\leavevmode\noindent
\begin{enumerate}[label=(\arabic*), ref=(\arabic*)]
\item\label{critF} $(x, y) \in (\g \setminus \{0\})^{2}$ is a critical point of $F$ if and only if 
\begin{align}\label{eq:F}
\qad(y)x = F(x,y)\norm{y}^2 x, \qquad \qad(x)y = F(x,y)\norm{x}^2 y .
\end{align}
If $(x,y)\in\crtnt(F)$, then $H(x, y) = 0$ and $(x,y,\tau[x,y])\in\crtnt(W)$, with $W(x,y,\tau[x,y])=F(x,y)$.
\item\label{critK} $(x, y) \in \stiefel(\g)$ is a critical point of $\sect$ if and only if
\begin{align}\label{eq:K}
\qad(y)x = \sect(x, y)\big(\norm{y}^{2}x - H(x, y)y\big), \qquad \qad(x)y = \sect(x, y)\big(\norm{x}^{2}y - H(y, x)x\big).
\end{align}
\item\label{corr3} If $H(x, y) = 0$, then $(x, y)$ is a critical point of $\sect$ if and only if it is a critical point of $F$, and $(x, y)$ is a critical point of $\sect$ if and only if $\spn_{\com}\{x, y\}$ is a critical point of $\sect$ on $\grass(2, \g)$.
\item\label{corr2} If $a \in\crtnt(W)$, then $(a_{j}, a_{k})\in\crtnt(F)$ and $F(a_{j}, a_{k}) = W(a)$ for $j \neq k \in \{1, 2, 3\}$.
\end{enumerate}
\end{lemma}
\begin{proof}
By \eqref{basic5}, the derivative of $\norm{[x, y]}^{2}$ in the direction $(\dot{x}, \dot{y})$ is $2\re H(\qad(y)x, \dot{x}) + 2\re H(\qad(x)y, \dot{y})$, and the derivatives of $\norm{x}^{2}\norm{y}^{2}$ and of $\norm{x}^{2}\norm{y}^{2} - |H(x, y)|^{2}$ are $2\re H(\norm{y}^{2}x, \dot{x}) + 2\re H(\norm{x}^{2}y, \dot{y})$ and $2\re H(\norm{y}^{2}x - H(x, y)y, \dot{x}) + 2\re H(\norm{x}^{2}y - H(y, x)x, \dot{y})$. There follow \eqref{eq:F} and \eqref{eq:K}. 

If \eqref{eq:F} holds, then $F(x, y)\norm{y}^{2}H(x, y) = H(\qad(y)x, y) = H([y, x], [y, y]) = 0$. Let $z := \tau[x, y]$ and $t := \om(x, y, z) = -\norm{[x, y]}^{2}$, which is nonzero. Because $\tau$ is an automorphism, \eqref{eq:F} and $\norm{[x, y]}^{2} = F(x, y)\norm{x}^{2}\norm{y}^{2}$ give $[y, z] = \tau[\tau y, [x, y]] = \tau\qad(y)x = -(t/\norm{x}^{2})\tau x$ and $[z, x] = -\tau[\tau x, [x, y]] = \tau\qad(x)y = -(t/\norm{y}^{2})\tau y$, while $[x, y] = \tau z = -(t/\norm{z}^{2})\tau z$ because $\norm{z}^{2} = -t$. So \eqref{eq:W} holds, and $W(x, y, z) = |t|^{2}/(\norm{x}^{2}\norm{y}^{2}\norm{z}^{2}) = F(x, y)$. This completes the proof of \ref{critF}.

If $H(x, y) = 0$, then $\sect(x, y) = F(x, y)$ and \eqref{eq:K} coincides with \eqref{eq:F}. The map $\stiefel(\g) \to \grass(2, \g)$ is a surjective submersion along whose fibers $\sect$ is constant by \eqref{eq:sectinvariance}, which proves \ref{corr3}.

Because $F$ and $W$ are invariant under rescaling their arguments, by Lemma \ref{thm:idem}\ref{idem1} to prove \ref{corr2} it suffices to suppose that $a \in \crtnt(W)$ is a $\tau$-twisted idempotent $e \in \idem_{\tau}(\g^{3}, \tmlt)$. By Lemma \ref{thm:idem}\ref{idem2}, 
\begin{align}\label{eq:corrcheck}
\qad(e_{2})e_{1} = [\tau e_{2}, [e_{1}, e_{2}]] = \tfrac{1}{2}\tau[e_{2}, e_{3}] = \tfrac{1}{4}e_{1}, \qquad \qad(e_{1})e_{2} = -[\tau e_{1}, [e_{1}, e_{2}]] = \tfrac{1}{2}\tau[e_{3}, e_{1}] = \tfrac{1}{4}e_{2},
\end{align}
and $F(e_{1}, e_{2})\norm{e_{1}}^{2} = F(e_{1}, e_{2})\norm{e_{2}}^{2} = \tfrac{1}{4}$ by \eqref{eq:idemnorms}, so, by \eqref{eq:corrcheck}, $(e_{1}, e_{2})$ satisfies \eqref{eq:F}; the same computation applies to $(e_{2}, e_{3})$ and $(e_{3}, e_{1})$. This proves \ref{corr2}.
\end{proof}

Lemma \ref{thm:sl2crit} shows that every $\tau$-stable subalgebra of $\g$ isomorphic to $\sll(2, \com)$ and having Dynkin index $j$ determines critical points of each of $F$, $\sect$, and $W$ with critical value $1/(j\hv)$. 

\begin{lemma}\label{thm:sl2crit}
Let $\subs \subset \g$ be a $\tau$-stable subalgebra isomorphic to $\sll(2, \com)$ and having Dynkin index $j$, and let $z_{1}, z_{2}, z_{3}$ be as in \eqref{eq:xirelations} of Lemma \ref{lem:sl2str}\ref{sl2a}.
\begin{enumerate}[label=(\arabic*), ref=(\arabic*)]
\item\label{sl2crit2} Every pair of nonzero $H$-orthogonal $x, y \in \subs$ lies in $\crtnt(F)$, with $F(x, y) = 1/(j\hv)$.
\item\label{sl2crit1} Every $\sigma \in \grass(2, \subs)$ is a critical point of $\sect$ on $\grass(2, \g)$, with $\sect(\sigma) = 1/(j\hv)$.
\item\label{sl2crit3} $e := \tfrac{1}{2}(z_{1}, z_{2}, z_{3})$ is a $\tau$-twisted idempotent with $W(e) = 1/(j\hv)$ and $H(e, e) = \tfrac{3}{4}j\hv$.
\end{enumerate}
Consequently, by Lemma \ref{lem:lower}\ref{lower1}, for every subalgebra of $\g$ isomorphic to $\sll(2, \com)$, having Dynkin index $j$, the number $1/(j\hv)$ is a critical value of each of $F$, $\sect$, and $W$.
\end{lemma}
\begin{proof}
Let $\beta = j\hv$. For \ref{sl2crit2}, use the coordinates of the proof of Lemma \ref{lem:sl2str}\ref{sl2b}, in which $\tau$ acts by complex conjugation of coefficient vectors because the $z_{a}$ lie in $\gz$. Since $H(x, y) = \beta\,x\cdot\bar{y} = 0$, the identity $a \times (b \times c) = b(a \cdot c) - c(a \cdot b)$ gives
\begin{align}\label{eq:triplecross}
\bar{y}\times(x \times y) = (\bar{y}\cdot y)\,x, \qquad -\bar{x}\times(x \times y) = (\bar{x}\cdot x)\,y,
\end{align}
and these are the coefficient vectors of $\qad(y)x = [\tau y, [x, y]]$ and $\qad(x)y = -[\tau x, [x, y]]$. Since $\bar{y}\cdot y = \norm{y}^{2}/\beta$ and $F(x, y) = 1/\beta$ by \eqref{eq:sl2norm}, \eqref{eq:triplecross} is \eqref{eq:F}. Claim \ref{sl2crit1} follows from \ref{sl2crit2} and Lemma \ref{lem:critFK}\ref{corr3}, since every plane in $\subs$ has an $H$-orthogonal basis. For \ref{sl2crit3}, let $z := (z_{1}, z_{2}, z_{3})$. By \eqref{eq:xirelations} and Lemma \ref{lem:nahmbasic}\ref{nahm2}, $z \tmlt z = 2z$, and $\tau z = z$ because the $z_{a}$ lie in $\gz$, so $e = \tfrac{1}{2}z$ satisfies $e \tmlt e = \tfrac{1}{2}z = \tau e$. Since $\norm{e_{a}}^{2} = \beta/4$ for $a \in \{1, 2, 3\}$, \eqref{eq:idemnorms} gives the values.
\end{proof}

The Dynkin indices of the subalgebras of $\g$ isomorphic to $\sll(2, \com)$ are tabulated by Dynkin \cite{Dynkin-semisimplesubalgebras}; for the principal subalgebra of $\sll(n, \com)$ the index is $\binom{n+1}{3}$. In general these subalgebras correspond to the nilpotent orbits of $\g$ \cite{Collingwood-McGovern}.

\begin{remark}\label{rem:meshcheryakov}
For planes in $\gz$, Lemma \ref{thm:sl2crit}\ref{sl2crit1} and its converse are in \cite[Proposition~1 and Theorems~1, 2, and~4]{Meshcheryakov}: the critical points of the sectional curvature of $(G_{0}, -B)$ with nonzero critical value are the planes contained in subalgebras of $\gz$ isomorphic to $\su(2)$, and the critical value on such a plane is $1/(4j\hv)$, where $j$ is the Dynkin index of the complexification of the subalgebra. The converse follows from Lemma \ref{lem:critFK}. For $x, y \in \gz$ both sides of \eqref{eq:K} lie in $\gz$, and $\re H$ restricts to $-B$ on $\gz$, so a pair $(x, y) \in \gz^{2}$ orthonormal with respect to $-B$ is a critical point of the restriction of $\sect$ to $\stiefel(\gz)$ if and only if it satisfies \eqref{eq:K}, which, since $\tau$ fixes $x$ and $y$, reads $[y, z] = \sect(x, y)x$ and $[z, x] = \sect(x, y)y$ with $z := [x, y]$; these are the criticality conditions of \cite[Proposition~1]{Meshcheryakov}. If $\sect(x, y) \neq 0$, they show that $x$, $y$, and $z$ span a subalgebra of $\gz$ isomorphic to $\su(2)$, whose complexification is $\tau$-stable, so $\sect(x, y) = 1/(j\hv)$ by \eqref{eq:sl2norm}. In \cite[Theorem~4]{Meshcheryakov} the critical value is written $3/\sum_{s=1}^{m}n_{s}(n_{s}+1)(n_{s}+2)$, where $n_{1}+1, \dots, n_{m}+1$ are the dimensions of the irreducible summands of $\gz$ under the subalgebra; this agrees with $1/(4j\hv)$ because, by \eqref{eq:indexweights}, the Dynkin index of the irreducible $(n+1)$-dimensional representation of $\sll(2, \com)$ is $n(n+1)(n+2)/6$, so $\sum_{s=1}^{m}n_{s}(n_{s}+1)(n_{s}+2)$ is $6$ times the Dynkin index $2j\hv$ of the restriction of the adjoint representation of $\g$ to the complexified subalgebra.

For complex planes this description fails. If $\mu_{1}, \mu_{2}, \mu_{3} \in \Phi$ are long roots with $\mu_{1} + \mu_{2} + \mu_{3} = 0$ and $0 \neq e_{a} \in \g_{\mu_{a}}$ for $a \in \{1, 2, 3\}$, then $H(e_{1}, e_{2}) = 0$, and by Lemmas \ref{thm:roottriangle} and \ref{lem:critFK}\ref{corr2},\ref{corr3} the plane $\spn_{\com}\{e_{1}, e_{2}\}$ is a critical point of $\sect$ on $\grass(2, \g)$ with critical value $1/(2\hv)$; but $[e_{a}, [e_{1}, e_{2}]] \in \g_{\mu_{a} - \mu_{3}} = 0$ for $a \in \{1, 2\}$, so $e_{1}$, $e_{2}$, and $[e_{1}, e_{2}]$ span a nilpotent subalgebra, and the plane lies in no subalgebra isomorphic to $\sll(2, \com)$.
\end{remark}

Hitchin \cite[Theorem~1]{Hitchin-calibrations} showed that the normalized Cartan $3$-form $\phi$ of \eqref{eq:tasakiform} defined by Tasaki, viewed as a function on the Grassmannian $\widetilde{\grass}(3, \gz)$ of oriented three-dimensional subspaces, has a critical point at a principal three-dimensional subalgebra. For $\sigma \in \widetilde{\grass}(3, \gz)$ write $\om(\sigma)$ for the value of $\om$ on a positively oriented basis of $\sigma$ orthonormal with respect to $-B$. Since $\phi = -(\hv)^{1/2}\om$ on $\gz$, the functions $\phi$ and $\om$ have the same critical points on $\widetilde{\grass}(3, \gz)$. Corollary \ref{cor:hitchin} extends Hitchin's observation to all subalgebras isomorphic to $\su(2)$ and computes the critical values; it is the counterpart for the Cartan $3$-form of Meshcheryakov's result for the sectional curvature (Remark \ref{rem:meshcheryakov}).

\begin{corollary}\label{cor:hitchin}
Let $\subs_{0} \subset \gz$ be a subalgebra isomorphic to $\su(2)$ whose complexification has Dynkin index $j$. Then $\subs_{0}$, with either orientation, is a critical point of $\om$ on $\widetilde{\grass}(3, \gz)$, and $|\om(\subs_{0})| = (j\hv)^{-1/2}$.
\end{corollary}
\begin{proof}
The complexification $\subs = \subs_{0} \oplus i\subs_{0}$ is $\tau$-stable, so Lemma \ref{lem:sl2str}\ref{sl2a} applies: a triple $z_{1}, z_{2}, z_{3}$ as there is a basis of $\subs_{0}$, orthogonal with respect to $-B$ with $-B(z_{a}, z_{a}) = j\hv$ for $a \in \{1, 2, 3\}$. Hence $X = (j\hv)^{-1/2}(z_{1}, z_{2}, z_{3})$ is an orthonormal $3$-frame in $\gz$ and a rescaling of the $\tau$-twisted idempotent of Lemma \ref{thm:sl2crit}\ref{sl2crit3}, hence a critical point of $W$ on $\g^{3}$, and so also of the restriction of $W$ to the manifold $\mathcal{V}$ of orthonormal $3$-frames in $\gz$. Since $\om$ is real on $\gz$, $W = \om^{2}$ on $\mathcal{V}$, and $\om^{2}$ on $\mathcal{V}$ is the pullback of $\om^{2}$ on $\widetilde{\grass}(3, \gz)$ via the submersion $\mathcal{V} \to \widetilde{\grass}(3, \gz)$ sending a frame to its oriented span. Hence $\om^{2}$, and so $\om$ (which is nonzero there), is critical at $\subs_{0}$, and $\om(\subs_{0})^{2} = W(X) = 1/(j\hv)$ by Lemma \ref{thm:sl2crit}\ref{sl2crit3}.
\end{proof}
For Tasaki's calibration $\phi = -(\hv)^{1/2}\om$, the critical value in Corollary \ref{cor:hitchin} is $|\phi(\subs_{0})| = j^{-1/2}$, to be compared with the comass $1$ of $\phi$. For the principal subalgebra of $\su(n)$, $j = \binom{n+1}{3}$, and for that of a compact Lie algebra of type $G_{2}$, $j = 28$ (Example \ref{ex:g2}); the critical values of $\phi$ at principal subalgebras are thus far from the maximal value $1$.

Not all critical points of $W$ are of the type described in Lemma \ref{thm:sl2crit}. For $\mu \in \Phi$ let $t_{\mu} \in \h$ be as in \ref{R2}.
\begin{lemma}\label{thm:roottriangle}
Let $\mu_{1}, \mu_{2}, \mu_{3} \in \Phi$ satisfy $\mu_{1} + \mu_{2} + \mu_{3} = 0$, and let $0 \neq e_{a} \in \g_{\mu_{a}}$ for $a \in \{1, 2, 3\}$.
\begin{enumerate}[label=(\arabic*), ref=(\arabic*)]
\item\label{tri1} $(e_{1}, e_{2}, e_{3}) \in \crtnt(W)$.
\item\label{tri2} If $\mu_{a} - \mu_{b} \notin \Phi$ for all $a \neq b$, then the $\mu_{a}$ have a common length, $\subs := \spn_{\com}\{t_{\mu_{1}}, t_{\mu_{2}}\} \oplus \bigoplus_{a=1}^{3}(\g_{\mu_{a}} \oplus \g_{-\mu_{a}})$ is a $\tau$-stable subalgebra isomorphic to $\sll(3, \com)$ and having Dynkin index $j = 2/(\mu_{1}, \mu_{1})$, and
\begin{align}\label{eq:trianglevalue}
W(e_{1}, e_{2}, e_{3}) = \frac{(\mu_{1}, \mu_{1})}{4\hv} = \frac{1}{2j\hv}.
\end{align}
\item\label{tri3} The hypothesis of \ref{tri2} holds if the $\mu_{a}$ are long, and then $W(e_{1}, e_{2}, e_{3}) = 1/(2\hv)$. If the $\mu_{a}$ have a common length, the hypothesis of \ref{tri2} fails only for short roots of $G_{2}$, for which $\mu_{a} - \mu_{b}$ is a long root.
\end{enumerate}
\end{lemma}
\begin{proof}
For $(a, b, c)$ a cyclic permutation of $(1, 2, 3)$, $\mu_{b} + \mu_{c} = -\mu_{a}$ is a root, so $[\g_{\mu_{b}}, \g_{\mu_{c}}] = \g_{-\mu_{a}}$ \cite[Proposition~8.4]{Humphreys}, which equals $\com\,\tau e_{a}$ by \eqref{eq:tauroot} and \ref{R2}. Moreover $C(e_{1}, e_{2}, e_{3}) = B([e_{1}, e_{2}], e_{3}) \neq 0$, because $[e_{1}, e_{2}]$ spans $\g_{-\mu_{3}}$ and $B$ pairs $\g_{-\mu_{3}}$ and $\g_{\mu_{3}}$ nondegenerately. Lemma \ref{lem:critW} shows $(e_{1}, e_{2}, e_{3}) \in \crtnt(W)$, which is \ref{tri1}.

For \ref{tri3}, if the $\mu_{a}$ are long and $a \neq b$, then $(\mu_{a} + \mu_{b}, \mu_{a} + \mu_{b}) \leq 2$ forces $(\mu_{a}, \mu_{b}) \leq -1$, so $(\mu_{a} - \mu_{b}, \mu_{a} - \mu_{b}) \geq 6$ and $\mu_{a} - \mu_{b} \notin \Phi$ by \ref{R1}. If the $\mu_{a}$ have a common squared length $\ell$, then $(\mu_{a}, \mu_{b}) = -\ell/2$ for $a \neq b$, because $\mu_{a} + \mu_{b} = -\mu_{c}$, so $(\mu_{a} - \mu_{b}, \mu_{a} - \mu_{b}) = 3\ell$, which is the squared length of a root only for short roots of $G_{2}$. The value in \ref{tri3} follows from \eqref{eq:trianglevalue} with $(\mu_{1}, \mu_{1}) = 2$, once \ref{tri2} is proved.

For \ref{tri2}, suppose $\mu_{a} - \mu_{b} \notin \Phi$ for $a \neq b$, and let $(a, b, c)$ be a cyclic permutation of $(1, 2, 3)$. The $\mu_{a}$-string through $\mu_{b}$ is $\mu_{b}, \mu_{b} + \mu_{a}$, because $\mu_{b} - \mu_{a} \notin \Phi$ and $\mu_{b} + 2\mu_{a} = \mu_{a} - \mu_{c} \notin \Phi$, so $\mu_{b}(\mu_{a}^{\vee}) = -1$ \cite[Proposition~8.4]{Humphreys}. By \ref{R2} and $t_{\mu_{a}} = (\mu_{a}, \mu_{a})\mu_{a}^{\vee}/(4\hv)$, the elements $X_{+} := e_{a}$ and $X_{-} := -4\hv\,\tau e_{a}/((\mu_{a}, \mu_{a})\norm{e_{a}}^{2})$ satisfy $[X_{+}, X_{-}] = \mu_{a}^{\vee}$. Since $[X_{-}, e_{b}] \in \g_{\mu_{b} - \mu_{a}} = 0$ and $[\mu_{a}^{\vee}, e_{b}] = -e_{b}$, the Jacobi identity gives $[X_{-}, [X_{+}, e_{b}]] = -[\mu_{a}^{\vee}, e_{b}] = e_{b}$, that is,
\begin{align}\label{eq:triangleQ}
\qad(e_{a})e_{b} = -[\tau e_{a}, [e_{a}, e_{b}]] = \frac{(\mu_{a}, \mu_{a})\norm{e_{a}}^{2}}{4\hv}\,e_{b}.
\end{align}
By \eqref{eq:triangleQ}, $F(e_{a}, e_{b}) = H(\qad(e_{a})e_{b}, e_{b})/(\norm{e_{a}}^{2}\norm{e_{b}}^{2}) = (\mu_{a}, \mu_{a})/(4\hv)$. By Lemma \ref{lem:critFK}\ref{corr2} and the symmetry of $F$, $F(e_{a}, e_{b}) = W(e_{1}, e_{2}, e_{3})$ for all $a \neq b$, so the $\mu_{a}$ have a common length and $W(e_{1}, e_{2}, e_{3}) = (\mu_{1}, \mu_{1})/(4\hv)$.

The space $\subs$ contains $t_{\mu_{3}} = -t_{\mu_{1}} - t_{\mu_{2}}$, and it is closed under brackets because $[\g_{\mu_{a}}, \g_{\mu_{b}}] = \g_{-\mu_{c}}$, $[\g_{-\mu_{a}}, \g_{-\mu_{b}}] = \g_{\mu_{c}}$, $[\g_{\mu_{a}}, \g_{-\mu_{b}}] \subset \g_{\mu_{a} - \mu_{b}} = 0$ for $a \neq b$, and $[\g_{\mu_{a}}, \g_{-\mu_{a}}] = \com t_{\mu_{a}}$ by \ref{R2}. It is $\tau$-stable by \eqref{eq:tauroot}, because $t_{\mu} \in i\mathfrak{t}_{0}$ satisfies $\tau t_{\mu} = -t_{\mu}$; hence, as in the proof of Lemma \ref{thm:idem}\ref{idem3}, $\subs$ is reductive. Its center is trivial, because $B$ is nondegenerate on $\spn_{\com}\{t_{\mu_{1}}, t_{\mu_{2}}\}$, so no nonzero element of this space is annihilated by $\mu_{1}$ and $\mu_{2}$. Thus $\subs$ is semisimple of dimension $8$ with the $2$-dimensional Cartan subalgebra $\spn_{\com}\{t_{\mu_{1}}, t_{\mu_{2}}\}$, so $\subs \simeq \sll(3, \com)$, all of whose roots are long. By \eqref{eq:dynkinpullback}, $(\dum, \dum)$ restricts to $j(\dum, \dum)_{\subs}$ on $\subs$, so the induced forms on the duals of the Cartan subalgebras satisfy $(\mu_{1}, \mu_{1}) = (\mu_{1}, \mu_{1})_{\subs}/j = 2/j$, which gives the second equality in \eqref{eq:trianglevalue}.
\end{proof}

Let $\subs\subset \g$ be as in Lemma \ref{thm:sl2crit}, a $\tau$-stable subalgebra isomorphic to $\sll(2, \com)$ and having Dynkin index $j$. Because $\subs$ is $\tau$-stable, the $H$-orthogonal complement $\subs^{\perp}$ of $\subs$ equals its $B$-orthogonal complement. Consequently $\subs^{\perp}$ is $\ad(\subs)$-invariant, and it decomposes as a sum of irreducible $\subs$-modules. For $d \in \integer_{\geq 0}$ let $V(d)$ be the irreducible $\subs$-module with highest weight $d$, that is, of dimension $d+1$, on which the neutral element $h$ of an $\sll(2, \com)$-triple spanning $\subs$ has eigenvalues $d, d-2, \dots, -d$, and write
\begin{align}\label{eq:sperpdecomp}
\subs^{\perp} \simeq \bigoplus_{d \in \integer_{\geq 0}}n_{d}V(d).
\end{align}
Write $\mathcal{P} := \proj(\g)^{3}$, let $\pi_{i}:\mathcal{P} \to \proj(\g)$, $i \in \{1, 2, 3\}$, be the projections onto the factors, let $\mathcal{X} := \{[a] \in \mathcal{P}: C(a) \neq 0\}$, and let $D := 3(\dim_{\com}\g - 1) = \dim_{\com}\mathcal{P}$. Since $W$ is $(\com^{\ast})^{3}$-invariant, it is a function on $\mathcal{P}$, positive exactly on $\mathcal{X}$, whose critical points in $\mathcal{X}$ are the images of the points of $\crtnt(W)$. At the image $[e]$ of a $\tau$-twisted idempotent $e$, the tangent space $T_{[e]}\mathcal{P}$ is identified, as a complex vector space, with $\Uperp_{e}$, and by Lemma \ref{thm:hessW} the Hessian of $W$ on $\mathcal{P}$ at $[e]$ is the form \eqref{eq:hessR} on $\Uperp_{e}$. The \emph{coindex}, \emph{nullity}, and \emph{index} of a critical point of $W$ in $\mathcal{X}$ are the numbers of positive, zero, and negative eigenvalues of this form, counted with real multiplicity; their sum is $2D$.

\begin{lemma}\label{lem:indexbound}
Let $e \in \idem_{\tau}(\g^{3}, \tmlt)$, and let $\mu$ run over the eigenvalues of $\tau\lin(e)$ on $\Uperp_{e}$, counted with real multiplicity. Then the coindex, nullity, and index of $[e]$ are $\#\{\mu > \tfrac{1}{2}\}$, $\#\{\mu = \tfrac{1}{2}\}$, and $\#\{\mu < \tfrac{1}{2}\} \geq D$, with equality in the last if and only if $\nnorm{\tau\lin(e)b} \geq \tfrac{1}{2}\nnorm{b}$ for all $b \in \Uperp_{e}$. In particular, no critical point of $W$ at which $W \neq 0$ is a local minimum.
\end{lemma}
\begin{proof}
By Lemma \ref{hess2}, $\tau\lin(e)$ is self-adjoint on $\Uperp_{e}$, so \eqref{eq:hessR} gives the counts. If $\hess W_{e}$ is positive semidefinite on a subspace $V \subset \Uperp_{e}$, then by \eqref{eq:hessJ} it is negative definite on $JV$, so the index is at least $\dim_{\rea}V$. Taking $V$ maximal, so that $\dim_{\rea}V$ is $\dim_{\rea}\Uperp_{e}$ minus the index, shows that the index is at least $\tfrac{1}{2}\dim_{\rea}\Uperp_{e} = D$. Since $\tau\lin(e)$ anticommutes with $J$, its eigenvalues on $\Uperp_{e}$ come in pairs $\pm\mu$, with $\mu \geq 0$ and eigenvectors $b$ and $Jb$, on which $\hess W_{e}$ takes the values $16W(e)^{2}(\pm\mu - \tfrac{1}{2})\nnorm{b}^{2}$; such a pair contributes two negative directions if $\mu < \tfrac{1}{2}$ and one otherwise, which gives the equality statement.
\end{proof}

The Hessian at the critical points of Lemma \ref{thm:sl2crit} is computed using a description of the Nahm algebra as a tensor product. Let $\mathsf{c}_{1}, \mathsf{c}_{2}, \mathsf{c}_{3}$ be the standard basis of $\com^{3}$, identify $\g^{3}$ with $\g \otimes \com^{3}$ via $a \mapsto \sum_{i=1}^{3}a_{i} \otimes \mathsf{c}_{i}$, and let $\times$ be the complex bilinear cross product on $\com^{3}$, so that $(\com^{3}, \times) \simeq \so(3, \com)$. By \eqref{nahmproduct},
\begin{align}\label{eq:nahmtensor}
(x \otimes u) \tmlt (y \otimes v) = [x, y] \otimes (u \times v), \qquad x, y \in \g,\ u, v \in \com^{3},
\end{align}
so that $(\g^{3}, \tmlt)$ is the tensor product of the Lie algebras $\g$ and $(\com^{3}, \times)$, and $SO(3, \com)$, acting on the second factor, acts by automorphisms of $(\g^{3}, \tmlt)$ \cite[Corollary~9.1]{Kinyon-Sagle}. In the same terms, the forms $B$ and $H$ and the conjugation $\tau$ of \eqref{eq:componentwise} are those of $\g$ tensored with the bilinear form $u \cdot v = \sum_{i=1}^{3}u_{i}v_{i}$, the Hermitian form $u \cdot \bar{v}$, and the complex conjugation $u \mapsto \bar{u}$ of $\com^{3}$, respectively.

\begin{theorem}\label{thm:morse}
Let $e \in \idem_{\tau}(\g^{3}, \tmlt)$ be the $\tau$-twisted idempotent of Lemma \ref{thm:sl2crit}\ref{sl2crit3}, let $n_{d}$ be as in \eqref{eq:sperpdecomp}, and let $D = 3(\dim_{\com}\g - 1)$ be as above.
\begin{enumerate}[label=(\arabic*), ref=(\arabic*)]
\item\label{morse1} On $\g^{3} = \g \otimes \com^{3}$, $\tau\lin(e) = \tfrac{1}{2}\tau M$, where
\begin{align}\label{eq:spinorbit}
M := \sum_{l=1}^{3}\ad(z_{l})\otimes T_{l}, \qquad T_{l}v := \mathsf{c}_{l} \times v, \quad v \in \com^{3}.
\end{align}
The operator $M$ is complex linear, $H$-Hermitian, and commutes with $\tau$, and its eigenvalues on the summands of $V(d) \otimes \com^{3} \simeq V(d+2) \oplus V(d) \oplus V(d-2)$ (for $d \geq 2$) are $-\tfrac{d}{2}$, $1$, and $\tfrac{d}{2}+1$; for $d = 1$ they are $-\tfrac{1}{2}$ on $V(3)$ and $1$ on $V(1)$, and for $d = 0$, $M = 0$. On each eigenspace of $M$ with eigenvalue $m$, $\tau\lin(e)$ has eigenvalues $\pm\tfrac{m}{2}$, each with real multiplicity equal to the complex dimension of the eigenspace.
\item\label{morse2} The coindex, nullity, and index of the critical point $[e] \in \mathcal{P}$ of $W$ are
\begin{align}\label{eq:morseindex}
\begin{aligned}
\operatorname{coind} &= n_{2} + \sum_{d \geq 3}(2d+2)n_{d}, \qquad \operatorname{null} = 6 + 2n_{1} + 8n_{2} + \sum_{d \geq 3}(d+1)n_{d},\\
\operatorname{ind} &= D + 3n_{0} + 4n_{1}.
\end{aligned}
\end{align}
\item\label{morse3} The Hessian of $W$ at $[e]$ is negative semidefinite if and only if $j = 1$.
\item\label{morse4} $\operatorname{ind} = D$ if and only if the centralizer of $\subs$ in $\g$ is trivial, which is the case if and only if $\subs$ is a distinguished subalgebra in the sense of Bala and Carter \cite[Section~8.2]{Collingwood-McGovern}. In particular $\operatorname{ind} = D$ for the principal subalgebra.
\end{enumerate}
\end{theorem}
\begin{proof}
Let $A_{l} := \ad(z_{l})$ for $l \in \{1, 2, 3\}$. These operators are skew-Hermitian, commute with $\tau$, and satisfy $[A_{1}, A_{2}] = A_{3}$, $[A_{2}, A_{3}] = A_{1}$, $[A_{3}, A_{1}] = A_{2}$ by \eqref{eq:xirelations}. Since $e = \tfrac{1}{2}\sum_{l=1}^{3}z_{l} \otimes \mathsf{c}_{l}$, \eqref{eq:nahmtensor} gives $\lin(e) = \tfrac{1}{2}\sum_{l=1}^{3}A_{l} \otimes T_{l} = \tfrac{1}{2}M$, so $\tau\lin(e) = \tfrac{1}{2}\tau M$. The operators $T_{l}$ are the adjoint operators of $(\com^{3}, \times)$, so $[T_{1}, T_{2}] = T_{3}$, $[T_{2}, T_{3}] = T_{1}$, and $[T_{3}, T_{1}] = T_{2}$. Each $A_{l} \otimes 1 + 1 \otimes T_{l}$ is skew-Hermitian, and these operators satisfy the relations of $\su(2)$, whose Casimir $\sum_{l=1}^{3}(A_{l} \otimes 1 + 1 \otimes T_{l})^{2}$ acts on the irreducible module $V(q)$ as $-\tfrac{1}{4}q(q+2)$. Expanding the square shows
\begin{align}\label{eq:casimirM}
M = \tfrac{1}{2}\Big(\sum_{l=1}^{3}(A_{l} \otimes 1 + 1 \otimes T_{l})^{2} - \sum_{l=1}^{3}A_{l}^{2}\otimes 1 - 1 \otimes \sum_{l=1}^{3}T_{l}^{2}\Big).
\end{align}
Since $\sum_{l=1}^{3}A_{l}^{2} = -\tfrac{1}{4}d(d+2)$ on $V(d)$ and $\sum_{l=1}^{3}T_{l}^{2} = -2$, by \eqref{eq:casimirM} the operator $M$ acts on the summand $V(q)$ of $V(d) \otimes \com^{3}$ as $\tfrac{1}{2}\big(-\tfrac{1}{4}q(q+2) + \tfrac{1}{4}d(d+2) + 2\big)$, which is $-\tfrac{d}{2}$, $1$, $\tfrac{d}{2}+1$ for $q = d+2, d, d-2$. $M$ is Hermitian because the $A_{l}$ and $T_{l}$ are skew, and commutes with $\tau$ because $\tau$ commutes with the $A_{l}$ and the $T_{l}$ are real. An eigenspace $E_{m}$ of $M$ is therefore $\tau$-stable, $E_{m} = E_{m}^{\tau} \oplus iE_{m}^{\tau}$ with $E_{m}^{\tau}$ the fixed points of $\tau$, and $\tfrac{1}{2}\tau M$ acts as $\tfrac{m}{2}$ on $E_{m}^{\tau}$ and as $-\tfrac{m}{2}$ on $iE_{m}^{\tau}$. This proves \ref{morse1}.

For \ref{morse2}, $\g^{3} = (\subs \otimes \com^{3}) \oplus (\subs^{\perp}\otimes \com^{3})$ is an $H$-orthogonal decomposition into $M$-invariant subspaces, and $\Usp_{e} \subset \subs \otimes \com^{3}$, so $\Uperp_{e} = (\Uperp_{e} \cap (\subs \otimes \com^{3})) \oplus (\subs^{\perp}\otimes \com^{3})$. By \eqref{eq:hessR} and \ref{morse1}, each eigenspace $E_{m}$ of $M$ contained in $\Uperp_{e}$ contributes eigenvalues $16W(e)^{2}(\tfrac{|m|}{2} - \tfrac{1}{2})$ and $16W(e)^{2}(-\tfrac{|m|}{2} - \tfrac{1}{2})$, each with real multiplicity $\dim_{\com}E_{m}$. Hence $E_{m}$ contributes $\dim_{\com}E_{m}$ to the index, and in addition $\dim_{\com}E_{m}$ to the coindex if $|m| > 1$, to the nullity if $|m| = 1$, and to the index if $|m| < 1$. Identify $\subs \otimes \com^{3}$ with the $3\times 3$ complex matrices via $z_{a} \otimes \mathsf{c}_{b} \mapsto E_{ab}$. Then $A_{l}$ acts on $\subs$ as $T_{l}$, the decomposition $V(2) \otimes \com^{3} = V(4) \oplus V(2) \oplus V(0)$ is the decomposition into symmetric trace-free, antisymmetric, and scalar matrices, with $M$-eigenvalues $-1$, $1$, $2$, and $\Usp_{e}$ is the space of complex diagonal matrices. Hence $\Uperp_{e} \cap (\subs \otimes \com^{3})$, the matrices with vanishing diagonal, lies in the sum of the eigenspaces with $|m| = 1$ and has complex dimension $6$, contributing $6$ to the nullity and $6$ to the index. The summands of \eqref{eq:sperpdecomp} contribute as follows: $V(0)\otimes\com^{3}$, with $m = 0$, contributes index $6$; $V(1)\otimes\com^{3} = V(3) \oplus V(1)$, with $m = -\tfrac{1}{2}, 1$, contributes nullity $2$ and index $8 + 2 = 10$; $V(2)\otimes\com^{3}$, with $m = -1, 1, 2$, contributes coindex $1$, nullity $5 + 3 = 8$, and index $9$; and for $d \geq 3$, $V(d)\otimes\com^{3}$, with $m = -\tfrac{d}{2}, 1, \tfrac{d}{2}+1$, contributes coindex $(d+3) + (d-1) = 2d+2$, nullity $d+1$, and index $3(d+1)$. Summing gives the coindex and nullity in \eqref{eq:morseindex}, and the index $6 + 6n_{0} + 10n_{1} + 9n_{2} + \sum_{d \geq 3}3(d+1)n_{d}$, which equals $D + 3n_{0} + 4n_{1}$ because $\dim_{\com}\g = 3 + n_{0} + 2n_{1} + 3n_{2} + \sum_{d \geq 3}(d+1)n_{d}$.

For \ref{morse3}, let $(e_{+}, e_{-}, h)$ be an $\sll(2, \com)$-triple spanning $\subs$. The Hessian is negative semidefinite if and only if the coindex vanishes, which by \eqref{eq:morseindex} happens if and only if $n_{d} = 0$ for $d \geq 2$, that is, if and only if the eigenvalues of $\ad(h)$ on $\subs^{\perp}$ lie in $\{0, \pm 1\}$. If $j = 1$, then $h$ is conjugate to $\theta^{\vee}$ by Lemma \ref{lem:Bh}, and by \ref{R1} and the equality case of \eqref{eq:CSh} the only roots taking the value $\pm 2$ on $\theta^{\vee}$ are $\pm\theta$ and no root takes a larger value, so the condition holds. Conversely, suppose the condition holds, and, as in the proof of Lemma \ref{lem:Bh}, conjugate $h$ to a dominant element of the real span of the coroots in $\h$. Then $2$ is the largest eigenvalue of $\ad(h)$ and its eigenspace is $\com e_{+}$, and similarly for $-2$ and $\com e_{-}$. Since $\theta(h) \geq 2$ by the proof of Lemma \ref{lem:Bh}, $\theta(h) = 2$ and $\g_{\theta} = \com e_{+}$; likewise $\g_{-\theta} = \com e_{-}$, so $h = [e_{+}, e_{-}] \in \com t_{\theta}$ by \ref{R2}, and $\theta(h) = 2$ forces $h = \theta^{\vee}$, so $j = 1$.

For \ref{morse4}, by \eqref{eq:morseindex}, $\operatorname{ind} = D$ if and only if $n_{0} = n_{1} = 0$. The centralizer $\mathfrak{z}_{\g}(\subs)$ is the isotypic component $n_{0}V(0)$, so $n_{0} = \dim_{\com}\mathfrak{z}_{\g}(\subs)$. A subalgebra isomorphic to $\sll(2, \com)$ is distinguished if and only if its centralizer contains no nonzero semisimple element, which for the reductive Lie algebra $\mathfrak{z}_{\g}(\subs)$ means $\mathfrak{z}_{\g}(\subs) = 0$; and distinguished subalgebras are even, so $n_{0} = 0$ forces $n_{1} = 0$ \cite[Section~8.2]{Collingwood-McGovern}. The principal subalgebra is distinguished.
\end{proof}

\begin{remark}\label{rem:wolf}
For $j = 1$, the proof of Theorem \ref{thm:morse}\ref{morse3} shows that $\subs^{\perp}$ is the sum of the centralizer of $\subs$, on which $\subs$ acts trivially, and the isotypic component $n_{1}V(1)$. This is the decomposition of the quaternionic symmetric pair attached to the highest root by Wolf \cite{Wolf-quaternionic}, in which the real form of $n_{1}V(1)$ is the isotropy representation and $\subs \cap \gz \simeq \mathfrak{sp}(1)$ acts on it by the imaginary quaternions, and it is the fact used by Ohnita and Tasaki \cite[(8)]{Ohnita-Tasaki} in their uniqueness theorem for the long-root subgroup.
\end{remark}

By Theorem \ref{thm:morse}\ref{morse3}, the critical points of Lemma \ref{thm:sl2crit} with $j \geq 2$ are saddle points of $W$. For $j = 1$ the Hessian is negative semidefinite but degenerate; that these points are global maxima is the content of Theorem \ref{thm:A}, and Corollary \ref{cor:morsebott} below shows that the degeneracy is accounted for exactly by the manifold of maximum points.

The computation of Theorem \ref{thm:morse} applies to every $\tau$-twisted idempotent with components in $\subs$, not only to the one of Lemma \ref{thm:sl2crit}\ref{sl2crit3}, and the reason is visible in \eqref{eq:nahmtensor}. Let $GL(3, \com)$ act on $\g^{3} = \g \otimes \com^{3}$ through the second factor, $(Rb)_{i} = \sum_{m=1}^{3}R_{im}b_{m}$ for $b \in \g^{3}$ and $i \in \{1, 2, 3\}$, and let $\bar{R}$ be the entrywise complex conjugate of $R$. Since $\big((Ru) \times (Rv)\big) \cdot Rw = \det(Ru, Rv, Rw) = \det(R)\,(u \times v)\cdot w$ for $u, v, w \in \com^{3}$,
\begin{align}\label{eq:crossequivariance}
(Ru) \times (Rv) = \det(R)\,(R^{t})^{-1}(u \times v), \qquad R \in GL(3, \com),
\end{align}
which for $R \in SO(3, \com)$ is the automorphism property, and for unitary $R$, where $(R^{t})^{-1} = \bar{R}$, reads $(Ru) \times (Rv) = \det(R)\bar{R}(u \times v)$. Equivalently, the conjugate-bilinear product $(u, v) \mapsto \overline{u \times v}$, whose value is the vector $p \in \com^{3}$ satisfying $w \cdot \bar{p} = \det(u, v, w)$ for all $w \in \com^{3}$, is invariant under the group $SU(3)$ of transformations preserving both the Hermitian form and the volume form of $\com^{3}$. By \eqref{eq:nahmtensor}, $SU(3)$ therefore acts by automorphisms of the conjugate-bilinear product $(a, b) \mapsto \tau(a \tmlt b)$ on $\g^{3}$, whose nonzero idempotents are the $\tau$-twisted idempotents and whose multiplication operators are the operators $\tau\lin(a)$.
\begin{lemma}\label{lem:rotation}
Let $\subs$, $z_{1}, z_{2}, z_{3}$, and $e$ be as in Theorem \ref{thm:morse}.
\begin{enumerate}[label=(\arabic*), ref=(\arabic*)]
\item\label{rot1} Each $R \in U(3)$ acts on $\g^{3}$ as an $H$-unitary operator commuting with $J$ and preserving $\subs\otimes\com^{3}$ and $\subs^{\perp}\otimes\com^{3}$, $\tau R = \bar{R}\tau$, and
\begin{align}\label{eq:rotation}
(Ra) \tmlt (Rb) = \det(R)\,\bar{R}(a \tmlt b), \qquad a, b \in \g^{3}.
\end{align}
If $R \in SU(3)$, then $Re$ is a $\tau$-twisted idempotent and $\tau\lin(Re)R = R\,\tau\lin(e)$.
\item\label{rot2} Every nonzero $\tau$-twisted idempotent whose components lie in $\subs$ equals $Re$ for a unique $R \in SU(3)$.
\item\label{rot3} The conclusions \ref{morse2}--\ref{morse4} of Theorem \ref{thm:morse} hold at the image in $\mathcal{P}$ of every nonzero $\tau$-twisted idempotent whose components lie in $\subs$.
\end{enumerate}
\end{lemma}
\begin{proof}
\ref{rot1} The first claims hold because $R$ acts on the second factor of $\g \otimes \com^{3}$ unitarily with respect to $u \cdot \bar{v}$, while $\tau$ is $\tau$ on $\g$ tensored with complex conjugation of $\com^{3}$. By \eqref{eq:nahmtensor} and \eqref{eq:crossequivariance}, $(Ra) \tmlt (Rb) = \det(R)(R^{t})^{-1}(a \tmlt b)$ for pure tensors $a$ and $b$, hence for all $a, b \in \g^{3}$, which is \eqref{eq:rotation} because $(R^{t})^{-1} = \bar{R}$ for unitary $R$. If $\det R = 1$, then, by \eqref{eq:rotation}, $(Re) \tmlt (Re) = \bar{R}(e \tmlt e) = \bar{R}\tau e = \tau(Re)$ and $\tau\lin(Re)Rb = \tau\bar{R}(e \tmlt b) = R\tau(e \tmlt b) = R\,\tau\lin(e)b$.

\ref{rot2} Let $e' \in\idem_{\tau}(\g^{3}, \tmlt)$ be a nonzero $\tau$-twisted idempotent with components in $\subs$, and put $\beta := j\hv$. By Lemma \ref{thm:idem}\ref{idem2}, its components are pairwise $H$-orthogonal of a common norm and $F(e'_{1}, e'_{2}) = W(e') = 3/(4H(e', e'))$, while $F(e'_{1}, e'_{2}) = 1/\beta$ by \eqref{eq:sl2norm}; so $H(e'_{i}, e'_{p}) = \tfrac{\beta}{4}\delta_{ip}$ for $i, p \in \{1, 2, 3\}$. Write $e'_{i} = \tfrac{1}{2}\sum_{m=1}^{3}R_{im}z_{m}$ for $i \in \{1, 2, 3\}$, that is, $e' = Re$. Since $H(z_{l}, z_{m}) = \beta\delta_{lm}$ by \eqref{eq:xirelations}, also $H(e'_{i}, e'_{p}) = \tfrac{\beta}{4}(RR^{\dagger})_{ip}$, so $R$ is unitary. By \eqref{eq:rotation}, $\tau e' = e' \tmlt e' = \det(R)\bar{R}\tau e$, while $\tau e' = \tau Re = \bar{R}\tau e$, so $\det R = 1$. Uniqueness holds because $z_{1}, z_{2}, z_{3}$ is a basis of $\subs$.

\ref{rot3} Let $e' = Re$ with $R \in SU(3)$, as in \ref{rot2}, so $W(e') = W(e)$. By \ref{rot1}, $\tau\lin(e') = R\,\tau\lin(e)R^{-1}$, so $\tau\lin(e')$ and $\tau\lin(e)$ have the same spectrum on $\subs\otimes\com^{3}$ and on $\subs^{\perp}\otimes\com^{3}$. Write $\Usp_{e'}$ and $\Uperp_{e'}$ for the spaces \eqref{eq:UR} of $e'$. By \eqref{idem4} in the proof of Lemma \ref{hess2}, $\tau\lin(e')$ acts on $\spn_{\rea}\{\hat{e}'_{i}: i \in \{1, 2, 3\}\}$ with eigenvalues $1, -\tfrac{1}{2}, -\tfrac{1}{2}$, and so, by the anticommutation with $J$, on the image of this space under $J$ with eigenvalues $-1, \tfrac{1}{2}, \tfrac{1}{2}$; the same holds for $e$. Since $\Usp_{e'} \subset \subs \otimes \com^{3}$ and $\Uperp_{e'} = (\Uperp_{e'} \cap (\subs\otimes\com^{3})) \oplus (\subs^{\perp}\otimes\com^{3})$, and likewise for $e$, the operators $\tau\lin(e')$ on $\Uperp_{e'}$ and $\tau\lin(e)$ on $\Uperp_{e}$ have the same spectrum. By \eqref{eq:hessR}, the coindex, nullity, and index of $[e']$ and $[e]$ coincide.
\end{proof}

Let $\mathcal{S}$ be the set of $\tau$-stable subalgebras of $\g$ isomorphic to $\sll(2, \com)$ of Dynkin index $1$, and let $\subs_{\theta} \in \mathcal{S}$ be as in Lemma \ref{lem:lower}\ref{lower2}. For $\subs \in \mathcal{S}$, the proof of Theorem \ref{thm:morse}\ref{morse3} shows $\g = \subs \oplus n_{0}V(0) \oplus n_{1}V(1)$, with multiplicities that by the following corollary do not depend on $\subs$.
\begin{corollary}\label{cor:morsebott}
Let $\mathfrak{M} \subset \mathcal{P}$ be the set of maximum points of $W$.
\begin{enumerate}[label=(\arabic*), ref=(\arabic*)]
\item\label{mb1} $G_{0}$ acts transitively on $\mathcal{S}$, so $\mathcal{S} = G_{0}/N$ with $N := \{g \in G_{0}: \Ad(g)\subs_{\theta} = \subs_{\theta}\}$, a compact manifold of real dimension $2n_{1}$.
\item\label{mb2} The map $[a] \mapsto \spn_{\com}\{a_{1}, a_{2}, a_{3}\}$ is a smooth fiber bundle $\mathfrak{M} \to \mathcal{S}$ whose fiber over $\subs$ is the manifold of pairs $([a_{1}], [a_{2}])$ of $H$-orthogonal lines in $\subs \simeq \com^{3}$, that is, the manifold of complete flags in $\com^{3}$, which is $SU(3)/T^{2}$ with $T^{2}$ the maximal torus. In particular $\mathfrak{M}$ is a compact connected submanifold of $\mathcal{P}$ of real dimension $6 + 2n_{1}$.
\item\label{mb3} $\mathfrak{M}$ is a nondegenerate critical manifold of $W$ in the sense of Bott \cite{Bott-nondegenerate}: at every point of $\mathfrak{M}$, the kernel of the Hessian of $W$ on $\mathcal{P}$ is the tangent space of $\mathfrak{M}$.
\end{enumerate}
\end{corollary}
\begin{proof}
\ref{mb1} Transitivity is Lemma \ref{lem:lower}\ref{lower3}. The Lie algebra of $N$ is the normalizer of $\subs_{\theta} \cap \gz$ in $\gz$, which is $(\subs_{\theta} \cap \gz) \oplus \mathfrak{z}_{\gz}(\subs_{\theta} \cap \gz)$ because every derivation of a semisimple Lie algebra is inner. Since $\mathfrak{z}_{\g}(\subs_{\theta}) = n_{0}V(0)$ is $\tau$-stable, $\dim_{\rea}\mathcal{S} = \dim_{\rea}\gz - 3 - n_{0} = \dim_{\com}\g - 3 - n_{0} = 2n_{1}$.

\ref{mb2} By Theorem \ref{thm:A} and Lemma \ref{fkwmaxlemma}\ref{fkw2}, $[a] \in \mathfrak{M}$ if and only if $F(a_{1}, a_{2}) = 1/\hv$ and $a_{3} \in \com\,\tau[a_{1}, a_{2}]$, that is, if and only if $a_{1}$ and $a_{2}$ are $H$-orthogonal elements of some $\subs \in \mathcal{S}$ and $a_{3} \in \com\,\tau[a_{1}, a_{2}]$. In that case $a_{3} \in \subs$, and $a_{1}, a_{2}, a_{3}$ are pairwise $H$-orthogonal by Lemma \ref{thm:idem}, so $\subs = \spn_{\com}\{a_{1}, a_{2}, a_{3}\}$ is determined by $[a]$. Let $\widetilde{\mathfrak{M}}$ be the set of triples $(\subs, \ell_{1}, \ell_{2})$ with $\subs \in \mathcal{S}$ and $\ell_{1}, \ell_{2}$ $H$-orthogonal lines in $\subs$; it is the bundle $G_{0} \times_{N} \mathcal{F}$ associated with $G_{0} \to \mathcal{S}$, where $\mathcal{F}$ is the manifold of pairs of $H$-orthogonal lines in $\subs_{\theta}$, of real dimension $6$, on which $N$ acts through $\Ad$. The map $\Psi:\widetilde{\mathfrak{M}} \to \mathcal{P}$, $(\subs, \ell_{1}, \ell_{2}) \mapsto (\ell_{1}, \ell_{2}, \tau[\ell_{1}, \ell_{2}])$, is therefore a smooth $G_{0}$-equivariant bijection onto $\mathfrak{M}$. It restricts on each fiber to an embedding into $\proj(\subs)^{3}$, and its composition with the map $[a] \mapsto \spn_{\com}\{a_{1}, a_{2}, a_{3}\}$ to the Grassmannian of $3$-planes in $\g$, which is smooth on the triples of linearly independent lines, is the bundle projection onto the orbit $\mathcal{S}$; hence the differential of $\Psi$ is injective. Being an injective immersion of a compact connected manifold, $\Psi$ is an embedding.

\ref{mb3} Since $\mathfrak{M}$ consists of critical points, its tangent space is contained in the kernel of the Hessian. Let $[a] \in \mathfrak{M}$. By Lemma \ref{thm:idem}\ref{idem1}, $[a] = [e']$ for a $\tau$-twisted idempotent $e' \in \idem_{\tau}(\g^{3}, \tmlt)$, whose components lie in some $\subs \in \mathcal{S}$ by \ref{mb2}. By Lemma \ref{lem:rotation}\ref{rot3} and Theorem \ref{thm:morse}\ref{morse2} with $n_{d} = 0$ for $d \geq 2$, the nullity at $[a]$ is $6 + 2n_{1} = \dim_{\rea}\mathfrak{M}$.
\end{proof}

\begin{remark}\label{rem:minimalvalue}
Among the critical points of Lemma \ref{thm:sl2crit}, the smallest critical value is $1/(j\hv)$ for the largest Dynkin index $j$ of a subalgebra isomorphic to $\sll(2, \com)$, which is attained exactly by the principal subalgebra, for which $j = 2(\rho^{\vee}, \rho^{\vee})$, where $\rho^{\vee}$ is half the sum of the positive coroots. Indeed, the dominant characteristic $h$ of any such subalgebra satisfies $\al_{i}(h) \in \{0, 1, 2\}$ for the simple roots $\al_{i}$, $i \in \{1, \dots, \operatorname{rk}\g\}$ \cite[Section~3.5]{Collingwood-McGovern}, so $h = \sum_{i=1}^{\operatorname{rk}\g}c_{i}\varpi_{i}^{\vee}$ with $0 \leq c_{i} \leq 2$ in terms of the fundamental coweights $\varpi_{i}^{\vee}$, the basis dual to the simple roots, whose pairwise inner products are positive because the inverse Cartan matrix has positive entries; since $2\rho^{\vee} = \sum_{i=1}^{\operatorname{rk}\g}2\varpi_{i}^{\vee}$, $(h, h) \leq (2\rho^{\vee}, 2\rho^{\vee})$ with equality if and only if $c_{i} = 2$ for all $i$ (see also \cite{Panyushev-dynkinindex}). By Theorem \ref{thm:morse}\ref{morse4}, the index at this point is $D$, the smallest value permitted by Lemma \ref{lem:indexbound}. The critical points of Lemma \ref{thm:roottriangle} show that not all critical points are of the type of Lemma \ref{thm:sl2crit}. It is not known whether $1/(2(\rho^{\vee}, \rho^{\vee})\hv)$ is the infimum of the nonzero critical values of $W$, nor whether the index at a critical point realizing this infimum must equal $D$. On $\grass(2, \gz)$ the nonzero critical values of $\sect$ are exactly the numbers $1/(j\hv)$ (Remark \ref{rem:meshcheryakov}), so there the infimum is $1/(2(\rho^{\vee}, \rho^{\vee})\hv)$, attained exactly on the planes in principal subalgebras.
\end{remark}

\begin{example}\label{ex:sl3}
Let $\g = \sll(3, \com)$, viewed as trace-free $3\times 3$ complex matrices, and $\g_{0} = \su(3)$, so $\hv = 3$, $D = 21$, and $H(X, Y) = 6\tr XY^{\dagger}$. There are two conjugacy classes of subalgebras isomorphic to $\sll(2, \com)$: the long-root subalgebras, with $j = 1$ and $\sll(3, \com) = V(2) \oplus 2V(1) \oplus V(0)$, and the principal subalgebra $\so(3, \com)$, with $j = 4$ and $\sll(3, \com) = V(2) \oplus V(4)$. By Lemma \ref{thm:sl2crit} and Theorem \ref{thm:morse} they yield critical values $\tfrac{1}{3}$ and $\tfrac{1}{12}$ of $F$, $\sect$, and $W$, and at the corresponding critical points of $W$ on $\mathcal{P}$ the triple (coindex, nullity, index) equals $(0, 10, 32)$ and $(10, 11, 21)$. In the first case the nullity is the dimension $6 + 2n_{1} = 10$ of the manifold of maximum points (Corollary \ref{cor:morsebott}); in the second the index is $D$, as Theorem \ref{thm:morse}\ref{morse4} predicts for the principal subalgebra. The root vectors of the long roots $\ep_{1} - \ep_{2}$, $\ep_{2} - \ep_{3}$, $\ep_{3} - \ep_{1}$ give critical points of the type of Lemma \ref{thm:roottriangle}: for $x = E_{12}$ and $y = E_{23}$, $H(x, y) = 0$, $[x, y] = E_{13}$, and
\begin{align}\label{eq:sl3critical}
\qad(y)x = [-E_{32}, E_{13}] = E_{12} = \tfrac{1}{6}\norm{y}^{2}x, \qquad \qad(x)y = [E_{21}, E_{13}] = E_{23} = \tfrac{1}{6}\norm{x}^{2}y,
\end{align}
so by \eqref{eq:sl3critical} and Lemma \ref{lem:critFK}, $(x, y) \in \crtnt(F)$ with $F(x, y) = \tfrac{1}{6}$. By Lemma \ref{lem:critFK}\ref{critF}, $(x, y, -E_{31}) \in \crtnt(W)$, with critical value $\tfrac{1}{6} = 1/(2\hv)$, as in \eqref{eq:trianglevalue}, and the subalgebra generated by $E_{12}$, $E_{23}$, $E_{31}$ is all of $\sll(3, \com)$, in accord with Lemma \ref{thm:idem}\ref{idem3}.
\end{example}

\begin{example}\label{ex:g2}
Let $\g$ be of type $G_{2}$, so $\hv = 4$ and $D = 39$. Its four conjugacy classes of subalgebras isomorphic to $\sll(2, \com)$ are the long-root and short-root subalgebras, the principal subalgebra of a long-root $\sll(3, \com)$ (subregular), and the principal subalgebra \cite{Collingwood-McGovern, Dynkin-semisimplesubalgebras}. Their Dynkin indices, the decompositions \eqref{eq:sperpdecomp}, the critical values $1/(4j)$, and the data \eqref{eq:morseindex} are:
\begin{align}\label{eq:g2table}
\begin{aligned}
&\text{long root:} && j = 1, && \subs^{\perp} = 4V(1) \oplus 3V(0), && \tfrac{1}{4}, && (\operatorname{coind}, \operatorname{null}, \operatorname{ind}) = (0, 14, 64),\\
&\text{short root:} && j = 3, && \subs^{\perp} = 2V(3) \oplus 3V(0), && \tfrac{1}{12}, && (\operatorname{coind}, \operatorname{null}, \operatorname{ind}) = (16, 14, 48),\\
&\text{subregular:} && j = 4, && \subs^{\perp} = V(4) \oplus 2V(2), && \tfrac{1}{16}, && (\operatorname{coind}, \operatorname{null}, \operatorname{ind}) = (12, 27, 39),\\
&\text{principal:} && j = 28, && \subs^{\perp} = V(10), && \tfrac{1}{112}, && (\operatorname{coind}, \operatorname{null}, \operatorname{ind}) = (22, 17, 39).
\end{aligned}
\end{align}
The decompositions in \eqref{eq:g2table} follow from the values of the roots on the coroot of a long, respectively short, root, from $\g = \sll(3, \com) \oplus \com^{3} \oplus (\com^{3})^{\ast}$ for a long-root $\sll(3, \com)$, and from Kostant's theorem that the principal subalgebra decomposes $\g$ according to the exponents $1$, $5$ \cite{Kostant-triples}. The subregular and principal subalgebras are distinguished, and the index is $D$ for both, in accord with Theorem \ref{thm:morse}\ref{morse4}. For three short roots $\mu_{1}, \mu_{2}, \mu_{3}$ of $G_{2}$ with sum zero, the hypothesis $\mu_{a} - \mu_{b} \notin \Phi$ of Lemma \ref{thm:roottriangle}\ref{tri2} fails, and the corresponding root vectors $e_{1}, e_{2}, e_{3}$ generate all of $\g$. Indeed, the subalgebra they generate contains $[\g_{\mu_{b}}, \g_{\mu_{c}}] = \g_{-\mu_{a}}$ for $\{a, b, c\} = \{1, 2, 3\}$, hence all six short root spaces, and then $[\g_{\mu_{a}}, \g_{-\mu_{b}}] = \g_{\mu_{a} - \mu_{b}}$ for $a \neq b$ \cite[Proposition~8.4]{Humphreys}; the six roots $\mu_{a} - \mu_{b}$, $a \neq b$, are distinct and long, so they are all the long roots of $G_{2}$, and the brackets $[\g_{\al}, \g_{-\al}]$, $\al \in \Phi$, span $\h$.
\end{example}

\subsection*{Acknowledgment}
I thank M.~V. Meshcheryakov for bringing \cite{Meshcheryakov} to my attention after the first version of this paper was posted; references to it have been added throughout.

I used the AI system Claude (Anthropic), through the claude.ai interface, in September and October 2026. It was used to propose lemmas and to assess proof strategies, to check computations, to check citations against the cited sources, and to draft and revise parts of the text. 

This work has been supported partially by Project PID2024-156578NB-I00 funded by MICIU /AEI /10.13039/501100011033 / FEDER, EU.

\bibliographystyle{amsplain}
\def\polhk#1{\setbox0=\hbox{#1}{\ooalign{\hidewidth
  \lower1.5ex\hbox{`}\hidewidth\crcr\unhbox0}}} \def\cprime{$'$}
  \def\cprime{$'$} \def\cprime{$'$}
  \def\polhk#1{\setbox0=\hbox{#1}{\ooalign{\hidewidth
  \lower1.5ex\hbox{`}\hidewidth\crcr\unhbox0}}} \def\cprime{$'$}
  \def\cprime{$'$} \def\cprime{$'$} \def\cprime{$'$}
  \def\polhk#1{\setbox0=\hbox{#1}{\ooalign{\hidewidth
  \lower1.5ex\hbox{`}\hidewidth\crcr\unhbox0}}} \def\cprime{$'$}
  \def\Dbar{\leavevmode\lower.6ex\hbox to 0pt{\hskip-.23ex \accent"16\hss}D}
  \def\cprime{$'$} \def\cprime{$'$} \def\cprime{$'$} \def\cprime{$'$}
  \def\cprime{$'$} \def\cprime{$'$} \def\cprime{$'$} \def\cprime{$'$}
  \def\cprime{$'$} \def\cprime{$'$} \def\cprime{$'$} \def\dbar{\leavevmode\hbox
  to 0pt{\hskip.2ex \accent"16\hss}d} \def\cprime{$'$} \def\cprime{$'$}
  \def\cprime{$'$} \def\cprime{$'$} \def\cprime{$'$} \def\cprime{$'$}
  \def\cprime{$'$} \def\cprime{$'$} \def\cprime{$'$} \def\cprime{$'$}
  \def\cprime{$'$} \def\cprime{$'$} \def\cprime{$'$} \def\cprime{$'$}
  \def\cprime{$'$} \def\cprime{$'$} \def\cprime{$'$} \def\cprime{$'$}
  \def\cprime{$'$} \def\cprime{$'$} \def\cprime{$'$} \def\cprime{$'$}
  \def\cprime{$'$} \def\cprime{$'$} \def\cprime{$'$} \def\cprime{$'$}
  \def\cprime{$'$} \def\cprime{$'$} \def\cprime{$'$} \def\cprime{$'$}
  \def\cprime{$'$} \def\cprime{$'$}
\providecommand{\bysame}{\leavevmode\hbox to3em{\hrulefill}\thinspace}
\providecommand{\MR}{\relax\ifhmode\unskip\space\fi MR }
\providecommand{\MRhref}[2]{%
  \href{http://www.ams.org/mathscinet-getitem?mr=#1}{#2}
}
\providecommand{\href}[2]{#2}

\end{document}